\documentclass{amsart}
\usepackage{bm}
\usepackage{bbm}
\usepackage{wasysym}
\usepackage{cite}
\usepackage{mathrsfs}
\usepackage{cases}
\usepackage{latexsym}
\usepackage[all]{xy}
\usepackage{stmaryrd}
\usepackage{amsfonts,enumitem}
\usepackage{float}
\usepackage{amssymb,amscd,amsthm,dsfont,pb-diagram}
\usepackage{amsmath}
\usepackage{fancyhdr}
\usepackage{amsxtra,ifthen}
\usepackage{verbatim}
\usepackage{tikz-cd}
\usepackage[vcentermath]{youngtab}
\usepackage{xcolor}
\usepackage{hyperref}
\usepackage{young}
\usepackage{geometry}
\usetikzlibrary{arrows,decorations.markings}
\tikzstyle{dot}=[draw, fill =black, circle, inner sep=0pt, minimum size=2pt]

\theoremstyle{plain}

\numberwithin{equation}{subsection}
\newtheorem{theorem}{Theorem}[subsection]
\newtheorem{lemma}[theorem]{Lemma}
\newtheorem{proposition}[theorem]{Proposition}
\newtheorem{corollary}[theorem]{Corollary}

\newtheorem{definition}[theorem]{Definition}
\newtheorem{example}[theorem]{Example}

\newtheorem{remarks}[theorem]{Remarks}

\newcommand{\lam}{\lambda}
\newcommand{\ft}{\mathfrak t}
\newcommand{\fs}{\mathfrak s}

\newcommand{\bvarnothing}{{\boldsymbol \varnothing}}

\def\bi{{\mathbf i}}

\def\bs{{\mathbf s}}

\def\Std{{\rm Std}}

\def\N{{\mathbb N}}
\def\Z{{\mathbb Z}}
\def\t{{\mathfrak t}}

\def\blam{{\boldsymbol \lambda}}
\def\bmu{{\boldsymbol \mu}}
\def\bnu{{\boldsymbol \nu}}
\def\bLambda{{\boldsymbol \Lambda}}
\def\balpha{{\boldsymbol \alpha}}
\def\bbeta{{\boldsymbol \beta}}

\def\bv{{\bf v}}
\def\bs{{\bf s}}
\def\bu{{\bf u}}

\def\<{\langle}
\def\>{\rangle}

\DeclareMathOperator{\rad}{rad}

\DeclareMathOperator{\res}{res}
\DeclareMathOperator{\Top}{top}

\begin{document}

\title[Moving vectors I: Representation type]
{Moving vectors I: Representation type of blocks of Ariki-Koike algebras}

\thanks  {The work is supported by the Natural Science Foundation of Hebei Province,
China (A2021501002); China Scholarship Council (202008130184)  and
Natural Science Foundation of China (11871107).}

\author{Yanbo Li}

\address{Li: School of Mathematics and Statistics, Northeastern
University at Qinhuangdao, Qinhuangdao, 066004, P.R. China}

\email{liyanbo707@163.com}

\author{Xiangyu Qi}

\address{Qi: Department of mathematics, School of Science, Northeastern
University, Shenyang, 110819, P.R. China}

\email{qixiangyumath@163.com}

\begin{abstract}
We introduce a new invariant for blocks of Ariki-Koike algebras, 
called block moving vector, which is a vector of non-negative integers summing up to the weight of the block. 
In this paper, we use moving vectors to classify representation-finite blocks of Ariki-Koike algebras.
As applications, we obtain examples of blocks with the same weight associated with the same multicharge that are not derived equivalent and examples 
of derived equivalent blocks being in different orbits under the adjoint action of the affine Weyl group.
We also determine the representation type for blocks of cyclotomic $q$-Schur algebras.
\end{abstract}



\maketitle

%%%%%%%%%%%%%%%%%%%%%%%%%%%%%%%%%%%%%%%%%%%%%%%%%%%%%%%%%%%%%%%%%%%%%%%%%%%%%%%%%%%%%%%%
%%%%%%%%%%%%%%%%%%%%%%%%%%%%%%%%%%%%%%%%%%%%%%%%%%%%%%%%%%%%%%%%%%%%%%%%%%%%%%%%%%%%%%%%
%%%%%%%%%%%%%%%%%%%%%%%%%%%%%%%%%%%%%%%%%%%%%%%%%%%%%%%%%%%%%%%%%%%%%%%%%%%%%%%%%%%%%%%%
%%%%%%%%%%%%%%%%%%%%%%%%%%%%%%%%%%%%%%%%%%%%%%%%%%%%%%%%%%%%%%%%%%%%%%%%%%%%%%%%%%%%%%%%

%\tableofcontents

\section{Introduction}

The Ariki-Koike algebras (the cyclotomic Hecke algebras of type $G(r, 1, n)$) were introduced in \cite{AK}, \cite{BM} and \cite{C},
which include Iwahori-Hecke algebras of types A and B as special cases. They play an important role
in modular representation theory of finite groups of Lie type, and are in relation
with various significant objects such as quantum groups and rational Cherednik algebras.
Consequently, the Ariki-Koike algebras have received intensive and
continuous study since their appearance (see \cite{A1, A2, BK2, DJM, DM, F1, J, LM, LR} for examples).
The interest on these algebras has been strengthened
by a fundamental result of Brundan and Kleshchev \cite{BK}, in which an explicit isomorphism between blocks of Ariki-Koike algebras
and cyclotomic Khovanov-Lauda-Rouquier algebras (introduced in \cite{KL, Rou}) in type A were constructed, and many profound results emerged.
%For example, Hu and Shi proved the famous center conjectures in \cite{HS} recently.

\smallskip

In the modular representation theory of blocks of an Iwahori-Hecke algebra of type A, the weight is important because
it can measure how complicated a block is. For example, a block $B$ of an Iwahori-Hecke algebra of type A is of finite type 
if and only if the weight of $B$ is not more than one \cite{EN}. However, this is no longer true for Ariki-Koike algebras. 
In fact, many classical results of type A related with weights have no cyclotomic versions,
including representation type of blocks. During our study, we found that the structure of a block, to some extent,
is controlled by the process of an arbitrary abacus lying in the block moving to its core
through elementary operations. Accordingly, we define a new invariant for each block, called
{\bf block moving vector} (see Definition \ref{block moving vector} for details).
Roughly speaking,  just as its name implies,
a block moving vector is a vector of non-negative integers that records the process of an
abacus in the block moving to its core. The sum of components of
a block moving vector is just the weight of the corresponding block (see Lemma \ref{weight moving vector} for details).
By using moving vectors, we classify representation-finite blocks of Ariki-Koike algebras in this paper. 
Moreover, we consider in \cite{LQT} the core blocks (introduced by Fayers in \cite{F2})
and will apply moving vectors to relevant Lie theory in \cite{HLQ}.

\smallskip

In general, it is an essential and difficult problem in representation theory to
determine the representation type of a finite dimensional algebra.
Recall that a finite dimensional algebra $A$ has finite type if the
number of indecomposable $A$-modules up to isomorphism is finite. Otherwise, $A$ has infinite representation type.
If $A$ has infinite type, Drozd \cite{Dr} proved that $A$ is either tame or wild but not both.
Back to Hecke algebras, the representation type of blocks of an Iwahori-Hecke algebra of type A was determined by Erdmann and Nakano in \cite{EN}.
Ariki and Mathas \cite{AM} determined the representation type of the Iwahori-Hecke algebras of type B and
Ariki \cite{A3} determined the representation type for all of
the Iwahori-Hecke algebras of classical type.
Then Ariki \cite{A4} determined representation type for blocks of Iwahori-Hecke algebras of type B
by using a result obtained in \cite{AKMW} and the techniques developed in \cite{AIP, AP1}.
In 2010, Wada \cite{W} gave a necessary and sufficient condition on parameters for an Ariki-Koike algebra
to have finite type. However, it is still an open problem to determine when a block of an
Ariki-Koike algebra has finite type ({\em after we had posted our paper on arXiv, Ariki et al posted a paper \cite{ASW},
in which tame blocks are completely described}).

\smallskip

Let $K$ be an algebraically closed field, $1\neq q\in K^\times$. Define the {\em quantum characteristic} of $q$
to be the positive integer $e$ which is minimal such that
$1+q+\cdots+q^{e-1}=0$. If no such $e$ exists, we set $e=\infty$. Fix positive integers $n$ and $r$.
Let $(s_1, s_2, \cdots, s_r)\in \mathbb{Z}^r$ be a {\em multicharge} and
define $Q=(Q_1, Q_2, \cdots, Q_r)\in K^r$ with
$Q_i=q^{s_i}$. Then $\mathcal {H}_n(q, Q)$, the {\em Ariki-Koike algebra}
 with parameters $q$ and $Q$, is the unital
associative $K$-algebra with generators $T_0$, $T_1$, $\cdots$,
$T_{n-1}$ subject to the following relations:
\begin{enumerate}
\item[(H1)]\, $(T_0-Q_1)(T_0-Q_2)\cdots(T_0-Q_r)=0$;
\item[(H2)]\, $T_0T_1T_0T_1=T_1T_0T_1T_0$;
\item[(H3)]\, $(T_i+1)(T_i-q)=0\quad\quad\,\,\text{for}\,\, 1\le i \le n-1$;
\item[(H4)]\, $T_iT_{i+1}T_i=T_{i+1}T_iT_{i+1}\quad\,\text{for}\,\,1\le i \le n-2$;
\item[(H5)]\, $T_iT_j=T_jT_i\quad\quad\quad\quad\quad\,\,\,\,\,\text{for}\,\, 0\le i<j-1\le n-2$.
\end{enumerate}

The {\em cyclotomic $q$-Schur algebra} $\mathcal{S}_{n, r}(q, Q_1, Q_2, \cdots, Q_r)$
associated to $\mathcal {H}_n(q, Q)$ is the endomorphism algebra
$\mathcal{S}_{n, r}={\rm End}_{\mathcal {H}_n(q, Q)}(\bigoplus_{\bmu} M^{\bmu})$,
where $\bmu$ runs all over the $r$-partitions of $n$ and where
$M_{\bmu}$ is a certain $\mathcal {H}_n(q, Q)$-module (see \cite{DJM} and \cite{LM} for more details).

\smallskip

We now describe the main result of this paper. Note that the representation types of blocks of Iwahori-Hecke algebras 
of type A and type B have been determined in \cite{AM, EN}. So we assume $r>2$ here. 

\smallskip

\noindent{\bf Theorem.} {\em Let $K$ be an algebraically closed field with $char K\neq 2$
and $q\in K^\times$, $q\neq 1$. Let $(s_1, s_2, \cdots, s_r)\in {\mathcal A}^r_e$ (see Section 3.1), 
 $\mathcal {H}_n(q, Q)$ an Ariki-Koike algebra with $r>2$ and $B$
a block with block moving vector $\mathcal{M}=(m_1, m_2, \cdots, m_r)$ (see Definition \ref{block moving vector} for details).
Denote by $w(B)$ the weight of $B$.
Then $B$ has finite type if and only if $w(B) \leq 1$,  or $w(B) > 1$, $m_r=0$ and there exists $j$ 
such that $m_j = \dots = m_{j+w(B)}=1$ are all nonzero integers in $\mathcal{M}$ and
$s_j=\dots=s_{j+w(B)+1}$.
In the later case, block $\mathcal {H}_{\bbeta}^{\bLambda}$ is Morita equivalent to $K[x]/(x^{w(B)+1})$.}

\smallskip

As a by-product, we can determine the representation type of blocks of a cyclotomic $q$-Schur algebra.

\smallskip

\noindent{\bf Corollary.} {\em Let $\mathcal{B}$ be a non-simple block of $\mathcal{S}_{n, r}$ and $B$ the corresponding block in $\mathcal{H}_{n}(q, Q)$.
Then $\mathcal{B}$ has finite type if and only if $w(\mathcal{B})=1$, or $w(\mathcal{B})=2$ and $B$ has finite type.}

\smallskip

Moreover, we can also study the derived equivalence among blocks of $\mathcal{H}_n(q, Q)$.
It is known that two blocks of Iwahori-Hecke algebras of type A are derived equivalent if and only if
they have the same weight, and if and only if they are in the
same orbit under the adjoint action of the affine Weyl group \cite{CR}.
This is no longer true in cyclotomic case. We will construct examples of blocks
with the same weight associated with the same multicharge that are not derived equivalent.
Examples of derived equivalent blocks in different orbits are also given.

\smallskip

The paper is organized as follows. We begin our study with the so-called multiplication poset of
an indecomposable self-injective cellular algebra $A$ in Section 2.
The main result is that if $A$ has finite representation type,
then $A$ has only two different cell chains which implies that the multiplication poset must be a totally ordered set.
We emphasize that the result can be used to determine the representation type
of most blocks in an Ariki-Koike algebra and may be useful for studying blocks of other algebras.
Note that the cell modules of an Ariki-Koike algebra are indexed by $r$-partitions
and each $r$-partition can be represented by an abacus.
In Section 3, after reviewing some preliminaries on blocks and abaci,
we give a description of the action of an affine Weyl group on blocks by the language of abaci.
Then we define incomparable abaci and prove that given a block,
the existence of a pair of incomparable abaci implies the existence of a pair of
incomparable $r$-partitions with respect to the dominance order.
In Section 4, we define the (block) moving vector and conduct some preliminary study. 
In particular, we give a construction that will be used not only in this paper, but also in coming papers \cite{LQT, HLQ}. 
In Section 5, we give the proof of {\bf Theorem}, which is divided into three parts. In Section 6, we apply our results to study derived equivalence
and in Section 7, we determine the representation type
of blocks of a cyclotomic $q$-Schur algebra.

\medskip

%%%%%%%%%%%%%%%%%%%%%%%%%%%%%%%%%%%%%%%%%%%%%%%%%%%%%%%%%%%%%%%%%%%%%%%%%%%%%%%%%%%%%%%%
%%%%%%%%%%%%%%%%%%%%%%%%%%%%%%%%%%%%%%%%%%%%%%%%%%%%%%%%%%%%%%%%%%%%%%%%%%%%%%%%%%%%%%%%
%%%%%%%%%%%%%%%%%%%%%%%%%%%%%%%%%%%%%%%%%%%%%%%%%%%%%%%%%%%%%%%%%%%%%%%%%%%%%%%%%%%%%%%%
%%%%%%%%%%%%%%%%%%%%%%%%%%%%%%%%%%%%%%%%%%%%%%%%%%%%%%%%%%%%%%%%%%%%%%%%%%%%%%%%%%%%%%%%
%%%%%%%%%%%%%%%%%%%%%%%%%%%%%%%%%%%%%%%%%%%%%%%%%%%%%%%%%%%%%%%%%%%%%%%%%%%%%%%%%%%%%%%%

\section{Poset and representation type}

Let $A$ be an indecomposable self-injective cellular $K$-algebra with cell datum $(\Lambda, M, C, \ast)$,
where $K$ is a field with the characteristic different from two.
We prove in this section that if $A$ is of finite type, then $A$ has only two different cell chains, 
and consequently the poset $\Lambda$ must be a totally ordered set.
We begin with the definitions and some well-known results of cellular algebras and Brauer tree algebras.
The main references are \cite{A, ARS, GL}. The main result of this section is given in Section 2.3.

%%%%%%%%%%%%%%%%%%%%%%%%%%%%%%%%%%%%%%%%%%%%%%%%%%%%%%%%%%%%%%%%%%%%%%%%%%%%%%%%%%%%%%%%%%%%%%%%%%%%%%%%%%%%%%%%%%
%%%%%%%%%%%%%%%%%%%%%%%%%%%%%%%%%%%%%%%%%%%%%%%%%%%%%%%%%%%%%%%%%%%%%%%%%%%%%%%%%%%%%%%%%%%%%%%%%%%%%%%%%%%%%%%%%%

\subsection{Cellular algebras} Cellular algebras were introduced by Graham and Lehrer in \cite{GL} in 1996. Cellular theory provides a systematic framework
for studying the representation theory of many interesting and important algebras
coming from mathematics and physics, such as Iwahori-Hecke algebras of finite type \cite{G2}, Brauer algebras \cite{GL} and so on.
The main research object of this paper, Ariki-Koike algebras, are cellular too.

\begin{definition}[\protect{\cite[Definition~1.1]{GL}}]\label{2.1}
Let $R$ be a noetherian integral domain. An associative unital
$R$-algebra $A$ is called a cellular algebra with cell datum $(\check{\Lambda}, M, C, \ast)$ if the
following conditions are satisfied:

\begin{enumerate}
\item[{\rm (C1)}] The finite set $\check{\Lambda}$ is a poset with
order relation $\geq$. Associated with
each $\lam\in\check{\Lambda}$, there is a finite set $M(\lam)$. The
algebra $A$ has an $R$-basis $\{C_{S,T}^\lam \mid \lam\in\check{\Lambda},\,\, S,T\in
M(\lam)\}$.

\item[{\rm (C2)}] The map $\ast$ is an $R$-linear anti-automorphism of $A$
such that $(C_{S,T}^\lam)^{\ast}= C_{T,S}^\lam $ for
all $\lam\in\check{\Lambda}$ and $S,T\in M(\lam)$.

\item[{\rm (C3)}] Let $\lam\in\check{\Lambda}$ and $S,T\in M(\lam)$. For any
element
$a\in A$, we have
$$aC_{S,T}^\lam\equiv\sum_{S^{'}\in
M(\lam)}r_{a}(S',S)C_{S^{'},T}^{\lam} \,\,\,\,{\rm mod}\,\,\,
A(>\lam),$$ where $r_{a}(S^{'},S)\in R$ is independent of $T$ and
$A(>\lam)$ is the $R$-submodule of $A$ generated by
$\{C_{U,V}^\mu \mid U,\,\,V\in M(\mu),\,\,\mu>\lam\}$.
\end{enumerate}
\end{definition}

Before giving a rather lengthy list of definitions associated with cellular algebras, let us
give a remark about the poset $\check{\Lambda}$ in Definition \ref{2.1}. Firstly, we define another poset $\Lambda$,
which is equal to $\check{\Lambda}$ as a set, and a partial ordering is as follows.
For arbitrary two elements $\lam, \mu\in \Lambda$, we say $\lam\geq\mu$ if
there exist some $C_{U, V}^{\mu}$ and $a\in A$ such that certain $C_{S, T}^{\lam}$ appears
in the linear expansion of $aC_{U, V}^{\mu}$ with nonzero coefficient, and then extend the relation ``$\geq$'' by transitivity. 
The extension makes $\Lambda$ to be a poset.
Clearly, $\check{\Lambda}$ is a refinement of $\Lambda$. It is worthwhile to note that $\Lambda$ is more
essential than $\check{\Lambda}$ for our study and will be called the multiplication poset of $A$.
We will always write poset $\check{\Lambda}$ by $\Lambda$ throughout this section unless otherwise specified.

\smallskip

We now turn to give a number of basic definitions and results connected with cellular algebras.
As a natural consequence of the axioms, the cell module is defined as follows.

\begin{definition}[\protect{\cite[Definition 2.1]{GL}}]\label{2.2}
Let $A$ be a cellular algebra with cell datum $(\Lambda, M, C,
\ast)$. For each $\lam\in\Lambda$, the cell module $W_\lam$ is an $R$-module with basis
$\{C_{S}\mid S\in M(\lam)\}$ and the left $A$-action
defined by
$$aC_{S}=\sum_{S^{'}\in M(\lam)}r_{a}(S^{'},S)C_{S^{'}}
\,\,\,\,(a\in A,\,\,S\in M(\lam)),$$ where $r_{a}(S^{'},S)$ is the
element of $R$ defined in Definition~\ref{2.1} {\rm(C3)}.
\end{definition}

Let $A$ be a cellular algebra. For arbitrary elements $S,T,U,V\in M(\lam)$,
Definition~\ref{2.1} implies that
$$C_{S,T}^\lam C_{U,V}^\lam \equiv\Phi(T,U)C_{S,V}^\lam\,\,\,\, {\rm mod}\,\,\, A(>\lam),$$
where $\Phi(T,U)\in R$ depends only on $T$ and $U$. It is easy to check that $\Phi(T,U)=\Phi(U,T)$
for arbitrary $T,U\in M(\lam)$. By using these $\Phi(T, U)$,
one can define a bilinear form for cell module $W_\lam$ introduced in Definition \ref{2.2}: $$\Phi
_{\lam}:\,\,W_\lam\times W_\lam\longrightarrow R$$
$$\quad\quad\quad\quad\quad (C_{S},\, C_{T})\longmapsto\Phi(S,T).$$
Define
$\rad\lam:= \{x\in W_\lam\mid \Phi_{\lam}(x,y)=0
\,\,\,\text{for all} \,\,\,y\in W_\lam\}.$ If $\Phi
_{\lam}\neq 0$, then $\rad\lam$ is the radical of
$W_\lam$.

When $R$ is a field, Graham and Lehrer \cite{GL} proved the following result.
\begin{lemma} {\rm\cite[Theorem 3.4]{GL}}\label{2.3}
For any
$\lam\in\Lambda$, denote the $A$-module $W_\lam/\rad \lam$ by $L_{\lam}$. Let
$\Lambda_{0}=\{\lam\in\Lambda\mid \Phi_{\lam}\neq 0\}$. Then
$\{L_{\lam}\mid \lam\in\Lambda_{0}\}$ is a complete set of
pairwise non-isomorphic absolutely simple
$A$-modules.
\end{lemma}

For $\lam\in \Lambda$ and $\mu\in\Lambda_0$, let $d_{\lam\mu}$ be the multiplicity of $L_\mu$ in $W_\lam$.
Sometimes we write $d_{\lam\mu}$ as $[W_\lam: L_\mu]$.
Denote the matrix $(d_{\lam\mu})_{\lam\in\Lambda,\, \mu\in\Lambda_0}$ by $D$, which will be called the decomposition matrix of $A$.

\begin{lemma} {\rm\cite[Proposition 3.6]{GL}}\label{2.4}
Let $\lam\in \Lambda$ and $\mu\in\Lambda_0$. Then $d_{\mu\mu}=1$. Moreover, if $d_{\lam\mu}\neq 0$, then $\lam\geq\mu$.
\end{lemma}

An equivalent basis-free definition of a cellular algebra was given by Koenig and Xi \cite{KX2}, which
is useful in dealing with structural problems.

\begin{definition}\cite[Definition 3.2]{KX2}\label{2.5}
Let $A$ be an algebra over a noetherian integral domain $R$ with  an $R$-involution $\ast$.
A two-sided ideal $J$ in $A$ is called a cell ideal if and only if
the following data are given and the following conditions are satisfied:\vskip2mm
\begin{enumerate}
\item[{\rm (1)}] The ideal $J$ is fixed by $\ast$: $(J)^{\ast}=J.$
\smallskip
\item[{\rm (2)}]There exists a free $R$-module $\Delta\subset J$ of finite rank, such that
there is an isomorphism of $A$-bimodules $\alpha: J\simeq \Delta\otimes_R\Delta^\ast$ ($\Delta^\ast\subset J$ is the $\ast$-image of $\Delta$)
making the following diagram commutative:
\[\begin{CD}
J   @>\alpha>>\Delta\otimes_{R}\Delta^\ast\\
@V \ast VV                  @VVv_1\otimes v_2\mapsto v_2^\ast\otimes v_1^\ast V\\
J         @>\alpha>>   \Delta\otimes_{R}\Delta^\ast
\end{CD}\]
\end{enumerate}

The  algebra $A$ with $R$-involution $\ast$ is called cellular if and only
if there is an $R$-module decomposition $A=J_{\mu_1}'\oplus J_{\mu_2}'\oplus\cdots \oplus J_{\mu_m}'$ (for some $m$)
with $(J_{\mu_j}')^{\ast}=J_{\mu_j}'$ for each $j$ $(j=1,\dots,m)$ and such that setting
$J_{\mu_j}: =\bigoplus_{l=j}^{m}J_{\mu_l}'$ gives a chain of two-sided ideals of $A$:
$$0=J_{\mu_{m+1}}\subset J_{\mu_m}\subset J_{\mu_{m-1}}\subset J_{\mu_{m-2}}\subset\cdots\subset J_{\mu_1}=A,$$
each of them fixed by $\ast$, and each $J_{\mu_j}'=J_{\mu_j}/J_{\mu_{j+1}}$
is a cell ideal of $A/J_{\mu_{j+1}}$ (with respect to the involution induced by $\ast$ on the quotient).
 \end{definition}

For $\lam\in\Lambda_0$, let $P_\lam$ be the projective cover of $L_\lam$.
Denote by $P_\lam(\mu_{i})$ the module obtained by
applying the functor $-\otimes_A P_\lam$ to $J_{\mu_{i}}$
in Definition \ref{2.5} for $i=1, \dots, m+1$. We get a sequence of $A$-submodules
$$0=P_\lam(\mu_{m+1})\subset P_\lam(\mu_m)\subset P_\lam(\mu_{m-1})\subset P_\lam(\mu_{m-2})\subset\cdots\subset P_\lam(\mu_1)=P_\lam,$$
 which is in fact a cell filtration of $P_\lam$.

\begin{lemma}\cite[Theorem 3.7]{GL}\label{2.6}
For $\lam\in\Lambda_0$, denote the multiplicity of a cell module $W_\mu$ in $P_\lam$ by $[P_\lam: W_\mu]$.
Then $[P_\lam: W_\mu]=[W_\mu: L_\lam]$.
This implies that a simple module $L_{\lam}$ is a composition factor of a cell module $W_{\mu}$
if and only if $W_{\mu}$ is a factor of $P_{\lam}$.
\end{lemma}

Note that given a cellular algebra with cell datum $(\Lambda, M, C, \ast)$,
the set $\{\mu_1, \mu_2, \dots, \mu_m\}$ in Definition \ref{2.5} is in fact a linear extension of $\Lambda$.
However, we can find some more subtle structure of $P_\lam$ if we work with the original poset $\Lambda$.
The following lemma is a simple corollary of Lemma \ref{2.6}.

\begin{lemma}\label{2.7}
Let $\lam\in \Lambda_0$, $\mu, \nu\in\Lambda$ and $\mu, \, \nu$ incomparable (denoted by $\mu\parallel\nu$).
Then $W_\nu^{\oplus d_{\nu\lam}}\oplus W_\mu^{\oplus d_{\mu\lam}}$ is a subquotient of $P_\lam$.
\end{lemma}

\begin{proof}
Clearly, $\overline{A}:=A/(A(>\nu) +A(>\mu))$ is a cellular algebra. By abusing the notations,
we still denote the cell modules corresponding to $\nu$ and $\mu$ by $W_\nu$ and $W_\mu$, respectively.
Since $\mu\parallel\nu$, $J_{\nu}'\oplus J_{\mu}'$ is an ideal of $\overline{A}$, and
thus $W_\nu$ and $W_\mu$ are both submodules of a quotient of $P_\lam$. Now the lemma is clear by Lemma \ref{2.6}.
\end{proof}

%%%%%%%%%%%%%%%%%%%%%%%%%%%%%%%%%%%%%%%%%%%%%%%%%%%%%%%%%%%%%%%%%%%%%%%%%%%%%%%%%%%%%%%%%%%%%%%%%%%%%%%%%%%%%%%%%%%%%%%
%%%%%%%%%%%%%%%%%%%%%%%%%%%%%%%%%%%%%%%%%%%%%%%%%%%%%%%%%%%%%%%%%%%%%%%%%%%%%%%%%%%%%%%%%%%%%%%%%%%%%%%%%%%%%%%%%%%%%%%

\subsection{Brauer tree algebras} 
We fix some notations by recall the definition of a Brauer tree algebra. 
The main references of this subsection are \cite{A, ARS}. 

A Brauer tree is a finite tree together with two additional structures on each vertex $i$:
\begin{enumerate}
\item[(1)] a circular ordering of the edges adjacent to $i$;
\item[(2)] a positive integer $m(i)$, called the multiplicity satisfying at most one $m(i)$ is larger than one.
\end{enumerate}
The vertex with multiplicity more than one is called the exceptional vertex, which is drawn customarily as a blackened circle.
A finite dimensional algebra $A$ is called a {\em Brauer tree algebra} if the structure of the indecomposable projective modules
can be described in terms of a Brauer tree in the following way.
\begin{enumerate}
\item[(1)] There is a one to one correspondence between the edges $\alpha$ of the tree
and the indecomposable projective $A$-modules $P_{\alpha}$ and hence the corresponding simple $A$-modules $L_\alpha$.
\item[(2)] For an edge $\alpha$ connecting vertices $i$ and $j$, let $(\alpha=\alpha_1, \dots, \alpha_t)$ and $(\alpha=\beta_1, \dots, \beta_r)$
be the circular orderings of edges adjacent to $i$ and $j$, respectively. Then  $\rad P_\alpha=U_\alpha+V_\alpha$,
where $U_\alpha \cap V_\alpha=L_\alpha$, $U_\alpha$ is uniserial with composition factors $L_{\alpha_2}, \dots, L_{\alpha_t}, L_{\alpha_1},$
$m(i)$ times from top to bottom, and $V_\alpha$ is uniserial with composition factors $L_{\beta_2}, \dots, L_{\beta_r}, L_{\beta_1},$
$m(j)$ times from top to bottom.
\end{enumerate}

Given a Brauer tree $T$, let $\Gamma$ be the quiver whose vertices are in one to one correspondence with the edges of $T$.
For a vertex $i$ in $T$, the circular ordering $(\alpha_1, \dots, \alpha_t)$ give rise to an oriented cycle $C_i$.
Then the vertex $v_{\alpha}$ in $\Gamma$ corresponding to edge $\alpha$ in $T$ connecting vertices $i$ and $j$
belongs to exactly two oriented cycles $C_i$ and $C_j$. Denote the path in $C_i$
without repeated arrows starting and ending at $\alpha_k$ by $p_{\alpha_k}^{(i)}$.
Define relations $\rho$ to be $(p_{\alpha}^{(i)})^{m(i)}-(p_{\alpha}^{(j)})^{m(j)}, uv$,
where $u$ is the arrow in $C_i$ with $e(u)=\alpha$ and $v$ the arrow in $C_j$ with $s(v)=\alpha$, or
$u$ is the arrow in $C_j$ with $e(u)=\alpha$ and $v$ the arrow in $C_i$ with $s(v)=\alpha$.
Then the algebra $K(\Gamma, \rho)$ is a Brauer tree algebra given by $T$.
It is worth to note that a Brauer tree determines a unique Brauer tree algebra
up to Morita equivalence (see \cite[Corollary 4.3.3]{KZ}).
Furthermore, in \cite{G1} Geck studied the Brauer tree of Hecke algebras, 
and Koenig and Xi studied in \cite{KX2} the cellularity of a Brauer tree algebra.

\begin{lemma}\cite[Proposition 5.3]{KX2}\label{2.8}
A Brauer tree algebra is a cellular algebra if and only if the Brauer tree is a straight line.
\end{lemma}

%%%%%%%%%%%%%%%%%%%%%%%%%%%%%%%%%%%%%%%%%%%%%%%%%%%%%%%%%%%%%%%%%%%%%%%%%%%%%%%%%%%%%%%%%%%%%%%%%%%%%%%%%%%%%%%%%%%%%%%
%%%%%%%%%%%%%%%%%%%%%%%%%%%%%%%%%%%%%%%%%%%%%%%%%%%%%%%%%%%%%%%%%%%%%%%%%%%%%%%%%%%%%%%%%%%%%%%%%%%%%%%%%%%%%%%%%%%%%%%

\subsection{Poset of a finite type self-injective cellular algebra}

Let us illustrate all indecomposable projective modules of a cellular Brauer tree algebra as a lemma for later use. 
By Lemma \ref{2.8}, we only need to consider the trees being a straight line.
\begin{lemma}\label{2.9}
{\rm(1)} Let $A$ be a Brauer tree algebra for the following Brauer tree $T_1$.
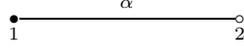
\begin{figure}[H]
	\[
	\xy
    (13,0)*{\scriptstyle\bullet}="1s";
    (28,2)*{\scriptstyle \alpha};
    (13,-2)*{\scriptstyle 1};
    (43,0)*{\scriptstyle\circ}="2s";
    (43,-2)*{\scriptstyle 2};
     "1s"; "2s" **\dir{-};
    \endxy
	\]
\end{figure}
The indecomposable projective module is uniserial with composition factor $L_{\alpha}$, $m(1)+1$ times from top to bottom.
Let $\Gamma_{\rm I}$ be the quiver
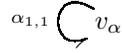
\begin{align*}
\xymatrix{
 & v_\alpha \ar@(ul,dl)_{\alpha_{1,1} }
}
\end{align*}
and let $\rho$ be $(\alpha_{1,1})^{m(1)+1}$. Then $K(\Gamma_{\rm I}, \rho)$ is a Brauer tree algebra given by $T_1$.

\smallskip

{\rm(2)} Let $A$ be a Brauer tree algebra for the following Brauer tree $T_2$ with $n>1$.
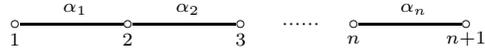
\begin{figure}[H]
	\[
	\xy
    (-2,0)*{\scriptstyle\circ}="2s";
    (-2,-2)*{\scriptstyle 1};
    (6,2)*{\scriptstyle \alpha_1};
    (13,0)*{\scriptstyle\circ}="3s";
    (13,-2)*{\scriptstyle 2};
    (28,0)*{\scriptstyle\circ}="4s";
    (28,-2)*{\scriptstyle 3};
    (21,2)*{\scriptstyle \alpha_2};
    (36,0)*{\scriptstyle\cdots\cdots};
    (43,0)*{\scriptstyle\circ}="5s";
    (43,-2)*{\scriptstyle n};
    (58,0)*{\scriptstyle\circ}="6s";
    (58,-2)*{\scriptstyle n+1};
    (51,2)*{\scriptstyle \alpha_{n}};
     "2s"; "3s" **\dir{-};
     "3s"; "4s" **\dir{-};
     "5s"; "6s" **\dir{-};
    \endxy
	\]
\end{figure}
The structures of indecomposable projective $A$-modules are illustrated as follows.
$$\begin{array}{cccc}
P_{\alpha_1}=\left(
    \begin{array}{c}
       L_{\alpha_1} \\
       L_{\alpha_2} \\
       L_{\alpha_1}
     \end{array}
     \right);
     & P_{\alpha_2}=\left(
     \begin{array}{ccc}
          & L_{\alpha_2} &  \\
         L_{\alpha_1} &  & L_{\alpha_3} \\
          & L_{\alpha_2} &
       \end{array}
       \right);
      & \cdots &
       P_{\alpha_{n}}=\left(
       \begin{array}{c}
           L_{\alpha_{n}} \\
           L_{\alpha_{n-1}} \\
           L_{\alpha_{n}}
         \end{array}
         \right).
\end{array}$$

\smallskip

{\rm(3)} Let $A$ be a Brauer tree algebra for the following Brauer tree $T_3$ with $n>1$.
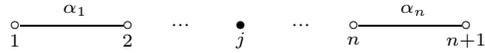
\begin{figure}[H]
	\[
	\xy
    (-2,0)*{\scriptstyle\circ}="2s";
    (-2,-2)*{\scriptstyle 1};
    (6,2)*{\scriptstyle \alpha_1};
    (13,0)*{\scriptstyle\circ}="3s";
    (13,-2)*{\scriptstyle 2};
    (20,0)*{\scriptstyle\cdots};
    (28,0)*{\scriptstyle\bullet};
    (28,-2)*{\scriptstyle j};
    (36,0)*{\scriptstyle\cdots};
    (43,0)*{\scriptstyle\circ}="5s";
    (43,-2)*{\scriptstyle n};
    (58,0)*{\scriptstyle\circ}="6s";
    (58,-2)*{\scriptstyle n+1};
    (51,2)*{\scriptstyle \alpha_{n}};
     "2s"; "3s" **\dir{-};
     "5s"; "6s" **\dir{-};
    \endxy
	\]
\end{figure}

The indecomposable projective $A$-modules are illustrated as follows.
$$\begin{array}{cc}
P_{\alpha_1}=\left(
    \begin{array}{c}
       L_{\alpha_1} \\
       L_{\alpha_2} \\
       L_{\alpha_1}
     \end{array}
     \right), \,j\neq 1, 2; & P_{\alpha_{1=j}}=\left(
    \begin{array}{ccc}
     & L_{\alpha_1} & \\
      L_{\alpha_1} & & L_{\alpha_2} \\
     \vdots &  & \\
      L_{\alpha_1} & &\\
      & L_{\alpha_1} &
     \end{array}
     \right);\quad
      \cdots;\\
P_{\alpha_{j-1}}=\left(
     \begin{array}{ccc}
          & L_{\alpha_{j-1}} &  \\
          &   & L_{\alpha_{j}}\\
          L_{\alpha_{j-2}} &  & L_{\alpha_{j-1}} \\
          &  &  \vdots\\
          &  &  L_{\alpha_{j}}\\
          & L_{\alpha_{j-1}} &
       \end{array}
       \right); &
       P_{\alpha_{j}}=\left(
     \begin{array}{ccc}
          & L_{\alpha_{j}} &  \\
          L_{\alpha_{j-1}} & &\\
          L_{\alpha_{j}} &  & L_{\alpha_{j+1}} \\
          \vdots & &\\
          L_{\alpha_{j-1}} & &\\
          & L_{\alpha_{j}} &
       \end{array}
       \right);\,\,\,\cdots;\\
       P_{\alpha_{n}}=\left(
       \begin{array}{c}
           L_{\alpha_{n}} \\
           L_{\alpha_{n-1}} \\
           L_{\alpha_{n}}
         \end{array}
         \right),\, j\neq n+1; & P_{\alpha_{n}}=\left(
    \begin{array}{ccc}
     & L_{\alpha_n} & \\
      L_{\alpha_{n-1}} & & L_{\alpha_n} \\
     & & \vdots \\
     & & L_{\alpha_n}\\
      & L_{\alpha_n} &
     \end{array}
     \right),\,j=n+1.
\end{array}$$
\end{lemma}

Let $A$ be a cellular Brauer tree algebra with cell datum $(\Lambda, M, C, \ast)$.
Then by Lemma \ref{2.3}, there is a one to one correspondence between the set of edges $\alpha_i$ and $\Lambda_0$.
For simplicity, we still denote the image of $\alpha_i$ in $\Lambda_0$ by $\alpha_i$.
Let us determine the form of a cell module of $A$.

\begin{lemma}\label{2.10}
Keep notations as Lemma \ref{2.9} and let $W_\lam$ be a cell module. We have
\begin{enumerate}
\item[{\rm (1)}] $d_{\lam\alpha_i}\leq 1$ for all $i$.
\item[{\rm (2)}] Neither $\left(
\begin{array}{ccc}
L_{\alpha_{i-1}} &  &  L_{\alpha_{i+1}} \\
&  L_{\alpha_{i}} &  \\
\end{array}
\right)$ nor $\left(
\begin{array}{ccc}
L_{\alpha_{1}} &  &  L_{\alpha_{2}} \\
&  L_{\alpha_{1}} &  \\
\end{array}
\right)$ is a submodule of a cell module $W$.
\item[{\rm (3)}] Neither $\left(
\begin{array}{ccc}
&  L_{\alpha_{i}} &  \\
L_{\alpha_{i-1}} &  &  L_{\alpha_{i+1}} \\
\end{array}
\right)$ nor  $\left(
\begin{array}{ccc}
&  L_{\alpha_{n}} &  \\
L_{\alpha_{i-1}} &  &  L_{\alpha_{n}} \\
\end{array}
\right)$ is a quotient module of a cell module $W$.
\end{enumerate}
\end{lemma}

\begin{proof}
(1) If $d_{\lam\alpha_i}> 1$ for some $\alpha_i$, then by Lemma \ref{2.6},
$W_\lam^{\oplus d_{\lam\alpha_i}}$ is a subquotient of $P_{\alpha_i}$.
This is impossible for a Brauer tree algebra whose Brauer tree is a straight line.
We refer the reader to Lemma \ref{2.9} for structures of $P_{\alpha_i}$.

(2) If there exists a cell module $W$ containing a submodule $\left(
                                                 \begin{array}{ccc}
                                                   L_{\alpha_{i-1}} &  &  L_{\alpha_{i+1}} \\
                                                    &  L_{\alpha_{i}} &  \\
                                                 \end{array}
                                               \right)$ ($i>1$),
then we conclude from Lemma \ref{2.6} that $L_{\alpha_{i+1}}$ is a composition factor of $P_{\alpha_{i-1}}$.
However, according to the definition of a Brauer tree algebra, the possible composition factors of
$P_{\alpha_{i-1}}$ are $L_{\alpha_{i-2}}$, $L_{\alpha_{i-1}}$ and $L_{\alpha_{i}}$. It is a contradiction.
Moreover, if $j\neq 1$, then clearly $\left(
\begin{array}{ccc}
L_{\alpha_{1}} &  &  L_{\alpha_{2}} \\
&  L_{\alpha_{1}} &  \\
\end{array}
\right)$ is not a submodule of any cell module $W$. If $j=1$ and $\left(
\begin{array}{ccc}
L_{\alpha_{1}} &  &  L_{\alpha_{2}} \\
&  L_{\alpha_{1}} &  \\
\end{array}
\right)$ is a submodule of some cell module $W$, then the multiplicity of $L_{\alpha_1}$ in $P_{\alpha_2}$ is at least 2. This is impossible.

(3) is proved similarly as (2).
\end{proof}

\begin{corollary}\label{2.11}
Each cell module is of the form $L_{\alpha_1}$, $L_{\alpha_n}$,
$\left(
\begin{array}{c}
L_{\alpha_{i-1}} \\
L_{\alpha_{i}} \\
\end{array}
\right)$ or
$\left(
\begin{array}{c}
L_{\alpha_{i+1}} \\
L_{\alpha_{i}} \\
\end{array}
 \right)$.
\end{corollary}

\begin{proof}
An immediate corollary of Lemma \ref{2.6} is that each cell module must be a cell factor of an indecomposable projective module.
Applying Lemma \ref{2.10} to the structures of indecomposable projective modules of a
cellular Brauer tree algebra $A$, we get that a cell module $W_\lam$ must be of the form $L_{\alpha_i}$,
$\left(
\begin{array}{c}
L_{\alpha_{i-1}} \\
L_{\alpha_{i}} \\
\end{array}
\right)$ or
$\left(
\begin{array}{c}
L_{\alpha_{i+1}} \\
L_{\alpha_{i}} \\
\end{array}
 \right).$
Hence we only need to prove each $L_{\alpha_i}$ with $1<i<n$ is not a cell module.
In fact, $L_{\alpha_i}$ with  $1<i<n$ is a composition factor of $P_{\alpha_{i-1}}$.
If $L_{\alpha_i}$ a cell module, then by Lemma \ref{2.6},
$[L_{\alpha_i}: L_{\alpha_{i-1}}]=[P_{\alpha_{i-1}}: L_{\alpha_i}]\geq 1$. It is a contradiction.
\end{proof}

\begin{definition}\label{2.12}
We call the cell modules $\left(
\begin{array}{c}
L_{\alpha_{i-1}} \\
L_{\alpha_{i}} \\
\end{array}
\right)$ are of type I and $\left(
\begin{array}{c}
L_{\alpha_{i+1}} \\
L_{\alpha_{i}} \\
\end{array}
\right)$ are of type II.
\end{definition}

Let us investigate the cell filtration of indecomposable projective modules $P_{\alpha_i}$. We begin with $P_{\alpha_1}$.

\begin{lemma}\label{2.13}
Let $A$ be a cellular Brauer tree algebra with cell datum $(\Lambda, M, C, \ast)$. We have
\begin{enumerate}
\item[{\rm (1)}]
If vertex 1 is not exceptional, then a cell filtration of $P_{\alpha_1}$ is either of the form
$L_{\alpha_1}\subset \left(
\begin{array}{c}
L_{\alpha_1} \\
L_{\alpha_2} \\
L_{\alpha_1} \\
\end{array}
\right)\subset \cdots \subset P_{\alpha_1}
$, or of the form
$\left(
\begin{array}{c}
L_{\alpha_2} \\
L_{\alpha_1} \\
\end{array}
\right)
\subset \left(
\begin{array}{c}
L_{\alpha_2} \\
L_{\alpha_1} \\
L_{\alpha_2} \\
L_{\alpha_1} \\
\end{array}
\right)\subset \cdots \subset P_{\alpha_1}.$
\item[{\rm (2)}] If vertex 1 is exceptional, then  cell filtration of $P_{\alpha_1}$ is either of the form
$L_{\alpha_1}\subset \left(
\begin{array}{c}
L_{\alpha_1} \\
L_{\alpha_1} \\
\end{array}
\right)\subset \cdots \subset P_{\alpha_1}
$, or of the form
$\left(
\begin{array}{c}
L_{\alpha_2} \\
L_{\alpha_1} \\
\end{array}
\right)
\subset \left(
\begin{array}{c}
L_{\alpha_1} \\
L_{\alpha_2} \\
L_{\alpha_1} \\
\end{array}
\right)\subset \cdots \subset P_{\alpha_1}.$
\end{enumerate}
\end{lemma}

\begin{proof}
It follows from Corollary \ref{2.11} that each $L_{\alpha_2}$ is either a submodule of a type I cell module
 or a quotient module of a type II cell module  defined in Definition \ref{2.12}.

(1)  If vertex 1 is not exceptional, then $P_{\alpha_1}$ is uniserial and $\rad P_{\alpha_1}$
has composition factors $(L_{\alpha_2}, L_{\alpha_1})$, $m(2)$ times from top to bottom.
By Corollary \ref{2.11}, each $L_{\alpha_2}$ has to combine with an $L_{\alpha_1}$ to form a cell module.
This forces the top or the socle of $P_{\alpha_1}$ to be a cell module.
If the socle of $P_{\alpha_1}$ is a cell module, then all the other $m(2)$ cell modules have to be of type I,
and if the top is a cell module, then all the others have to be of type II.

(2) When vertex 1 is exceptional, only one $L_{\alpha_2}$ appears in $P_{\alpha_1}$,
and this $L_{\alpha_2}$ combining with the top or the socle of $P_{\alpha_1}$ forms a cell module.
By Corollary \ref{2.11} again, each $L_{\alpha_1}$ left is a cell module.
\end{proof}

Interestingly, for a cellular Brauer tree algebra, all the cell factors of $P_{\alpha_i}$ will be determined once the type of cell factors of $P_{\alpha_1}$ is fixed.
We describe it as a lemma as follows.

\begin{lemma}\label{2.14}
Let $A$ be a cellular Brauer tree algebra with cell datum $(\Lambda, M, C, \ast)$.
Then all cell modules with two composition factors are of the same type.
\end{lemma}

\begin{proof}
Assume that the cell factors with two composition factors of $P_{\alpha_1}$ are of type I.
By Lemma \ref{2.6} all of these cell modules are cell factors of $P_{\alpha_2}$.
Then the structure of $P_{\alpha_2}$ forces $\Top P_{\alpha_2}$ to be a quotient module of a cell module
$\left(
\begin{array}{c}
L_{\alpha_2} \\
L_{\alpha_3} \\
\end{array}
\right)$, which is of type I. Hence all cell factors with two composition factors of $P_{\alpha_2}$ are of type I.
Continuing this analysis for $P_{\alpha_i}$ one by one, one can deduce
all cell modules with two composition factor are of type I.
A similar analysis is efficient too when the cell factors with two composition factors of $P_{\alpha_1}$ are of type II.
\end{proof}

Based on the above preparation, we can give the following characterization on the poset of a cellular Brauer tree algebra.

\begin{lemma}\label{2.15}
Let $A$ be a cellular Brauer tree algebra with the poset $\Lambda$. Then $\Lambda$ is a totally ordered set.
\end{lemma}

\begin{proof}
Denote the set of indexes of cell factors of $P_{\alpha_i}$ by $\Lambda_{\alpha_i}$. We claim that $\Lambda_{\alpha_i}$
is a totally ordered set.
In fact, let $W_{\alpha_i}$ be the cell module with $\Top W_{\alpha_i}=\Top P_{\alpha_i}$.
Then according to Lemmas \ref{2.11}, \ref{2.13} and \ref{2.14}, $\Top W_{\alpha_i}$
is a composition factor of each cell factor of $P_{\alpha_i}$.
By Lemma \ref{2.4} this implies that $\alpha_i$ is minimal in $\Lambda_{\alpha_i}$.
Moreover, assume that there exist $\mu, \nu\in \Lambda_{\alpha_i}$ with $\mu\, ||\,\nu$.
Then $W_{\mu}\oplus W_{\nu}$ is a subquotient of $P_{\alpha_i}$ due to Lemma \ref{2.7}.
It is in conflict with the structures of cell factors of $P_{\alpha_i}$ determined by Lemma \ref{2.13} and \ref{2.14}.

On the other hand, all cell modules with two composition factors are of the same type by Lemma \ref{2.14}.
If a type I cell module appears, Lemma \ref{2.4} implies that $\alpha_i\geq \alpha_{i+1}$ for $1\leq i< n$.
Combine it with the above claim makes $\Lambda$ to be totally ordered and the minimal element of $\Lambda$ is $\alpha_n$.
For the case of type II, we can deduce from Lemma \ref{2.4} that $\alpha_{i+1}\geq \alpha_i$,
and this also forces $\Lambda$ to be a totally ordered set, in which $\alpha_1$ is the minimal element.
\end{proof}

Now we are ready to give the main result of this section.

\begin{proposition}\label{2.16}
Let $A$ be an indecomposable self-injective cellular $K$-algebra with cell datum $(\Lambda, M, C, \ast)$ and let
$K$ be an algebraically closed field with $char K\neq 2$.
If $A$ has finite type, then $\Lambda$ must be a totally ordered set.
\end{proposition}

\begin{proof}
By \cite[Theorem 6.8]{AKMW} and \cite[Theorem 8.1]{KX5},
if $A$ is an indecomposable self-injective cellular $K$-algebra of finite type, then
$A$ is Morita equivalent to a cellular Brauer tree algebra.
%It is well-known that a Brauer tree determine up to Morita equivalence a symmetric algebra.
%This implies that an algebra Morita equivalent to a Brauer tree algebra is again a Brauer tree algebra.
The proposition then follows from Lemma \ref{2.15}.
\end{proof}

Recall that in Lemma \ref{2.9} we illustrate all indecomposable projective modules for a cellular Brauer tree algebra 
and then  determine their cell filtrations in Lemma \ref{2.13}. In Lemma \ref{2.15}, we on one hand 
determine the order of the index set of cell factors of $P_{\alpha_i}$ and on the hand determine the order for these index sets.
To sum up, one can determine $\Lambda$ completely
for an indecomposable self-injective cellular $K$-algebra of finite type.

\begin{example}\label{cellchain}
Let $A$ be a Brauer tree algebra for the following Brauer tree:

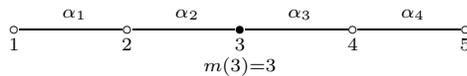
\begin{figure}[H]
	\[
	\xy
    (-2,0)*{\scriptstyle\circ}="2s";
    (-2,-2)*{\scriptstyle 1};
    (6,2)*{\scriptstyle \alpha_1};
    (13,0)*{\scriptstyle\circ}="3s";
    (13,-2)*{\scriptstyle 2};
    (28,0)*{\scriptstyle\bullet}="4s";
    (28,-2)*{\scriptstyle 3};
    (28,-5)*{\scriptstyle m(3)=3};
    (21,2)*{\scriptstyle \alpha_2};
    (43,0)*{\scriptstyle\circ}="5s";
    (43,-2)*{\scriptstyle 4};
    (36,2)*{\scriptstyle \alpha_{3}};
    (58,0)*{\scriptstyle\circ}="6s";
    (58,-2)*{\scriptstyle 5};
    (51,2)*{\scriptstyle \alpha_{4}};
     "2s"; "3s" **\dir{-};
     "3s"; "4s" **\dir{-};
     "4s"; "5s" **\dir{-};
     "5s"; "6s" **\dir{-};
    \endxy
	\]
\end{figure}

Fix a cellular structure of $A$. 
Then the cell modules in sequence have to be one of the following:

$$L_{\alpha_1}<\left(\begin{array}{c}
L_{\alpha_2} \\
L_{\alpha_1} \\
\end{array}\right)<\left(\begin{array}{c}
L_{\alpha_3} \\
L_{\alpha_2} \\
\end{array}\right)<\left(\begin{array}{c}
L_{\alpha_3} \\
L_{\alpha_2} \\
\end{array}\right)<\left(\begin{array}{c}
L_{\alpha_3} \\
L_{\alpha_2} \\
\end{array}\right)<\left(\begin{array}{c}
L_{\alpha_4} \\
L_{\alpha_3} \\
\end{array}\right)<L_{\alpha_4}$$

$$L_{\alpha_4}<\left(\begin{array}{c}
L_{\alpha_3} \\
L_{\alpha_4} \\
\end{array}\right)<\left(\begin{array}{c}
L_{\alpha_2} \\
L_{\alpha_3} \\
\end{array}\right)<\left(\begin{array}{c}
L_{\alpha_2} \\
L_{\alpha_3} \\
\end{array}\right)<\left(\begin{array}{c}
L_{\alpha_2} \\
L_{\alpha_3} \\
\end{array}\right)<\left(\begin{array}{c}
L_{\alpha_1} \\
L_{\alpha_2} \\
\end{array}\right)<L_{\alpha_1}$$

For the first one, the poset $\Lambda=\{\alpha_1<\alpha_2<\alpha_3^1<\alpha_3^2<\alpha_3^3<\alpha_4^1<\alpha_4^2\}$ 
and  $\Lambda_0=\{\alpha_1, \alpha_2, \alpha_3^1, \alpha_4\}$. The number of the elements in $\Lambda$ is $4+m(3)$. The second one is similar and we omit the details here.
\end{example}

In Example \ref{cellchain}, the cell modules corresponding to $\alpha_3^2$ and $\alpha_3^3$ are 
isomorphic and consequently have the same dimension. In fact, this is a general phenomenon.

\begin{corollary}\label{2.18}
Let $A$ be an indecomposable self-injective cellular $K$-algebra of finite type with cell datum $(\Lambda, M, C, \ast)$
and $\lam, \mu\in \Lambda$. Assume that  $\lam, \mu\notin \Lambda_0$ and 
neither $\lam$ nor $\mu$ is maximal.
Then $\dim W_\lam=\dim W_\mu$.
\end{corollary}

\begin{proof}
By \cite[Theorem 6.8]{AKMW}, suppose that $A$ is a cellular Brauer tree algebra with 
cell datum $(\Lambda, M, C, \ast)$. If the Brauer tree has no exceptional vertex, 
then there are not $\lam, \mu\in \Lambda$ satisfying the conditions. Now assume that the Brauer tree 
has an exceptional vertex, and $\lam, \mu$ satisfying the conditions. 
If $1$ is the exceptional vertex, then $W(\lam)$ and $W(\mu)$ are all isomorphic to  $L_{\alpha_1}$. 
If $n+1$ is the exceptional vertex, then $W(\lam)$ and $W(\mu)$ are all isomorphic to  $L_{\alpha_n}$.
If $i$ is the exceptional vertex with $1<i<n+1$, then $W(\lam)$ and $W(\mu)$ have the same 
composition factors $L_{\alpha_{i-1}}$ and $L_{\alpha_i}$ and consequently, $\dim W_\lam=\dim W_\mu$.
\end{proof}

\medskip

%%%%%%%%%%%%%%%%%%%%%%%%%%%%%%%%%%%%%%%%%%%%%%%%%%%%%%%%%%%%%%%%%%%%%%%%%%%%%%%%%%%%
%%%%%%%%%%%%%%%%%%%%%%%%%%%%%%%%%%%%%%%%%%%%%%%%%%%%%%%%%%%%%%%%%%%%%%%%%%%%%%%%%%%%
%%%%%%%%%%%%%%%%%%%%%%%%%%%%%%%%%%%%%%%%%%%%%%%%%%%%%%%%%%%%%%%%%%%%%%%%%%%%%%%%%%%%
%%%%%%%%%%%%%%%%%%%%%%%%%%%%%%%%%%%%%%%%%%%%%%%%%%%%%%%%%%%%%%%%%%%%%%%%%%%%%%%%%%%%
%%%%%%%%%%%%%%%%%%%%%%%%%%%%%%%%%%%%%%%%%%%%%%%%%%%%%%%%%%%%%%%%%%%%%%%%%%%%%%%%%%%%

\section{Abacus orbits and incomparable abaci}

It is known that blocks of an Ariki-Koike algebra are symmetric \cite{MM}, and certainly self-injective. In light of Section 2,
to determine the representation type of an indecomposable symmetric cellular algebra,
one can analyze the poset first. The aim of this section is to make
some preparations. We begin with some preliminaries about
blocks and abaci. Then we translate the action of an affine Weyl group on blocks into the language
of abaci. Finally, we define and study the incomparable abaci, which is a key notion in this paper.

%%%%%%%%%%%%%%%%%%%%%%%%%%%%%%%%%%%%%%%%%%%%%%%%%%%%%%%%%%%%%%%%%%%%%%%%%%%%%%%%%%%%%%%%%%%%%%%%%%%%%%%%%%%%
%%%%%%%%%%%%%%%%%%%%%%%%%%%%%%%%%%%%%%%%%%%%%%%%%%%%%%%%%%%%%%%%%%%%%%%%%%%%%%%%%%%%%%%%%%%%%%%%%%%%%%%%%%%%

\subsection{Preliminaries on blocks and abaci}
Keeps notations as in Section 1. We first connect $\mathcal {H}_n(q, Q)$ to Lie theory.
Define $e$ to be the multiplicative order of $q$.
Let $\Gamma_e$ be the oriented quiver with vertex set $I=\mathbb{Z}/e\mathbb{Z}$
and with directed edges $i\longrightarrow i+1$, for all $i\in I$.
Thus $\Gamma_e$ is the quiver of type $A_\infty$ if $e=\infty$, and if $e\geq 2$ is finite then
it is a cyclic quiver of type $A^{(1)}_{e-1}$:
$$
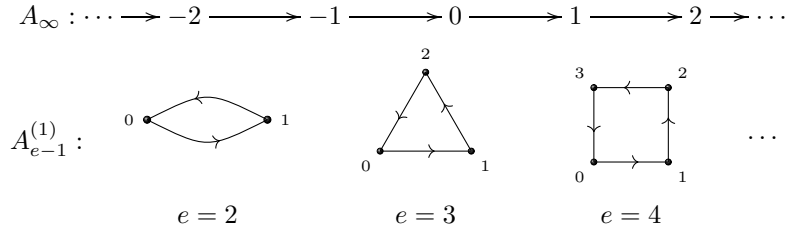
\xymatrix@C=0.5cm{
A_\infty:  \cdots \ar[r] & -2 \ar[rr] && -1 \ar[rr] && 0 \ar[rr] && 1 \ar[rr] && 2  \ar[r] & \cdots}$$

%\medskip

\begin{center}
$A_{e-1}^{(1)}:$
\begin{tabular}{*5c}
  \begin{tikzpicture}[scale=0.8,decoration={curveto, markings,
            mark=at position 0.6 with {\arrow{>}}
    }]
    \useasboundingbox (-1.7,-0.7) rectangle (1.7,0.7);
    \foreach \x in {0, 180} {
      \shade[ball color=black] (\x:1cm) circle(2pt);
    }
    \tikzstyle{every node}=[font=\tiny]
    \draw[postaction={decorate}] (0:1cm) .. controls (90:5mm) .. (180:1cm)
           node[left,xshift=-0.5mm]{$0$};
    \draw[postaction={decorate}] (180:1cm) .. controls (270:5mm) .. (0:1cm)
           node[right,xshift=0.5mm]{$1$};
  \end{tikzpicture}
& % e=3
  \begin{tikzpicture}[scale=0.7,decoration={ markings,% switch on markings
            mark=at position 0.6 with {\arrow{>}}
    }]
    \useasboundingbox (-1.7,-0.7) rectangle (1.7,1.4);
    \foreach \x in {90,210,330} {
      \shade[ball color=black] (\x:1cm) circle(2pt);
    }
    \tikzstyle{every node}=[font=\tiny]
    \draw[postaction={decorate}]( 90:1cm)--(210:1cm) node[below left]{$0$};
    \draw[postaction={decorate}](210:1cm)--(330:1cm) node[below right]{$1$};
    \draw[postaction={decorate}](330:1cm)--( 90:1cm)
           node[above,yshift=.5mm]{$2$};
  \end{tikzpicture}
& % e=4
  \begin{tikzpicture}[scale=0.7,decoration={ markings,% switch on markings
            mark=at position 0.6 with {\arrow{>}}
    }]
    \useasboundingbox (-1.7,-0.7) rectangle (1.7,0.7);
    \foreach \x in {45, 135, 225, 315} {
      \shade[ball color=black] (\x:1cm) circle(2pt);
    }
    \tikzstyle{every node}=[font=\tiny]
    \draw[postaction={decorate}] (135:1cm) -- (225:1cm) node[below left]{$0$};
    \draw[postaction={decorate}] (225:1cm) -- (315:1cm) node[below right]{$1$};
    \draw[postaction={decorate}] (315:1cm) -- (45:1cm) node[above right]{$2$};
    \draw[postaction={decorate}] (45:1cm) -- (135:1cm) node[above left]{$3$};
  \end{tikzpicture}
&\raisebox{3mm}{$\dots$}
\\[4mm]
  $e=2$&$e=3$&$e=4$&
\end{tabular}
\end{center}

Let $(a_{i,j})_{i,j\in I}$ be the symmetric Cartan matrix associated
with $\Gamma_e$, so that
$$a_{i,j}=\begin{cases}
    2 & \text{ if }i=j,\\
    0 & \text{ if }i\ne j\pm 1,\\
    -1 & \text{ if $e\ne2$ and }i=j\pm 1,\\
    -2 & \text{ if $e=2$ and }i=j+1.
\end{cases}$$
Following Kac~\cite[Chapt.~1]{Kac}, let $(\mathfrak h,\Pi,\check\Pi)$
be a realization of the Cartan matrix, and $\Pi=\{\balpha_i\mid i\in I\}\subset \mathfrak h^\ast$
the associated set of simple roots, $\{\bLambda_i\mid i\in I\}\subset \mathfrak h^\ast$ the
fundamental dominant weights, and $(\cdot,\cdot)$ the bilinear form
determined by $$(\balpha_i,\balpha_j)=a_{i,j}\quad \text{and} \quad
(\bLambda_i,\balpha_j)=\delta_{ij}, \quad\text{for} \quad i,j\in I.$$
Denote by $P=\bigoplus_{i\in I}\mathbb{Z}\bLambda_i$ and $P_+=\bigoplus_{i\in I}\mathbb{N}\bLambda_i$
the weight lattice and the {\em dominant weight lattice} of $(\mathfrak h,\Pi,\check\Pi)$,
$Q_+=\bigoplus_{i\in I}\mathbb{N}\balpha_i$ the {\em positive root lattice},
and $W$ the affine Weyl group, which is generated by $\{\sigma_i\}_{i\in I}$,
the fundamental reflections of the space $\mathfrak h^\ast$.
Note that $$\sigma_i\bLambda=\bLambda-(\bLambda, \,\balpha_i)\balpha_i \quad\text{for}\quad \bLambda\in P
\quad\text{and} \quad\sigma_i\balpha_j=\balpha_j-a_{ji}\balpha_i \quad\text{for}\quad j\in I.$$
The Kac-Moody algebra corresponding to this data is $\widehat{\mathfrak{sl}}_e$.

Given a {\em multicharge} $\bs=(s_1, s_2, \dots, s_r)$, where $s_i\in \mathbb{Z}$ for each $i$, one can associate it with a dominant weight
${\bf \Lambda}=\sum_{i\in I}k_i\bLambda_i$, where $$k_i=\begin{cases}
   \, \sharp\{1\leq j\leq r\mid s_j\equiv i \,({\rm mod} \,e)\} & \text{ if }e<\infty,\\
   \, \sharp\{1\leq j\leq r\mid s_j = i \} & \text{ if }e=\infty,
\end{cases}$$
and then $\mathcal {H}_n(q, Q)$ (recall $Q=(q^{s_1}, \dots, q^{s_r})$) will be denoted by $\mathcal {H}_n^{\bLambda}$ if necessary.
By \cite[Section 4.1]{BK}, there is a set $\{e(\bi) \mid \bi\in I^n\}$
(some could be zero) of pairwise orthogonal idempotents in $\mathcal {H}_n^{\bLambda}$.
Take ${\bbeta}\in Q_+$ with $\sum_{i\in I}(\bLambda_i,{\bbeta})=n$ and let
$$I^{\bbeta}=\{\bi=(i_1, i_2, \dots, i_n)\in I^n\mid\balpha_{i_1}+\dots+\balpha_{i_n}=\bbeta\}.$$
If $I^{\bbeta}\neq \varnothing$, by \cite[Theorem~2.11]{LM} $e_{\bbeta}=\sum_{\bi\in I^{\bbeta}}e(\bi)$
is a primitive central  idempotent of $\mathcal {H}_n^\Lambda$. Write by
$\mathcal{H}_{\bbeta}^{\bLambda}$ the block algebra $e_{\bbeta}\mathcal{H}_n^{\bLambda}.$
Then one can decompose $\mathcal{H}_n^{\bLambda}$ into a direct sum of blocks as
$\mathcal{H}_n^{\bLambda}=\bigoplus_{\bbeta\in Q_+,\,I^{\bbeta}\ne\emptyset}\mathcal{H}_{\bbeta}^{\bLambda}.$
In \cite[(3.4)]{BKW},  Brundan, Kleshchev and Wang defined the {\em defect} of $\bbeta\in Q_+$ to be
$(\bLambda,\bbeta)-\frac12(\bbeta,\bbeta).$
It coincides with Fayers definition~\cite{F1} of weight for the block algebras $\mathcal{H}_{\bbeta}^{\bLambda}$ (see \cite[Page 633]{HM})
and will be denoted by $w(\mathcal{H}_{\bbeta}^{\bLambda})$.
In addition, we emphasize that using the derived equivalence given by Chuang and Rouquier \cite{CR}, which lifts Weyl group action,
two blocks $\mathcal{H}_{\balpha}^{\bLambda}$ and $\mathcal{H}_{\bbeta}^{\bLambda}$ of an Ariki-Koike algebra are derived equivalent
if $\bLambda-\balpha$ and $\bLambda-\bbeta$ are in the same $W$-orbit.

\smallskip

We also need some combinatorics. Let $n$ be a positive integer. A {\em partition} $\lam$ of $n$ is a non-increasing sequence of
non-negative integers $\lam=(\lam_1,\dots,\lam_s)$ such that
$\sum_{i=1}^{s}\lam_i=n$ and we write $|\lam|=n$.
The {\em Young diagram} of a
partition $\lam$ is the set of nodes $[\lam]=\{(i,j)\mid 1\leq i,
1\leq j\leq\lam_{i}\}$.
The {\em conjugate} of $\lam$ is defined to be a partition
$\lam'=(\lam'_1,\lam'_2,\cdots)$, where $\lam'_j$ is equal to the
number of nodes in column $j$ of $[\lam]$ for $j=1,2,\cdots$.
A {\em rim $e$-hook} (or simply
an $e$-rim) of $[\lam]$ is a connected subset of the rim of $[\lam]$ with exactly $e$ nodes, which can be
removed from $[\lam]$ to obtain another Young diagram $[\mu]$.
Given a partition $\lam$,
unwrapping a rim $e$-hook of $[\lam]$ one by one until none can be unwrapped, one can get a partition, called
the $e$-{\em core} of $[\lam]$ and the number of rim $e$-hooks unwrapped is called the $e$-{\em weight} of $\lam$.

An {\em $r$-partition} of $n$ is an $r$-tuple
$\blam=(\blam^{(1)}, \dots, \blam^{(r)})$ of partitions such that
$|\blam|=\sum_{i=1}^r|\blam^{(i)}|=n$. The partitions $\blam^{(1)}, \dots,
\blam^{(r)}$ are components of $\blam$. The
conjugate of an $r$-partition $\blam$ is defined to be
$\blam'=(\blam^{(r)'}, \dots, \blam^{(1)'})$. For $\sigma\in\mathfrak{S}_r$, 
where $\mathfrak{S}_r$ is the symmetric group on $\{1, \dots, r\}$,
define $\blam^\sigma$ to be $(\blam^{(\sigma(1))}, \dots, \blam^{(\sigma(r))})$.
Denote the set of $r$-partitions of $n$ by $\mathscr{P}_{r,n}$. Then for $\blam,
\bmu\in \mathscr{P}_{r,n}$, we write $\blam\unrhd \bmu$ (or $\bmu\unlhd\blam$) if
$$\sum_{t=1}^{s-1}|\blam^{(t)}|+\sum_{i=1}^j\blam_i^{(s)}
\geq\sum_{t=1}^{s-1}|\bmu^{(t)}|+\sum_{i=1}^j\bmu_i^{(s)}$$ for all
$1\leq s\leq r$ and all $j\geq 1$. Write $\blam\rhd\bmu$ (or $\bmu\lhd\blam$)  if
$\blam\unrhd\bmu$ and $\blam\neq\bmu$. The Young diagram of an
$r$-partition $\blam$ is the set of nodes $$[\blam]=\{(i, j, k)\mid 1\leq
i, 1\leq j\leq\blam^{(k)}_i, 1\leq k\leq r\}.$$
Given $\bs=(s_1, s_2, \dots, s_r)\in \Z^r$, define the {\em residue of node} $(i, j, k)\in [\blam]$ to be $q^{j-i+s_k}$,  and define
$c_f(\blam)$ to be the number of nodes in $[\blam]$ of residue $f$
and $\mathcal{C}_{e, \bs}(\blam)=(c_1(\blam), c_q(\blam), \dots, c_{q^{e-1}}(\blam))$.

A $\blam$-{\em tableau} is a bijective
map $\ft: [\blam]\rightarrow \{1, 2, \dots, n\}$, and $\ft$ is
called {\em standard} if the entries increase along each row and down each
column in each component. The set of standard $\blam$-tableaux is denoted by $\Std(\blam)$. The {\em residue sequence} of $\ft$ is defined to be the sequence
$(\res \ft^{-1}(1), \dots, (\res \ft^{-1}(n))$. An $r$-partition $\blam$ is said to be $(Q, e)$-{\em restricted} if there exists $\ft\in \Std(\blam)$ such that the
residue sequence of any standard tableau of shape $\bmu\lhd\blam$ is not the same as the residue sequence of $\ft$.

\smallskip

As far as we know, algebra $\mathcal {H}_n(q, Q)$ has three known cellular bases. The first one was given by
Graham and Lehrer in \cite{GL} using Kazhdan-Lusztig basis of
$\mathcal {H}(S_n)$ (the Iwahori-Hecke algebra of $S_n$). We shall not use it in this paper.
The second one is the standard basis, which was constructed in \cite{DJM} by
Dipper, James and Mathas. It is a generalization of the Murphy basis of $\mathcal {H}(S_n)$ \cite{Mr}. Then one has cell (Specht) modules
$W(\blam)$, where $\blam$ are $r$-partitions. Moreover, Ariki
\cite{A2} proved that if $\blam$ is a Kleshchev $r$-partition, then
$L(\blam)$, the top of $W(\blam)$ is simple, and  $\{L(\blam)\mid
\blam\,\, \rm Kleshchev\}$ provide a complete set of simple $\mathcal
{H}_n(q, Q)$-modules up to isomorphism.

Recall that given a finite dimensional algebra $A$, two $A$-modules belong to the same block
if all of their composition factors belong to the same block. For a cellular algebra,
according to \cite[3.9.8]{GL}, all composition factors of a cell module belong to the same block.
Therefore, we can say in this sense a cell module belongs to some block.
Note that if $W(\blam)$ lies in block $\mathcal {H}_{\bbeta}^{\bLambda}$, then $\bbeta=\sum_{i\in I}c_{q^i}(\blam)\balpha_i$.
We often abuse notation to say that $\blam$ lies in block $\mathcal {H}_{\bbeta}^{\bLambda}$. Since the definition of cell module depends on the chosen $\bs$, 
we will write pair $(\blam, \bs)$ in some block if necessary.

Given two pairs $(\blam, \bs)$ and $(\bmu, \bs)$, the following lemma provides a criterion of being in the same block.

\begin{lemma}\cite[Proposition 5.9 (ii)]{GL}\cite[Theorem 2.11]{LM}\label{residueblock}
Pairs $(\blam, \bs)$ and $(\bmu, \bs)$ belong to the same block if and only of $c_f(\blam)=c_f(\bmu)$ for all $f\in K$.
\end{lemma}

The third cellular basis of $\mathcal {H}_n(q, Q)$ was given by Hu and Mathas in \cite{HM},
which is compatible with the block decomposition of $\mathcal {H}_n(q, Q)$.
\begin{lemma}\cite[Lemma 5.4, Corollary 5.5, Theorem 5.8 and Corollary 5.12]{HM}\label{HMbasis}
The following statements hold.
\begin{enumerate}
\item[{\rm (1)}] The algebra $\mathcal {H}_n(q, Q)$ is a graded cellular algebra with poset $(\mathscr{P}_{r,n}, \unrhd)$ and graded cellular basis $\{\psi_{\fs,\ft}^\blam\mid\blam\in\mathscr{P}_{r,n}, \fs, \ft\in {\rm Std}(\blam)\}$  {\rm(HM basis)}. 
    The corresponding ungraded cell modules coincide with the cell modules determined by the standard basis, respectively.
\item[{\rm (2)}] Let $\mathcal{H}_\bbeta^\bLambda$ be a block of $\mathcal {H}_n(q, Q)$. 
Then there exists $\mathscr{P}_{\bbeta}^{\bLambda}\subseteq\mathscr{P}_{r, n}$ such that $\{\psi_{\fs,\ft}^\blam\mid\blam\in\mathscr{P}_{\bbeta}^{\bLambda},
\fs, \ft\in {\rm Std}(\blam)\}$ is a graded cellular basis of $\mathcal {H}_\bbeta^\bLambda$.
\end{enumerate}
\end{lemma}

We also need the following easy results about $\mathcal {H}_n(q, Q)$.

\begin{lemma} \cite[Lemma 4.5]{HC}\label{permortationisomorphism}
For $0\neq c\in K$, we have $\mathcal {H}_n(q, Q_1, \dots, Q_r)\cong\mathcal {H}_n(q, cQ_1, \dots, cQ_r)$. For any permutation $\sigma$ of $r$ letters,  $\mathcal {H}_n(q, Q_1, \dots, Q_r)=\mathcal {H}_n(q, Q_{\sigma(1)}, \dots, Q_{\sigma(r)})$.
\end{lemma}

Based on Lemmas \ref{residueblock} and \ref{permortationisomorphism}, pairs $(\blam^\sigma, \bs^{\sigma})$ and $(\blam, \bs)$ are in the same block, where $\sigma\in \mathfrak{S}_r$.

\smallskip

We now begin to recall some results about $r$-abaci. The main reference is \cite{JL}.
Abaci first appeared in the work of James \cite{J2}.
Given a partition $\lam$ and $s\in \mathbb{Z}$, one can associate it
to a set of integers $L_s(\lam)=\{\lam_j-j+s\mid j\in\mathbb{N}^+\}$.
We assume that $\lam$ has an infinite number of zero parts.
As is well-known, the set $L_s(\lam)$ can be expressed by an {\em abacus}.
Let us illustrate an example.
\begin{example}
Let $\lam=(7, 5, 4, 1, 1)$ and $s=0$. For each $i\in L_s(\lam)$,
we set a bead at the $i$-th position on the horizontal abacus.
Then $L_s(\lam)$ is expressed as below. A position without a bead
will be called empty.
\[
\begin{tikzpicture}
[scale=0.5, bb/.style={draw,circle,fill,minimum size=2.5mm,inner sep=0pt,outer sep=0pt},
wb/.style={draw,circle,fill=white,minimum size=2.5mm,inner sep=0pt,outer sep=0pt}]
\foreach \x in {13,-10}
\foreach \y in {2}
{
\node at (\x,\y) {$\cdots$};
}	
\node [] at (2,1) {$0$};
\node [] at (3,1) {$1$};
\node [] at (4,1) {$2$};
\node [] at (5,1) {$3$};
\node [] at (7,1) {$\cdots$};
\node [] at (0.9,1) {$-1$};
\node [] at (-0.3,1) {$-2$};
\node [] at (-3,1) {$\cdots$};
\node [wb] at (12,2) {};
\node [wb] at (11,2) {};
\node [wb] at (10,2) {};
\node [wb] at (9,2) {};
\node [bb] at (8,2) {};
\node [wb] at (7,2) {};
\node [wb] at (6,2) {};
\node [bb] at (5,2) {};
\node [wb] at (4,2) {};
\node [bb] at (3,2) {};
\node [wb] at (2,2) {};
\node [wb] at (1,2) {};
\node [wb] at (0,2) {};
\node [bb] at (-1,2) {};
\node [bb] at (-2,2) {};
\node [wb] at (-3,2) {};
\node [bb] at (-4,2) {};
\node [bb] at (-5,2) {};
\node [bb] at (-6,2) {};
\node [bb] at (-7,2) {};
\node [bb] at (-8,2) {};
\node [bb] at (-9,2) {};
\draw[dashed](1.5,1)--node[]{}(1.5,2.5);
\end{tikzpicture}
\]
Note that we always draw a dashed vertical line between positions $-1$ and $0$. Then the numbers below the positions will often be omitted.
\end{example}

\begin{definition}\label{e-infinite}
If $e=\infty$, then both $j \pmod e$ and $j+ke$, $k\in \Z$ are defined to be $j$, and $0\le j\le e-1$ means $j \in \Z$.
\end{definition}

An abacus $L_s(\lam)$ can also be represented by an $e$-tuple of abacus,
in which the abaci are labeled by $L_0, L_1, \dots, L_{e-1}$ from bottom to top.
For each $k\in L_s(\lam)$, if $k=ye+x$ with $x, y\in\mathbb{Z}$ and $0\leq x<e$,
then place a bead in position $(x, y)$, which means the $y$-th position of $L_x$.
We will denote this $e$-tuple by $\mathcal{L}_s^e(\lam)$.

Let $\bs=(s_1, s_2, \dots, s_r)\in \mathbb{Z}^r$ be a multicharge and $\blam=(\blam^{(1)}, \dots, \blam^{(r)})$ an $r$-partition.
Then the pair $(\blam, \bs)$ can be associated with an $r$-abacus $L_\bs(\blam)$
by setting abaci $L_{s_i}(\blam^{(i)})$, $i=1, \dots, r$, from bottom to top so that all positions $0$ of each abacus appear in the same vertical line,
which will be called $(e, \bs)$-abacus of pair $(\blam, \bs)$.
Note that each $r$-abacus $L_\bs(\blam)$ can be mapped to an abacus by Uglov map, whose definition is as follows.

\begin{definition}\cite[Section 4.1]{U}, \cite[Section 2.4]{JL}\label{Uglovmap}
Let $\blam$ be an $r$-partition. Then the image of pair $(\blam, \bs)$ under Uglov map $\tau_{e, r}$ is $(\lam, s)$, 
which is defined as follows. For each bead at position $(x, y)$ in $L_\bs(\blam)$,
let $y = k.e + c$ with $k \in \Z$ and $c \in \{0, \dots, e-1\}$. Then we set a bead in the new 1-abacus $L_s(\lam)$ at position $(r-x)e+ker +c$.
\end{definition}

An intuitive explanation of Uglov map is as follows: cut up the abacus $L_\bs(\blam)$ into sections with positions 
$\{ae, ae+1, \dots, ae+e-1\} (a\in\mathbb{Z})$, and put the section with positions $\{ae, ae+1, \dots, ae+e-1\}$ 
on top of that with positions $\{(a+1)e, (a+1)e+1, \dots, (a+1)e+e-1\}$. This yields the $e$-tuple abacus display of $\tau_{e, r}(L_{\bs}(\blam))$.  
Let us illustrate an example below.

\begin{example}
Let $e=3$, $\bs=(0,0, 2)$ and $\blam=((2), (3,1),(1,1))$. Then $L_{\bs}(\blam)$ is as follows.
\begin{center}
\begin{tikzpicture}[scale=0.5, bb/.style={draw,circle,fill,minimum size=2.5mm,inner sep=0pt,outer sep=0pt}, wb/.style={draw,circle,fill=white,minimum size=2.5mm,inner sep=0pt,outer sep=0pt}]
	
\foreach \x in {10,-9}
\foreach \y in {2,1,0}
{
\node at (\x,\y) {$\cdots$};
}	
	\node [wb] at (9,2) {};
	\node [wb] at (8,2) {};
	\node [wb] at (7,2) {};
	\node [wb] at (6,2) {};
	\node [wb] at (5,2) {};
	\node [wb] at (4,2) {};
	\node [bb] at (3,2) {};
	\node [bb] at (2,2) {};
	\node [wb] at (1,2) {};
	\node [bb] at (0,2) {};
	\node [bb] at (-1,2) {};
	\node [bb] at (-2,2) {};
	\node [bb] at (-3,2) {};
	\node [bb] at (-4,2) {};
	\node [bb] at (-5,2) {};
	\node [bb] at (-6,2) {};
	\node [bb] at (-7,2) {};
	\node [bb] at (-8,2) {};
	
	\node [wb] at (9,1) {};
	\node [wb] at (8,1) {};
	\node [wb] at (7,1) {};
	\node [wb] at (6,1) {};
	\node [wb] at (5,1) {};
	\node [wb] at (4,1) {};
	\node [bb] at (3,1) {};
	\node [wb] at (2,1) {};
	\node [wb] at (1,1) {};
	\node [bb] at (0,1) {};
	\node [wb] at (-1,1) {};
	\node [bb] at (-2,1) {};
	\node [bb] at (-3,1) {};
	\node [bb] at (-4,1) {};
	\node [bb] at (-5,1) {};
	\node [bb] at (-6,1) {};
	\node [bb] at (-7,1) {};
	\node [bb] at (-8,1) {};
	
	\node [wb] at (9,0) {};
	\node [wb] at (8,0) {};
	\node [wb] at (7,0) {};
	\node [wb] at (6,0) {};
	\node [wb] at (5,0) {};
	\node [wb] at (4,0) {};
	\node [wb] at (3,0) {};
	\node [bb] at (2,0) {};
	\node [wb] at (1,0) {};
	\node [wb] at (0,0) {};
	\node [bb] at (-1,0) {};
	\node [bb] at (-2,0) {};
	\node [bb] at (-3,0) {};
	\node [bb] at (-4,0) {};
	\node [bb] at (-5,0) {};
	\node [bb] at (-6,0) {};
	\node [bb] at (-7,0) {};
	\node [bb] at (-8,0) {};
	
	\draw[](-5.5,-0.5)--node[]{}(-5.5,2.5);
	\draw[](-2.5,-0.5)--node[]{}(-2.5,2.5);
	\draw[dashed](0.5,-0.5)--node[]{}(0.5,2.5);
		\draw[](3.5,-0.5)--node[]{}(3.5,2.5);
	\draw[](6.5,-0.5)--node[]{}(6.5,2.5);
	\end{tikzpicture}
\end{center}
The $e$-tuple abacus of $\tau_{e, r}(L_{\bs}(\blam))$ is
\begin{center}
\begin{tikzpicture}[scale=0.5, bb/.style={draw,circle,fill,minimum size=2.5mm,inner sep=0pt,outer sep=0pt}, wb/.style={draw,circle,fill=white,minimum size=2.5mm,inner sep=0pt,outer sep=0pt}]
\foreach \x in {-1,0,1}
\foreach \y in {5,-3.5}
{
\node at (\x,\y) {$\vdots$};
}
\node[bb] at (-1, 4){};
\node[bb] at (-1, 3){};
\node[bb] at (-1, 2){};
\node[bb] at (-1, 1){};
\node[wb] at (-1, 0){};
\node[wb] at (-1, -1){}; 
\node[wb] at (-1, -2){}; 
\node[wb] at (-1, -3){}; 

\node[bb] at (0, 4){};
\node[bb] at (0, 3){};
\node[wb] at (0, 2){};
\node[bb] at (0, 1){};
\node[bb] at (0, 0){};
\node[wb] at (0, -1){}; 
\node[bb] at (0, -2){}; 
\node[wb] at (0, -3){}; 

\node[bb] at (1, 4){};
\node[bb] at (1, 3){};
\node[bb] at (1, 2){};
\node[wb] at (1, 1){};
\node[bb] at (1, 0){};
\node[bb] at (1, -1){}; 
\node[wb] at (1, -2){}; 
\node[wb] at (1, -3){}; 

\draw[dashed](-1.5,0.5)--node[]{}(1.5,0.5);
	\end{tikzpicture}
	\end{center}

\end{example}

It is not difficult to check that the Uglov map is a bijection (see \cite[Section 4.1]{U}) and $s=\sum_is_i$.

\smallskip

In an $r$-abacus $L_\bs(\blam)$, the number of beads in column $k$ is denoted by
$\mathfrak{c}_k(L_\bs(\blam))$, or $\mathfrak{c}_k$ if there is no danger of confusion.
Denote by $\CIRCLE_j^i(\blam, \bs)$ the $j$-th bead from right to left of $L_{s_i}(\blam^{(i)})$.
If there is no danger of confusion, we write it simply as $\CIRCLE_j^i$.
The following lemma is not difficult to check and we omit the proof.

\begin{lemma}\label{bead minus}
For the $(e, \bs)$-abacus $L_\bs(\blam)$ of pair $(\blam, \bs)$, we have
\begin{enumerate}
\item[{\rm (1)}]\, The number of empty positions between $\CIRCLE_x^i$ and $\CIRCLE_{x+1}^i$ is equal to $\blam^{(i)}_x-\blam^{(i)}_{x+1}$.
\item[{\rm (2)}]\, The number of empty positions to the left of $\CIRCLE_j^i$ is equal to $\blam^{(i)}_j$.
\item[{\rm (3)}]\, Assume that the number of beads in $L_{s_i}(\blam^{(i)})$ on the right side of the dashed vertical line (between positions $-1$ and $0$) is $n_{i1}$
and that the number of empty positions on the left side of the dashed vertical line is $n_{i2}$. Then $n_{i1}-n_{i2}=s_i$.
\end{enumerate}
\end{lemma}

Given a pair, the associated abacus is clearly uniquely determined.
Conversely,  Lemma \ref{bead minus} implies that one can determine a unique pair by a given abacus.
For this reason, we often do not distinguish between a pair and its associated abacus.

By using Lemma \ref{bead minus} we can study the relationship between the numbers of certain positions in two $r$-abaci. 
Let us define a notation first.

\begin{definition}\label{nxk}
Define $\mathfrak{n}_x^k(L_\bs(\blam))$ to be the number of beads on the right side of the $k$-th position in row $x$.
If there is no danger of confusion, we write it simply as $\mathfrak{n}_x^k$. 
\end{definition}

\begin{lemma}\label{charge minus}
Let $L_\bs(\blam)$ and $L_\bu(\bmu)$ be two $r$-abaci and $1\leq i,\, j\leq r$.
Let $h$ be an integer such that all positions $(i, l)$ in $L_{s_i}(\blam^{(i)})$ and $(j, l)$ in $L_{u_i}(\bmu^{(j)})$ have a bead placed,
where $l\leq h$. Then $\mathfrak{n}_j^h(L_\bu(\bmu))-\mathfrak{n}_i^h(L_\bs(\blam))=u_j-s_i$.
\end{lemma}

\begin{proof}
Without loss of generality, we can assume $h<0$. By dividing the $\mathfrak{n}_j^h(L_\bu(\bmu))$
beads into two parts with the dashed vertical line, one can obtained the following equality:
$\mathfrak{n}_j^h(L_\bu(\bmu))=-h-1-n_{j2}\,+\,n_{j1}$. Note that by Lemma \ref{bead minus} (3), $-n_{j2}+n_{j1}=u_j$
and hence $\mathfrak{n}_j^h(L_\bu(\bmu))=-h-1+u_j$. Similarly, we have $\mathfrak{n}_i^h(L_\bs(\blam))=-h-1+s_i$.
Combining the two equalities above, we complete the proof.
\end{proof}

Let us introduce the definition of subabacus, which is important for subsequent sections.

\begin{definition}\label{subabacus}
Given an abacus $L_{\bs}(\blam)$ and $0\leq i< e$, the diagram obtained
by putting all $j$-th columns with $j\equiv i \,({\rm mod} \,e)$ together
in the original order and with the original labels is called the $i$-th {\em subabacus} of $L_\bs(\blam)$.
For arbitrary $x\in\Z$, we sometimes say $x$-th subabacus,
which means the $i$-th one, $x\equiv i \,({\rm mod} \,e)$ and $0\leq i< e$.
\end{definition}

Clearly, if $e<\infty$, for a given abacus $L_\bs(\blam)$, both the number of the beads in $j-1$-th and $j$-th subabaci are infinite.
However, if we ignore positions $j-1+ke$ and $j+ke$, $k\in\Z$, as long as both them being occupied by a bead,
then in a natural way, we can say the difference between the number of beads in $j-1$-th subabacus and that in $j$-th one.
This induces the following definition.

\begin{definition}\label{beaddifference}
The difference between the number of beads in the $j-1$-th subabacus and that in the $j$-th one is defined to be $\mathfrak{m}^{j-1}_j(L_\bs(\blam))=\sum_{k\in\Z}(\mathfrak{c}_{j-1+ke}-\mathfrak{c}_{j+ke})$.
\end{definition}

For two abaci $L$ and $L'$, we write $L \subseteq L'$ following \cite{JL} if
for each bead in position $i$ of $L$, there is a bead in position $i$ in $L'$.
We also need to recall two subsets of $\mathbb{Z}^r$.
$$\overline{\mathcal A}^r_e:=\{ (s_1, \dots, s_r) \in \mathbb{Z}^r \mid \forall\, i,\, j\in \{1,\dots,r\}, i<j,\, 0\le s_j-s_i\le e\},$$
$${\mathcal A}^r_e:=\{(s_1,\dots,s_r)\in \mathbb{Z}^r \mid \forall\, i,\, j\in \{1,\dots,r\}, i<j, \,0\le s_j-s_i<e\}.$$
Clearly, ${\mathcal A}^r_e\subset\overline{\mathcal A}^r_e$.

\begin{definition}\cite[Definitions 2.8, 2.10]{JL}\label{complete}
An $r$-abacus is called $(e, \bs)$-complete (or ``complete" for simple) if 
$$L_{s_1}(\blam^{(1)})\subset L_{s_2}(\blam^{(2)}) \subset \cdots\subset L_{s_r}(\blam^{(r)}) \subset L_{s_1+e}(\blam^{(1)}).$$
A pair $(\blam, \bs)$ is said to be a reduce $(e, \bs)$-core if its $(e, \bs)$-abacus is $(e, \bs)$-complete.
\end{definition}

By a simple observation,  Jacon and Lecouvey revealed a relation
between reduced $(e, \bs)$-cores and $\overline{\mathcal A}^r_e$.

\begin{lemma}\label{3.1.6}
\cite[Proposition 2.11]{JL}
Given a pair $(\blam, \bs)$, if it is a reduced $(e, \bs)$-core, then $\bs\in \overline{\mathcal A}^r_e$.
\end{lemma}

For convenience of later use, we give the following definition.

\begin{definition}\label{nonexistposition}
For positions in an $r$-abacus, we set $(r+j, h):=(j, h-e)$ for $1\le j<r$ when $e\ne \infty$.
\end{definition}

Let us collect some simple properties about the $(e, \bs)$-abacus of a pair $(\blam, \bs)$.

\begin{lemma}\label{abacus simple property}
Given a pair $(\blam, \bs)$ and its $(e, \bs)$-abacus, we have
\begin{enumerate}
\item[{\rm (1)}]\, Let $1\leq i\leq r$ and $e<\infty$. If position $(i, l)$ is empty, then there exists some $k\in \mathbb{N}$,
such that there is a bead at position $(i, l-(k+1)e)$ and position $(i, l-ke)$ is empty.
If position $(i, l)$ has a bead, then there exists some $k\in\mathbb{N}$ such that
position $(i, l+ke)$ has a bead and position $(i, l+(k+1)e)$ is empty.
\item[{\rm (2)}]\, Let $\bs\in \overline{\mathcal A}^r_e$. If in $L_{\bs}(\blam)$,
positions $(i, l_t)$ have a bead placed and positions $(i+j, l_t)$ are empty (we assume $e\neq \infty$ whenever $i+j>r$),
where $t=1, \dots, m$, $1\leq i\leq r$ and $1\leq j< r$, then there exist $h_1, \dots, h_m$
such that positions $(i, h_x)$ are empty and positions $(i+j, h_x)$ have a bead placed for $x=1, \dots, m$.
\item[{\rm (3)}]\, Let $e\ne \infty$ and $L_\bs(\blam)$ be $(e, \bs)$-complete.
If there exists a bead in column $l$ of $L_\bs(\blam)$,
then all positions of column $l-ke$ have a bead placed for each $k\in \N^+$.
\item[{\rm (4)}]\, Assume that $s_i+k\le s_j$, for $1\le i, j\le r$ and $k\in \N^+$.
Then there exist $h_1, \dots, h_k\in \mathbb{Z}$ such that in $L_{\bs}(\blam)$,
positions $(i, h_t)$ are empty and positions $(j, h_t)$ have a bead placed, where $1\leq t\leq k$.
In particular, if $L_{\bs}(\blam)$ is complete, then $i<j$ and there exist $h_1, \dots, h_k\in \mathbb{Z}$
such that in $L_\bs(\blam)$, positions $(x, h_t)$ are empty and positions $(y, h_t)$ have a bead placed,
where $1\leq x\leq i$, $j\leq y\leq r$ and $1\leq t\leq k$.
\end{enumerate}
\end{lemma}

\begin{proof}
(1) is easy. (3) is a direct corollary of Definition \ref{complete}. 

(2) We divide the proof into two cases.

\begin{enumerate}
\item  $i+j\le r$. Then $1\le i\le i+j\le r$. Let $h$ be an integer such that all positions $(i, l)$ and $(i+j, l)$ in $L_\bs(\blam)$ have a bead placed for $l\le h$. 
Since $\bs \in \overline{\mathcal A}^r_e$, we have $s_i\le s_{i+j}$. 
By Lemma \ref{charge minus}, $\mathfrak n^h_{i+j}-\mathfrak n^h_i=s_{i+j}-s_i\ge 0$. 
This forces that if in $L_\bs(\blam)$, positions $(i, l_t)$ have a bead placed and $(i+j, l_t)$ are empty, where $t=1, \cdots, m$, then there exist $h_1, \cdots, h_m$ such that position $(i, h_x)$ are empty and positions $(i+j, h_x)$ have a bead placed for $x=1, \cdots, m$. 

\item  $i+j>r$. Recall Definition \ref{nonexistposition}. In this case, $1\le i+j-r<i\le r$. 
Let $h$ be an integer such that all positions $(i, l)$ and $(i+j-r, l)$ in $L_\bs(\blam)$ have a bead placed, 
where $l\le h$. It follows from $\bs \in \overline{\mathcal A}^r_e$ that $s_i\le s_{i+j-r}+e$. 
Note that $\mathfrak n^{h-e}_{i+j-r}=\mathfrak n^h_{i+j-r}+e$. Then by Lemma \ref{charge minus}, 
$\mathfrak n^{h-e}_{i+j-r}-\mathfrak n^h_i=\mathfrak n^h_{i+j-r}+e-\mathfrak n^h_i=s_{i+j-r}+e-s_i\ge 0$. As a result,
If in $L_\bs(\blam)$, position $(i, l_t)$ have a bead placed and $(i+j-r, l_t-e)$ are empty, $t=1, \cdots, m$, 
then there exist $h_1, \cdots, h_m$ such that position $(i, h_x)$ are empty and positions $(i+j-r, h_x-e)$ have a bead placed for $x=1, \cdots, m$. 
\end{enumerate}

(4) Let $l$ be an integer such that in $L_\bs(\blam)$ all positions $(x, y)$ have a bead placed, where $1\leq x\leq r$ and $y\leq l$.
Since $s_i+k\le s_j$, by Lemma \ref{charge minus} $\mathfrak{n}_j^l-\mathfrak{n}_i^l=s_j-s_i\geq k$.
We have proved the first half.
Moreover, if $L_\bs(\blam)$ is complete, then in $L_\bs(\blam)$ position $(i, h)$ being empty
forces positions $(x, h)$ to be empty for $1\leq x\leq i$, $h\in\mathbb{Z}$, and position $(j, h)$
being occupied by a bead implies that all positions $(y, h)$ have a bead placed,
where $j\leq y\leq r$. Consequently, the second half holds.
\end{proof}

%%%%%%%%%%%%%%%%%%%%%%%%%%%%%%%%%%%%%%%%%%%%%%%%%%%%%%%%%%%%%%%%%%%%%%%%%%%%%%%%%%%%%%%%%%%%%%%%%%%%%%%%%%%%%
%%%%%%%%%%%%%%%%%%%%%%%%%%%%%%%%%%%%%%%%%%%%%%%%%%%%%%%%%%%%%%%%%%%%%%%%%%%%%%%%%%%%%%%%%%%%%%%%%%%%%%%%%%%%%

\subsection{Orbits of abaci}
In this subsection, we give an explanation by abaci
for the action of an affine Weyl group on blocks.
Given a pair $(\blam, \bs)$, it clearly belongs a unique block.
Denote by $\bLambda_{\bs}$ and $\bbeta_{\blam, \bs}$ the
corresponding elements in dominant weight lattice and positive root lattice, respectively.
For $j\in I$ and pair $(\blam, \bs)$, define an action of $\sigma_j$ on $(\blam, \bs)$
by $\sigma_j(\blam, \bs)=(\sigma_j(\blam), \sigma_j(\bs))$,
where the abacus of $(\sigma_j(\blam), \sigma_j(\bs))$ is obtained
by interchanging columns $j-1+ke$ and $j+ke$ in $L_\bs(\blam)$ for all $k\in \Z$.
A simple result about the action is that the charge is invariant.

\begin{lemma}\label{3.2.1}
For $j\in I$, $\sigma_j(\bs)=\bs$.
\end{lemma}

\begin{proof}
It is a direct corollary of the definition of the action of $\sigma_j$ on a pair and Lemma \ref{bead minus} (3).
\end{proof}

To achieve the target of this subsection, we first do some work on 1-abaci. Given an abacus $L_s(\lam)$, let $x\ge 1$ and $\CIRCLE_x$ at the $h$-th position. If the $h+1$-th position is empty,
then move $\CIRCLE_x$ to the $h+1$-th position and denote by $(\mu, u)$ the new abacus.
By Lemma \ref{bead minus} (2), $\mu_x=\lambda_x+1$, and $\lambda_i=\mu_i$  for each $i\ne x$.
 Moreover, we have from Lemma \ref{bead minus} (3) that $u=s$.
Using this method, each abacus $L_s(\lam)$ can be obtained from $L_s(\varnothing)$ by finite steps.
Clearly, this process is invertible.

\begin{lemma}\label{3.2.2}
If $\CIRCLE_i$ is at $j+ke$-th position in $L_s(\lam)$, where $0\le j\le e-1$,
then the residue of node $(i, \lam_i)$ is $j$.
\end{lemma}

\begin{proof}
By the definition of abacus $L_s(\lam)$, we have $s+\lambda_i-i=j+ke$, that is,
the residue of node $(i, \lam_i)$ is $j$.
\end{proof}

\begin{lemma}\label{3.2.3}
Assume that in abacus $L_s(\lam)$, position $j+ke$ is empty and there is a bead at position $j-1+ke$, where $0\le j\le e-1$.
Denote by $L_s(\mu)$ the abacus obtained by moving the bead at position $j-1+ke$ to position $j+ke$.
Then $\bbeta_{\mu, s}=\bbeta_{\lambda, s}+\balpha_j$.
\end{lemma}

\begin{proof}
Suppose that the bead at position $j-1+ke$ is $\CIRCLE_i$. Then $\mu_i=\lambda_i+1$
and $\mu_x=\lambda_x$ for all $x\ne i$. We can deduce from Lemma \ref{3.2.2} that the residue of node $(i, \mu_i)$ in $[\mu]$ is $j$.
That is, $\bbeta_{\mu, s}=\bbeta_{\lambda, s}+\balpha_j$.
\end{proof}

Recall Definition \ref{beaddifference} of $\mathfrak{m}^{j-1}_j$.

\begin{lemma}\label{3.2.4}
Given an abacus $L_s(\lam)$ and $j\in I$, we have
$\bbeta_{\sigma_j(\lambda), s}=\bbeta_{\lambda, s}+\mathfrak{m}^{j-1}_j\balpha_j.$
\end{lemma}

\begin{proof}
Firstly, let us find all pairs of positions that need to interchange under $\sigma_j$.
Assume $X=\{b_i \mid i=1, \dots, l\}$ are all integers such that
position $j-1+b_ie$ is empty and position $j+b_ie$ has a bead. Note that $X$ may be empty.
Therefore, there is a set $Y=\{a_h\in \Z\mid h=1, \dots, l+\mathfrak{m}^{j-1}_j\}$ such that
position $j-1+a_he$ has a bead and position $j+a_he$ is empty. By Lemma \ref{3.2.3},
$\bbeta_{\sigma_j(\lambda, s)}=\bbeta_{\lambda, s}+(\mathfrak{m}^{j-1}_j+l)\balpha_j-l\balpha_j=\bbeta_{\lambda, s}+\mathfrak{m}^{j-1}_j\balpha_j.$
\end{proof}

The $r$-version of Lemma \ref{3.2.4} is not difficult to know. We write it below without a proof.

\begin{lemma}\label{rversion3.2.4}
Given an abacus $L_\bs(\blam)$ and $j \in I$, we have
$$\bbeta_{\sigma_j(\blam), \bs}=\bbeta_{\blam, \bs}+\mathfrak{m}^{j-1}_j\balpha_j.$$ 
\end{lemma}

The following result can be proved by induction. We omit the details here.

\begin{lemma}\label{3.2.5}
For $j\in I$ and a pair $(\lam, s)$, we have $\sigma_j(\bLambda_{s}-\bbeta_{\lambda, s})=\bLambda_{s}-\bbeta_{\sigma_j(\lambda), s}$.
\end{lemma}

The $r$-version of Lemma \ref{3.2.5} is as follows.

\begin{proposition}\label{3.2.6}
For $0\le j\le e-1$ and pair $(\blam, \bs)$, we have
$$\sigma_j(\bLambda_{\bs}-\bbeta_{\blam, \bs})=\bLambda_{\bs}-\bbeta_{\sigma_j(\blam), \bs}.$$
\end{proposition}

\begin{proof}
It follows from Lemma \ref{3.2.5} that

$\sigma_j(\bLambda_{\bs}-\bbeta_{\blam, \bs})
=\sum^r_{i=1}\sigma_j(\bLambda_{s_i}-\bbeta_{\blam^{(i)}, s_i}) 
=\sum^r_{i=1}(\bLambda_{s_i}-\bbeta_{\sigma_j(\blam^{(i)}), s_i}) 
=\bLambda_{\bs}-\bbeta_{\sigma_j(\blam), \bs}.$
\end{proof}

\begin{lemma}\label{new4.3.1}
Assume that the pair $(\blam, \bs)$ is in block $\mathcal H^\bLambda_\bbeta$. Then for $0\leq j \leq e-1$, $$\mathfrak{m}^{j-1}_j(L_\bs(\blam))=(\balpha_j, \bLambda-\bbeta).$$
\end{lemma}

\begin{proof}
For  $0\leq j \leq e-1$, it follows from Lemmas \ref{rversion3.2.4}, \ref{3.2.6} and the definition of $\sigma_j$ that

$\bLambda-\bbeta-(\balpha_j, \bLambda-\bbeta)\balpha_j=\sigma_j(\bLambda-\bbeta)=\bLambda-\bbeta_{\sigma_j(\blam), \bs}
=\bLambda-\bbeta_{\blam, \bs}-\mathfrak{m}^{j-1}_j\balpha_j.$
\end{proof}

%%%%%%%%%%%%%%%%%%%%%%%%%%%%%%%%%%%%%%%%%%%%%%%%%%%%%%%%%%%%%%%%%%%%%%%%%%%%%%%%%%%%%%%%%%%%%%%%%%%%%%%%%%%%%%%%%%%%%%%%%%%%
%%%%%%%%%%%%%%%%%%%%%%%%%%%%%%%%%%%%%%%%%%%%%%%%%%%%%%%%%%%%%%%%%%%%%%%%%%%%%%%%%%%%%%%%%%%%%%%%%%%%%%%%%%%%%%%%%%%%%%%%%%%%

\subsection{Incomparable abaci}
In this subsection we introduce the so-called incomparable abaci.
The main result is that from incomparable abaci
one can construct incomparable $r$-partitions with respect to the dominance order.

\begin{definition}\label{incomparable definition}
Given two $r$-abaci $L_\bs(\blam)$ and $L_\bs(\bmu)$ with $|\blam|=|\bmu|$.
Assume that there exist $\iota_1, \iota_2\in \mathbb{Z}$ and $\kappa_1\neq \kappa_2$, $1\leq \kappa_1, \kappa_2\leq r$ such that
\begin{enumerate}
\item[{\rm (1)}] In $L_\bs(\blam)$, position $(\kappa_1, \iota_1)$ has a bead and position $(\kappa_2, \iota_2)$ is empty,
and in $L_\bs(\bmu)$, position $(\kappa_1, \iota_1)$ is empty and position $(\kappa_2, \iota_2)$ has a bead.

\item[{\rm (2)}] The beads on the right side of $\iota_1$-th position in $L_{s_{\kappa_{1}}}(\blam^{(\kappa_{1})})$
are the same as those in $L_{s_{\kappa_{1}}}(\bmu^{(\kappa_1)})$,
and the beads on the left side of $\iota_2$-th position in $L_{s_{\kappa_{2}}}(\blam^{(\kappa_2)})$ are the same as those in $L_{s_{\kappa_{2}}}(\bmu^{(\kappa_2)})$.
\end{enumerate}
Then we say $L_\bs(\blam)$ and $L_\bs(\bmu)$ are incomparable, which will be denoted by $L_\bs(\blam)\parallel L_\bs(\bmu)$, or $L_\bs(\bmu)\parallel L_\bs(\blam)$.
\end{definition}

\begin{example}\label{3.3.2}
Let $e=5$, $\blam=((2, 1, 1), (2, 2, 1, 1), (3, 1, 1), (4, 3, 1, 1))$, and $\bs=(1,0 ,2, 0)$. Then $L_\bs(\blam)$ is

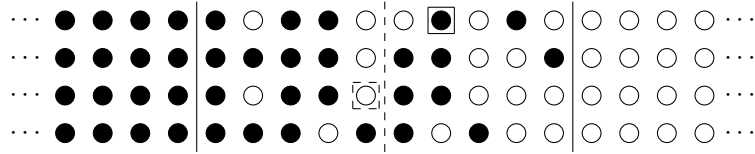
\begin{center}

\begin{tikzpicture}[scale=0.5, bb/.style={draw,circle,fill,minimum size=2.5mm,inner sep=0pt,outer sep=0pt},
wb/.style={draw,circle,fill=white,minimum size=2.5mm,inner sep=0pt,outer sep=0pt}]

\foreach \x in {9,-10}
\foreach \y in {4, 3, 2, 1}
{
\node at (\x,\y) {$\cdots$};
}

    \node [wb] at (8,4) {};
	\node [wb] at (7,4) {};
	\node [wb] at (6,4) {};
	\node [wb] at (5,4) {};
	\node [wb] at (4,4) {};
	\node [bb] at (3,4) {};
	\node [wb] at (2,4) {};
	\node [bb] at (1,4) {};
	\node [wb] at (0,4) {};
	\node [wb] at (-1,4) {};
	\node [bb] at (-2,4) {};
	\node [bb] at (-3,4) {};
	\node [wb] at (-4,4) {};
	\node [bb] at (-5,4) {};
	\node [bb] at (-6,4) {};
	\node [bb] at (-7,4) {};
	\node [bb] at (-8,4) {};
	\node [bb] at (-9,4) {};

	\node [wb] at (8,3) {};
	\node [wb] at (7,3) {};
	\node [wb] at (6,3) {};
	\node [wb] at (5,3) {};
	\node [bb] at (4,3) {};
	\node [wb] at (3,3) {};
	\node [wb] at (2,3) {};
	\node [bb] at (1,3) {};
	\node [bb] at (0,3) {};
	\node [wb] at (-1,3) {};
	\node [bb] at (-2,3) {};
	\node [bb] at (-3,3) {};
	\node [bb] at (-4,3) {};
	\node [bb] at (-5,3) {};
	\node [bb] at (-6,3) {};
	\node [bb] at (-7,3) {};
	\node [bb] at (-8,3) {};
	\node [bb] at (-9,3) {};

	\node [wb] at (8,2) {};
	\node [wb] at (7,2) {};
	\node [wb] at (6,2) {};
	\node [wb] at (5,2) {};
	\node [wb] at (4,2) {};
	\node [wb] at (3,2) {};
	\node [wb] at (2,2) {};
	\node [bb] at (1,2) {};
	\node [bb] at (0,2) {};
	\node [wb] at (-1,2) {};
	\node [bb] at (-2,2) {};
	\node [bb] at (-3,2) {};
	\node [wb] at (-4,2) {};
	\node [bb] at (-5,2) {};
	\node [bb] at (-6,2) {};
	\node [bb] at (-7,2) {};
	\node [bb] at (-8,2) {};
	\node [bb] at (-9,2) {};

	\node [wb] at (8,1) {};
	\node [wb] at (7,1) {};
	\node [wb] at (6,1) {};
	\node [wb] at (5,1) {};
	\node [wb] at (4,1) {};
	\node [wb] at (3,1) {};
	\node [bb] at (2,1) {};
	\node [wb] at (1,1) {};
	\node [bb] at (0,1) {};
	\node [bb] at (-1,1) {};
	\node [wb] at (-2,1) {};
	\node [bb] at (-3,1) {};
	\node [bb] at (-4,1) {};
	\node [bb] at (-5,1) {};
	\node [bb] at (-6,1) {};
	\node [bb] at (-7,1) {};
	\node [bb] at (-8,1) {};
	\node [bb] at (-9,1) {};
	
	\draw[](-5.5,0.5)--node[]{}(-5.5,4.5);
	\draw[dashed](-0.5,0.5)--node[]{}(-0.5,4.5);
	\draw[](4.5,0.5)--node[]{}(4.5,4.5);

\draw[](1.35,3.65)--node[]{}(1.35,4.35);
\draw[](0.65,4.35)--node[]{}(1.35,4.35);
\draw[](0.65,3.65)--node[]{}(0.65,4.35);
\draw[](0.65,3.65)--node[]{}(1.35,3.65);

\draw[dashed](-1.35,1.65)--node[]{}(-0.65,1.65);
\draw[dashed](-1.35,1.65)--node[]{}(-1.35,2.35);
\draw[dashed](-0.65,1.65)--node[]{}(-0.65,2.35);
\draw[dashed](-1.35,2.35)--node[]{}(-0.65,2.35);

	\end{tikzpicture}

\end{center}
Let $\bmu=((2, 2, 2), (5, 1, 1, 1), (3), (4, 2, 1))$ and  $\bs=(1,0 ,2, 0)$. Then $L_\bs(\bmu)$ is

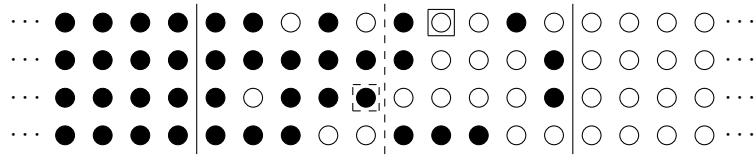
\begin{center}

\begin{tikzpicture}[scale=0.5, bb/.style={draw,circle,fill,minimum size=2.5mm,inner sep=0pt,outer sep=0pt},
wb/.style={draw,circle,fill=white,minimum size=2.5mm,inner sep=0pt,outer sep=0pt}]

\foreach \x in {9,-10}
\foreach \y in {4, 3, 2, 1}
{
\node at (\x,\y) {$\cdots$};
}

\node [wb] at (8,4) {};
	\node [wb] at (7,4) {};
	\node [wb] at (6,4) {};
	\node [wb] at (5,4) {};
	\node [wb] at (4,4) {};
	\node [bb] at (3,4) {};
	\node [wb] at (2,4) {};
	\node [wb] at (1,4) {};
	\node [bb] at (0,4) {};
	\node [wb] at (-1,4) {};
	\node [bb] at (-2,4) {};
	\node [wb] at (-3,4) {};
	\node [bb] at (-4,4) {};
	\node [bb] at (-5,4) {};
	\node [bb] at (-6,4) {};
	\node [bb] at (-7,4) {};
	\node [bb] at (-8,4) {};
	\node [bb] at (-9,4) {};

	\node [wb] at (8,3) {};
	\node [wb] at (7,3) {};
	\node [wb] at (6,3) {};
	\node [wb] at (5,3) {};
	\node [bb] at (4,3) {};
	\node [wb] at (3,3) {};
	\node [wb] at (2,3) {};
	\node [wb] at (1,3) {};
	\node [bb] at (0,3) {};
	\node [bb] at (-1,3) {};
	\node [bb] at (-2,3) {};
	\node [bb] at (-3,3) {};
	\node [bb] at (-4,3) {};
	\node [bb] at (-5,3) {};
	\node [bb] at (-6,3) {};
	\node [bb] at (-7,3) {};
	\node [bb] at (-8,3) {};
	\node [bb] at (-9,3) {};

	\node [wb] at (8,2) {};
	\node [wb] at (7,2) {};
	\node [wb] at (6,2) {};
	\node [wb] at (5,2) {};
	\node [bb] at (4,2) {};
	\node [wb] at (3,2) {};
	\node [wb] at (2,2) {};
	\node [wb] at (1,2) {};
	\node [wb] at (0,2) {};
	\node [bb] at (-1,2) {};
	\node [bb] at (-2,2) {};
	\node [bb] at (-3,2) {};
	\node [wb] at (-4,2) {};
	\node [bb] at (-5,2) {};
	\node [bb] at (-6,2) {};
	\node [bb] at (-7,2) {};
	\node [bb] at (-8,2) {};
	\node [bb] at (-9,2) {};

	\node [wb] at (8,1) {};
	\node [wb] at (7,1) {};
	\node [wb] at (6,1) {};
	\node [wb] at (5,1) {};
	\node [wb] at (4,1) {};
	\node [wb] at (3,1) {};
	\node [bb] at (2,1) {};
	\node [bb] at (1,1) {};
	\node [bb] at (0,1) {};
	\node [wb] at (-1,1) {};
	\node [wb] at (-2,1) {};
	\node [bb] at (-3,1) {};
	\node [bb] at (-4,1) {};
	\node [bb] at (-5,1) {};
	\node [bb] at (-6,1) {};
	\node [bb] at (-7,1) {};
	\node [bb] at (-8,1) {};
	\node [bb] at (-9,1) {};

	\draw[](-5.5,0.5)--node[]{}(-5.5,4.5);
	\draw[dashed](-0.5,0.5)--node[]{}(-0.5,4.5);
	\draw[](4.5,0.5)--node[]{}(4.5,4.5);

\draw[](1.35,3.65)--node[]{}(1.35,4.35);
\draw[](0.65,4.35)--node[]{}(1.35,4.35);
\draw[](0.65,3.65)--node[]{}(0.65,4.35);
\draw[](0.65,3.65)--node[]{}(1.35,3.65);

\draw[dashed](-1.35,1.65)--node[]{}(-0.65,1.65);
\draw[dashed](-1.35,1.65)--node[]{}(-1.35,2.35);
\draw[dashed](-0.65,1.65)--node[]{}(-0.65,2.35);
\draw[dashed](-1.35,2.35)--node[]{}(-0.65,2.35);

	\end{tikzpicture}

\end{center}

Take $\kappa_1=4, \,\kappa_2=2, \,\iota_1=1, \,\iota_2=-1$. It is not difficult to check that
the conditions $(1)$ and $(2)$ of Definition \ref{incomparable definition} are satisfied.
Then $L_\bs(\blam)$ and $L_\bs(\bmu)$ are incomparable abaci.
\end{example}

The original intention of defining incomparable abaci is
to find incomparable $r$-partitions that are in the same block.
An easy computation gives $\bmu\,\unrhd\,\blam$ in Example \ref{3.3.2}.
This implies that we can not conclude from two abaci $L_\bs(\blam)$ and $L_\bs(\bmu)$ being incomparable
that the corresponding $r$-partitions are incomparable with respect to the dominance order $\unrhd$.
A simple observation tells us that if $L_\bs(\blam)$ and $L_\bs(\bmu)$ are incomparable abaci,
then for arbitrary $\sigma\in\mathfrak{S}_r$, $L_{\bs^{\sigma}}(\blam^{\sigma})$ and $L_{\bs^{\sigma}}(\bmu^{\sigma})$ are incomparable too,
where $\bs^\sigma$ is defined to be $(\bs_{\sigma(1)}, \dots, \bs_{\sigma(r)})$. If we choose $\sigma=(124)$ in Example \ref{3.3.2},
then $\blam^\sigma=((4, 3, 1, 1), (2, 1, 1), (3, 1, 1), (2, 2, 1, 1))$, $\bmu^\sigma=((4, 2, 1), (2, 2, 2), (3), (5, 1, 1, 1))$, $\bs^\sigma=(0,1, 2, 0)$
and $L_{\bs^{\sigma}}(\blam^{\sigma})$ and $L_{\bs^{\sigma}}(\bmu^{\sigma})$ are

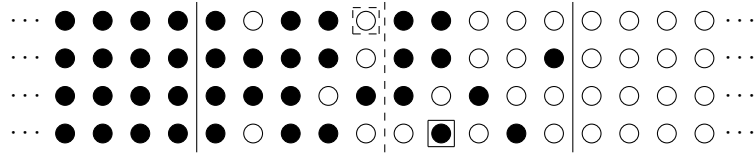
\begin{center}

\begin{tikzpicture}[scale=0.5, bb/.style={draw,circle,fill,minimum size=2.5mm,inner sep=0pt,outer sep=0pt},
wb/.style={draw,circle,fill=white,minimum size=2.5mm,inner sep=0pt,outer sep=0pt}]

\foreach \x in {9,-10}
\foreach \y in {4, 3, 2, 1}
{
\node at (\x,\y) {$\cdots$};
}

	\node [wb] at (8,4) {};
	\node [wb] at (7,4) {};
	\node [wb] at (6,4) {};
	\node [wb] at (5,4) {};
	\node [wb] at (4,4) {};
	\node [wb] at (3,4) {};
	\node [wb] at (2,4) {};
	\node [bb] at (1,4) {};
	\node [bb] at (0,4) {};
	\node [wb] at (-1,4) {};
	\node [bb] at (-2,4) {};
	\node [bb] at (-3,4) {};
	\node [wb] at (-4,4) {};
	\node [bb] at (-5,4) {};
	\node [bb] at (-6,4) {};
	\node [bb] at (-7,4) {};
	\node [bb] at (-8,4) {};
	\node [bb] at (-9,4) {};

	\node [wb] at (8,3) {};
	\node [wb] at (7,3) {};
	\node [wb] at (6,3) {};
	\node [wb] at (5,3) {};
	\node [bb] at (4,3) {};
	\node [wb] at (3,3) {};
	\node [wb] at (2,3) {};
	\node [bb] at (1,3) {};
	\node [bb] at (0,3) {};
	\node [wb] at (-1,3) {};
	\node [bb] at (-2,3) {};
	\node [bb] at (-3,3) {};
	\node [bb] at (-4,3) {};
	\node [bb] at (-5,3) {};
	\node [bb] at (-6,3) {};
	\node [bb] at (-7,3) {};
	\node [bb] at (-8,3) {};
	\node [bb] at (-9,3) {};

\node [wb] at (8,2) {};
	\node [wb] at (7,2) {};
	\node [wb] at (6,2) {};
	\node [wb] at (5,2) {};
	\node [wb] at (4,2) {};
	\node [wb] at (3,2) {};
	\node [bb] at (2,2) {};
	\node [wb] at (1,2) {};
	\node [bb] at (0,2) {};
	\node [bb] at (-1,2) {};
	\node [wb] at (-2,2) {};
	\node [bb] at (-3,2) {};
	\node [bb] at (-4,2) {};
	\node [bb] at (-5,2) {};
	\node [bb] at (-6,2) {};
	\node [bb] at (-7,2) {};
	\node [bb] at (-8,2) {};
	\node [bb] at (-9,2) {};

\node [wb] at (8,1) {};
	\node [wb] at (7,1) {};
	\node [wb] at (6,1) {};
	\node [wb] at (5,1) {};
	\node [wb] at (4,1) {};
	\node [bb] at (3,1) {};
	\node [wb] at (2,1) {};
	\node [bb] at (1,1) {};
	\node [wb] at (0,1) {};
	\node [wb] at (-1,1) {};
	\node [bb] at (-2,1) {};
	\node [bb] at (-3,1) {};
	\node [wb] at (-4,1) {};
	\node [bb] at (-5,1) {};
	\node [bb] at (-6,1) {};
	\node [bb] at (-7,1) {};
	\node [bb] at (-8,1) {};
	\node [bb] at (-9,1) {};
	
	\draw[](-5.5,0.5)--node[]{}(-5.5,4.5);
	\draw[dashed](-0.5,0.5)--node[]{}(-0.5,4.5);
	\draw[](4.5,0.5)--node[]{}(4.5,4.5);

\draw[](1.35,0.65)--node[]{}(1.35,1.35);
\draw[](0.65,1.35)--node[]{}(1.35,1.35);
\draw[](0.65,0.65)--node[]{}(0.65,1.35);
\draw[](0.65,0.65)--node[]{}(1.35,0.65);

\draw[dashed](-1.35,3.65)--node[]{}(-0.65,3.65);
\draw[dashed](-1.35,3.65)--node[]{}(-1.35,4.35);
\draw[dashed](-0.65,3.65)--node[]{}(-0.65,4.35);
\draw[dashed](-1.35,4.35)--node[]{}(-0.65,4.35);

	\end{tikzpicture}

\end{center}
and
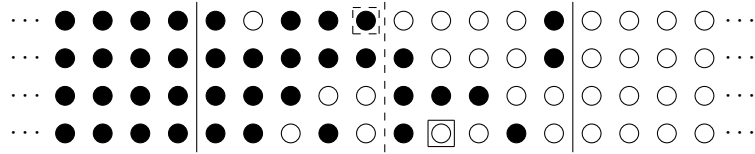
\begin{center}

\begin{tikzpicture}[scale=0.5, bb/.style={draw,circle,fill,minimum size=2.5mm,inner sep=0pt,outer sep=0pt},
wb/.style={draw,circle,fill=white,minimum size=2.5mm,inner sep=0pt,outer sep=0pt}]

\foreach \x in {9,-10}
\foreach \y in {4, 3, 2, 1}
{
\node at (\x,\y) {$\cdots$};
}

	\node [wb] at (8,4) {};
	\node [wb] at (7,4) {};
	\node [wb] at (6,4) {};
	\node [wb] at (5,4) {};
	\node [bb] at (4,4) {};
	\node [wb] at (3,4) {};
	\node [wb] at (2,4) {};
	\node [wb] at (1,4) {};
	\node [wb] at (0,4) {};
	\node [bb] at (-1,4) {};
	\node [bb] at (-2,4) {};
	\node [bb] at (-3,4) {};
	\node [wb] at (-4,4) {};
	\node [bb] at (-5,4) {};
	\node [bb] at (-6,4) {};
	\node [bb] at (-7,4) {};
	\node [bb] at (-8,4) {};
	\node [bb] at (-9,4) {};

\node [wb] at (8,3) {};
	\node [wb] at (7,3) {};
	\node [wb] at (6,3) {};
	\node [wb] at (5,3) {};
	\node [bb] at (4,3) {};
	\node [wb] at (3,3) {};
	\node [wb] at (2,3) {};
	\node [wb] at (1,3) {};
	\node [bb] at (0,3) {};
	\node [bb] at (-1,3) {};
	\node [bb] at (-2,3) {};
	\node [bb] at (-3,3) {};
	\node [bb] at (-4,3) {};
	\node [bb] at (-5,3) {};
	\node [bb] at (-6,3) {};
	\node [bb] at (-7,3) {};
	\node [bb] at (-8,3) {};
	\node [bb] at (-9,3) {};

	\node [wb] at (8,2) {};
	\node [wb] at (7,2) {};
	\node [wb] at (6,2) {};
	\node [wb] at (5,2) {};
	\node [wb] at (4,2) {};
	\node [wb] at (3,2) {};
	\node [bb] at (2,2) {};
	\node [bb] at (1,2) {};
	\node [bb] at (0,2) {};
	\node [wb] at (-1,2) {};
	\node [wb] at (-2,2) {};
	\node [bb] at (-3,2) {};
	\node [bb] at (-4,2) {};
	\node [bb] at (-5,2) {};
	\node [bb] at (-6,2) {};
	\node [bb] at (-7,2) {};
	\node [bb] at (-8,2) {};
	\node [bb] at (-9,2) {};

\node [wb] at (8,1) {};
	\node [wb] at (7,1) {};
	\node [wb] at (6,1) {};
	\node [wb] at (5,1) {};
	\node [wb] at (4,1) {};
	\node [bb] at (3,1) {};
	\node [wb] at (2,1) {};
	\node [wb] at (1,1) {};
	\node [bb] at (0,1) {};
	\node [wb] at (-1,1) {};
	\node [bb] at (-2,1) {};
	\node [wb] at (-3,1) {};
	\node [bb] at (-4,1) {};
	\node [bb] at (-5,1) {};
	\node [bb] at (-6,1) {};
	\node [bb] at (-7,1) {};
	\node [bb] at (-8,1) {};
	\node [bb] at (-9,1) {};

	\draw[](-5.5,0.5)--node[]{}(-5.5,4.5);
	\draw[dashed](-0.5,0.5)--node[]{}(-0.5,4.5);
	\draw[](4.5,0.5)--node[]{}(4.5,4.5);

\draw[](1.35,0.65)--node[]{}(1.35,1.35);
\draw[](0.65,1.35)--node[]{}(1.35,1.35);
\draw[](0.65,0.65)--node[]{}(0.65,1.35);
\draw[](0.65,0.65)--node[]{}(1.35,0.65);

\draw[dashed](-1.35,3.65)--node[]{}(-0.65,3.65);
\draw[dashed](-1.35,3.65)--node[]{}(-1.35,4.35);
\draw[dashed](-0.65,3.65)--node[]{}(-0.65,4.35);
\draw[dashed](-1.35,4.35)--node[]{}(-0.65,4.35);	

	\end{tikzpicture}

\end{center}

By a simple computation, we know $\blam^\sigma \parallel \bmu^\sigma$.
It is worth to point out that this has general significance.
That is, we can prove that if $L_\bs(\blam)\parallel L_\bs(\bmu)$, then
there exists some $\sigma\in\mathfrak{S}_r$ such that $\blam^{\sigma}$ and $\bmu^{\sigma}$ are incomparable.
To this aim, let us first give some characterizations of incomparable abaci.

\begin{lemma}\label{3.3.3}
Let $L_\bs(\blam)\parallel L_\bs(\bmu)$. Then
\begin{enumerate}
\item[{\rm (1)}] The number of empty positions on the left side of position $(\kappa_1, \iota_1+1)$
in $L_{s_{\kappa_1}}(\blam^{(\kappa_1)})$ is the same as that in $L_{s_{\kappa_1}}(\bmu^{(\kappa_1)})$.
\item[{\rm (2)}] The number of the beads on the right side of position $(\kappa_2, \iota_2-1)$
in $L_{s_{\kappa_2}}(\blam^{(\kappa_2)})$ is the same as that in $L_{s_{\kappa_2}}(\bmu^{(\kappa_2)})$.
\end{enumerate}
\end{lemma}

\begin{proof}
(1) Choose $k\in \mathbb{Z}$ such that all positions those are to the left of $k+1$-th
positions of both $L_\bs(\blam)$ and $L_\bs(\bmu)$ are occupied by a bead.
According to Definition \ref{incomparable definition} (1), $k<\iota_1$.
On the other hand, we have from Lemma \ref{charge minus} that
$\mathfrak{n}^k_{\kappa_1}(L_\bs(\blam))=\mathfrak{n}^k_{\kappa_1}(L_\bs(\bmu))$.
Then the result follows from Definition \ref{incomparable definition} (2).

(2) is proved similarly as (1).
\end{proof}

With the above result, we can get a condition for two $r$-partitions being incomparable.

\begin{lemma}\label{3.3.4}
Keep notations as in Definition \ref{incomparable definition}. Let $L_\bs(\blam)\parallel L_\bs(\bmu)$ and $\kappa_1<\kappa_2$.
If $L_\bs(\blam^{(c)})$ and $L_\bs(\bmu^{(c)})$ are the same for $c\in\{1, \dots, \kappa_1-1\}\cup\{\kappa_2+1, \dots, r\}$, then $\blam \parallel \bmu$.
\end{lemma}

\begin{proof}
Suppose that the bead at position $(\kappa_1, \iota_1)$ in $L_\bs(\blam)$ is
$\CIRCLE_k^{\kappa_1}(\blam, \bs)$. Since $L_\bs(\blam)\parallel L_\bs(\bmu)$,
by Definition \ref{incomparable definition}, the beads on the right side of
$\iota_1$-th position in $L_{s_{\kappa_1}}(\blam^{(\kappa_1)})$ are the same as those in $L_{s_{\kappa_1}}(\bmu^{(\kappa_1)})$.
This implies that for each $1\leq a<k$, if $\CIRCLE_a^{\kappa_1}(\blam, \bs)$
is in position $b$, then so is $\CIRCLE_a^{\kappa_1}(\bmu, \bs)$.
Moreover, we have from Lemma \ref{3.3.3} that the number of empty positions on the left side
of position $(\kappa_1, \iota_1+1)$ in $L_{s_{\kappa_1}}(\blam^{(\kappa_1)})$
is the same as that in $L_{s_{\kappa_1}}(\bmu^{(\kappa_1)})$. Consequently,
the number of empty positions on the left side
of $\CIRCLE_a^{\kappa_1}(\blam, \bs)$ in $L_{s_{\kappa_1}}(\blam^{(\kappa_1)})$ is the same as that on the left side
of $\CIRCLE_a^{\kappa_1}(\bmu, \bs)$ in $L_{s_{\kappa_1}}(\bmu^{(\kappa_1)})$ for $1\leq a<k$.
Then we obtain by Lemma \ref{bead minus} (2) that
\begin{align}
\blam_a^{(\kappa_1)}=\bmu_a^{(\kappa_1)}.
\end{align}

Furthermore, since the position $(\kappa_1, \iota_1)$ in $L_\bs(\bmu)$ is empty,
it is clear that $\CIRCLE_k^{\kappa_1}(\bmu, \bs)$ is on the left side of position $\iota_1$.
Using Lemma \ref{3.3.3} (1) again, there are fewer empty positions on the left side of $\CIRCLE_k^{\kappa_1}(\bmu, \bs)$ than
on the left side of $\CIRCLE_k^{\kappa_1}(\blam, \bs)$. We deduce from Lemma \ref{bead minus} (2) that
\begin{align}
\bmu_k^{(\kappa_1)}<\blam_k^{(\kappa_1)}.
\end{align}
Combining (3.3.1) with (3.3.2) yields
\begin{align}
\sum^k_{a=1}\bmu^{(\kappa_1)}_a<\sum^k_{a=1}\blam^{(\kappa_1)}_a.
\end{align}
Note that the assumption implies that the rows below $\kappa_1$-th row
in $L_\bs(\blam)$ are the same as that in $L_\bs(\bmu)$, and thus for
each $1\leq c<\kappa_1$, $\blam^{(c)}=\bmu^{(c)}$. This gives
\begin{align}
\sum^{\kappa_1-1}_{c=1}|\bmu^c|=\sum^{\kappa_1-1}_{c=1}|\blam^c|.
\end{align}
By combining (3.3.3) with (3.3.4), we get
\begin{align}
\sum^{\kappa_1-1}_{c=1}|\bmu^{(c)}|+\sum^k_{a=1}\bmu^{(\kappa_1)}_a<
\sum^{\kappa_1-1}_{c=1}|\blam^{(c)}|+\sum^k_{a=1}\blam^{(\kappa_1)}_a.
\end{align}

On the other hand, let $\CIRCLE_t^{\kappa_2}(\bmu, \bs)$ be the bead at position $(\kappa_2, \iota_2)$ in $L_\bs(\bmu)$.
Then by analyzing the number of beads as above, we can obtain the following two formulas:
\begin{enumerate}
\item[(1)] $\blam_x^{(\kappa_2)}=\bmu_x^{(\kappa_2)}$ for arbitrary $x>t$.
\smallskip
\item[(2)] $\bmu_t^{(\kappa_2)}<\blam_t^{(\kappa_2)}$.
\end{enumerate}
This implies that $\sum_{x\geq t}\bmu_x^{(\kappa_2)}<\sum_{x\geq t}\blam_x^{(\kappa_2)}.$
Combing this formula with $\sum^{r}_{c=\kappa_2+1}|\bmu^{(c)}|=\sum^{r}_{c=\kappa_2+1}|\blam^{(c)}|,$
which is an easy corollary of the assumption, leads to
\begin{align}
\sum_{x\geq t}\bmu^{(\kappa_2)}_x+\sum^r_{c=\kappa_2+1}|\bmu^{(c)}|<
\sum_{x\geq t}\blam^{(\kappa_2)}_x+\sum^r_{c=\kappa_2+1}|\blam^{(c)}|.
\end{align}

Moreover, it follows from $|\blam|=|\bmu|$ that
\begin{align}
\sum^r_{c=1}|\bmu^{(c)}|=\sum^r_{c=1}|\blam^{(c)}|.
\end{align}

Then $(3.3.7)-(3.3.6)$ is
\begin{align*}
\sum^r_{c=1}|\bmu^{(c)}|-(\sum_{x\geq t}\bmu^{(\kappa_2)}_x+
\sum^r_{c=\kappa_2+1}|\bmu^{(c)}|)>\sum^r_{c=1}|\blam^c|-(\sum_{x\geq t}\blam^{(\kappa_2)}_x+\sum^r_{c=\kappa_2+1}|\blam^{(c)}|).
\end{align*}

Note that
$$\sum^r_{c=1}|\bmu^{(c)}|-(\sum_{x\geq t}\bmu^{(\kappa_2)}_x+\sum^r_{c=\kappa_2+1}|\bmu^{(c)}|)=
\sum^{\kappa_2-1}_{c=1}|\bmu^{(c)}|+\sum^{t-1}_{x=1}\bmu^{(\kappa_2)}_x$$
and
$$\sum^r_{c=1}|\blam^c|-(\sum_{x\geq t}\blam^{(\kappa_2)}_x+\sum^r_{c=\kappa_2+1}|\blam^{(c)}|)=
\sum^{\kappa_2-1}_{c=1}|\blam^{(c)}|+\sum^{t-1}_{x=1}\blam^{(\kappa_2)}_x.$$
We arrive at
\begin{align}
\sum^{\kappa_2-1}_{c=1}|\bmu^{(c)}|+\sum^{t-1}_{x=1}\bmu^{(\kappa_2)}_x>
\sum^{\kappa_2-1}_{c=1}|\blam^{(c)}|+\sum^{t-1}_{x=1}\blam^{(\kappa_2)}_x.
\end{align}
Now we come to the conclusion $\blam\parallel\bmu$ by combining (3.3.5) with (3.3.8).
\end{proof}

We can deduce a simple fact immediately from Lemma \ref{3.3.4}, that is, two incomparable abaci
$L_\bs(\blam)\parallel L_\bs(\bmu)$ with $\kappa_1=1$ and $\kappa_2=r$ give
a pair of incomparable $r$-partitions $\blam \parallel \bmu$.
Clearly, for $L_\bs(\blam)\parallel L_\bs(\bmu)$, we can always find
$\sigma\in \mathfrak S_r$ such that $\sigma(\kappa_1)=1$ and $\sigma(\kappa_2)=r$.
As a result, $L_{\bs^{\sigma}}(\blam^{\sigma})\parallel L_{\bs^{\sigma}}(\bmu^{\sigma})$.

\begin{proposition}\label{3.3.5}
Suppose that $L_\bs(\blam)\parallel L_\bs(\bmu)$. Then there exists $\sigma\in \mathfrak S_r$
such that $\blam^{\sigma}\parallel\bmu^{\sigma}$.
\end{proposition}

%%%%%%%%%%%%%%%%%%%%%%%%%%%%%%%%%%%%%%%%%%%%%%%%%%%%%%%%%%%%%%%%%%%%%%%%%%%%%%%%%%%%%%%%%%%%%%%%%%%%%%%%
%%%%%%%%%%%%%%%%%%%%%%%%%%%%%%%%%%%%%%%%%%%%%%%%%%%%%%%%%%%%%%%%%%%%%%%%%%%%%%%%%%%%%%%%%%%%%%%%%%%%%%%%
%%%%%%%%%%%%%%%%%%%%%%%%%%%%%%%%%%%%%%%%%%%%%%%%%%%%%%%%%%%%%%%%%%%%%%%%%%%%%%%%%%%%%%%%%%%%%%%%%%%%%%%%
%%%%%%%%%%%%%%%%%%%%%%%%%%%%%%%%%%%%%%%%%%%%%%%%%%%%%%%%%%%%%%%%%%%%%%%%%%%%%%%%%%%%%%%%%%%%%%%%%%%%%%%%

\section{Moving vectors} 

The weight can measure how complicated a block is in type A.
However, this is no longer true in the cyclotomic case.
As a result, most results of type A related with weights have no cyclotomic versions. In order to resolve this problem,
we shall introduce a refinement of the weight, which is the so-called moving vector. It is
the key notion of this series of papers. 
We shall first give the definition of moving vectors and then study the simple properties.  
In Section 4.3 we provide a basic construction, which will be used in the whole series of articles.
Finally, we study incomparable abaci by the language of moving vectors.

\subsection{Definition of moving vectors}
We first do some preparation before giving the definition of moving vectors.
It is well-known that for an abacus $\mathcal{L}_s^e(\lam)$ of a partition $\lam$,
moving a bead from position $(x, y)$ to $(x, y-1)$
is equivalent to unwrapping a rim $e$-hook from $[\lam].$
An $r$-version of this operation was introduced by Jacon and Lecouvey.

\begin{definition}
\cite[Section 4.1]{JL}
Let $\blam$ be an $r$-partition of rank $n$. An elementary operation in
row $x$ on the $(e, \bs)$-abacus $L_\bs(\blam)$ is a move of one bead from
row $x$ to another.
\begin{enumerate}
\item[{\rm (1)}]\, First kind: if there is a bead at position $(x, y)$ with $1\leq x<r$ and position $(x+1, y)$ is empty,
then move the bead to position $(x+1, y)$.
\item[{\rm (2)}]\, Second kind: if $e\neq \infty$ and there is a bead at position $(r, y)$ and position $(1, y-e)$ is empty,
then move the bead to position $(1, y-e)$.
\end{enumerate}
\end{definition}

To continue our study, we need to define concepts ``before", ``after"
and ``between", which are about relationship between positions in an $r$-abacus.

\begin{definition}\label{beforeafter}
Given an abacus $L_\bs(\blam)$,
we say position $(j, h)$ is before position $(i, l)$ if one of the conditions below is satisfied
\begin{enumerate}
\item[{\rm (1)}] $h=l$ and $i<j\leq r$.
\item[{\rm (2)}] $h=l-(k+1)e$, where $k\in \mathbb{N}$.
\end{enumerate}
If position $(j, h)$ is before position $(i, l)$, then we also say position $(i, l)$ is after position $(j, h)$.
Let position $(j, h)$ be before position $(i, l)$. We say position $(x, y)$ is
between positions $(j, h)$ and $(i, l)$ if it is after $(j, h)$ and before $(i, l)$.
\end{definition}

According to Definition \ref{beforeafter}, one position is before (after) another one only if both positions are in the same subabcus (see Definition \ref{subabacus}).
Let us illustrate an example below.

\begin{example}
Take an abacus as follows.
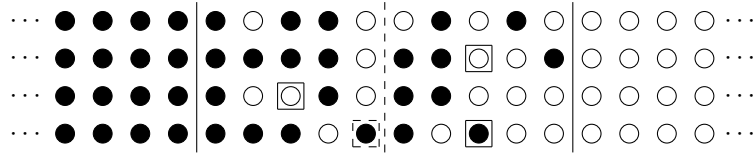
\begin{center}
\begin{tikzpicture}[scale=0.5, bb/.style={draw,circle,fill,minimum size=2.5mm,inner sep=0pt,outer sep=0pt},
wb/.style={draw,circle,fill=white,minimum size=2.5mm,inner sep=0pt,outer sep=0pt}]

\foreach \x in {9,-10}
\foreach \y in {4, 3, 2, 1}
{
\node at (\x,\y) {$\cdots$};
}

\node [wb] at (8,4) {};
	\node [wb] at (7,4) {};
	\node [wb] at (6,4) {};
	\node [wb] at (5,4) {};
	\node [wb] at (4,4) {};
	\node [bb] at (3,4) {};
	\node [wb] at (2,4) {};
	\node [bb] at (1,4) {};
	\node [wb] at (0,4) {};
	\node [wb] at (-1,4) {};
	\node [bb] at (-2,4) {};
	\node [bb] at (-3,4) {};
	\node [wb] at (-4,4) {};
	\node [bb] at (-5,4) {};
	\node [bb] at (-6,4) {};
	\node [bb] at (-7,4) {};
	\node [bb] at (-8,4) {};
	\node [bb] at (-9,4) {};

	\node [wb] at (8,3) {};
	\node [wb] at (7,3) {};
	\node [wb] at (6,3) {};
	\node [wb] at (5,3) {};
	\node [bb] at (4,3) {};
	\node [wb] at (3,3) {};
	\node [wb] at (2,3) {};
	\node [bb] at (1,3) {};
	\node [bb] at (0,3) {};
	\node [wb] at (-1,3) {};
	\node [bb] at (-2,3) {};
	\node [bb] at (-3,3) {};
	\node [bb] at (-4,3) {};
	\node [bb] at (-5,3) {};
	\node [bb] at (-6,3) {};
	\node [bb] at (-7,3) {};
	\node [bb] at (-8,3) {};
	\node [bb] at (-9,3) {};

	\node [wb] at (8,2) {};
	\node [wb] at (7,2) {};
	\node [wb] at (6,2) {};
	\node [wb] at (5,2) {};
	\node [wb] at (4,2) {};
	\node [wb] at (3,2) {};
	\node [wb] at (2,2) {};
	\node [bb] at (1,2) {};
	\node [bb] at (0,2) {};
	\node [wb] at (-1,2) {};
	\node [bb] at (-2,2) {};
	\node [wb] at (-3,2) {};
	\node [wb] at (-4,2) {};
	\node [bb] at (-5,2) {};
	\node [bb] at (-6,2) {};
	\node [bb] at (-7,2) {};
	\node [bb] at (-8,2) {};
	\node [bb] at (-9,2) {};

	\node [wb] at (8,1) {};
	\node [wb] at (7,1) {};
	\node [wb] at (6,1) {};
	\node [wb] at (5,1) {};
	\node [wb] at (4,1) {};
	\node [wb] at (3,1) {};
	\node [bb] at (2,1) {};
	\node [wb] at (1,1) {};
	\node [bb] at (0,1) {};
	\node [bb] at (-1,1) {};
	\node [wb] at (-2,1) {};
	\node [bb] at (-3,1) {};
	\node [bb] at (-4,1) {};
	\node [bb] at (-5,1) {};
	\node [bb] at (-6,1) {};
	\node [bb] at (-7,1) {};
	\node [bb] at (-8,1) {};
	\node [bb] at (-9,1) {};
	
	\draw[](-5.5,0.5)--node[]{}(-5.5,4.5);
	\draw[dashed](-0.5,0.5)--node[]{}(-0.5,4.5);
	\draw[](4.5,0.5)--node[]{}(4.5,4.5);

\draw[](-3.35,1.65)--node[]{}(-2.65,1.65);
\draw[](-3.35,1.65)--node[]{}(-3.35,2.35);
\draw[](-3.35,2.35)--node[]{}(-2.65,2.35);
\draw[](-2.65,1.65)--node[]{}(-2.65,2.35);

\draw[](1.65,2.65)--node[]{}(2.35,2.65);
\draw[](1.65,2.65)--node[]{}(1.65,3.35);
\draw[](1.65,3.35)--node[]{}(2.35,3.35);
\draw[](2.35,2.65)--node[]{}(2.35,3.35);

\draw[](1.65,0.65)--node[]{}(2.35,0.65);
\draw[](1.65,0.65)--node[]{}(1.65,1.35);
\draw[](1.65,1.35)--node[]{}(2.35,1.35);
\draw[](2.35,0.65)--node[]{}(2.35,1.35);

\draw[dashed](-1.35,0.65)--node[]{}(-0.65,0.65);
\draw[dashed](-1.35,0.65)--node[]{}(-1.35,1.35);
\draw[dashed](-0.65,0.65)--node[]{}(-0.65,1.35);
\draw[dashed](-1.35,1.35)--node[]{}(-0.65,1.35);

\end{tikzpicture}
\end{center}
Positions $(2, -3)$, $(3, 2)$ and $(1, 2)$ are all in the second subabacus. 
Positions $(2, -3)$ and $(3,2)$ are before position $(1, 2)$; 
positions $(3, 2)$ and $(1, 2)$ are after position $(2, -3)$;
position $(3, 2)$ is between positions $(1, 2)$ and $(2, -3)$. 
Position $(1, -1)$ is not in the same subabacus as others. 
Then it is neither before nor after any other positions.
\end{example}

We now give a simple result about complete abaci, in which the notions ``before" and ``after" are used.

\begin{lemma}\label{completeabaciproperty}
Let $L_\bs(\blam)$ be a complete abacus. Then for each subabacus, there exists a bead such that all positions 
before it have a bead placed and all positions after it are empty.
\end{lemma}

Let us list some observations about subabaci for later use.

\begin{enumerate}
\item[(1)] None of the positions of $L_\bs(\blam)$ belong to two different subabaci simultaneously.
\item[(2)] We can not move a bead from one subabacus to another by elementary operations.
\item[(3)] Index a bead in a subabacus by $x$ if there are exactly $x-1$ beads after it.
Then elementary operations do not change the index of a bead.
\end{enumerate}

Based on the above observations, elementary operations from $L_\bs(\blam)$ to $L_\bu(\bmu)$ is a definite set,
which is independent of the choice of the procedure of moving.
We often call this set the {\em operation set} and denote it by $\mathcal F$.
Moreover, an elementary operation happened in a subabacus does not affect other subabaci.
Therefore, each elementary operation of an abacus $L_\bs(\blam)$
can be naturally viewed as an elementary operation of a subabacus.
Given an elementary operation that moves the number $x$ bead
at position $(i, l)$, we record it by a triple $[(i, l), \,x]$.
When the index need not to be pointed out, the triple will be written as $[(i, l), \, \ast]$.

The purpose of introducing concept ``moving vector" is to study blocks of Ariki-Koike algebras.
Let us recall the definition of the core of a pair $(\blam, \bs)$, which introduced by Jacon and Lecouvey.

\begin{definition}\cite[Definition 4.2]{JL}\label{generalcore}
The core of pair $(\blam, \bs)$ is a pair $(\blam^{\ast}, \bs^{\ast})$,
whose abacus is complete and is obtained from $L_\bs(\blam)$ by elementary operations.
\end{definition}

The following is an obvious result about operation sets.

\begin{lemma}\label{3.4.3}
Let $\mathcal F$ be the operation set from $L_{\bs}(\blam)$ to its
core with $[(i, h), *]\in \mathcal F$,
where $1\le i\le r$ (if $e=\infty$, $i\ne r$) and $h\in \Z$.
We have in $L_\bs(\blam)$,
\begin{enumerate}
\item[{\rm (1)}] if position $(i+1, h)$ has a bead,
then there is an empty position before it;
\item[{\rm (2)}] if position $(i, h)$ is empty, then there is a position with a bead after it.
\end{enumerate}
\end{lemma}

An interesting fact about elementary operations is that they are ``commute" with Uglov map.

\begin{lemma}\label{movecorrespondence}
Let $\mathcal{F}$ be the operation set from $L_\bs(\blam)$ to $L_\bu(\bmu)$. Then there exists a definite operation set $\tau_{\mathcal{F}}$ from $L_s(\lam)$ to $L_u(\mu)$, where $(\lam, s)$ and $(\mu, u)$ are images of $(\blam, \bs)$ and $(\bmu, \bu)$ under Uglov map, respectively, such that the following diagram commute.
\[\begin{CD}
	L_\bs(\blam)   @>\tau_{e, r}>> L_s(\lam)\\
	@V \mathcal{F} VV                  @VV \tau_{\mathcal{F}} V\\
	L_\bu(\bmu)    @>\tau_{e, r}>> L_u(\mu)
\end{CD}\]
\end{lemma}

\begin{proof}
(Sketch) We only need to consider the case of $\mathcal{F}$ containing only one elementary operation $[(x, y), \ast]$. Let $y=ke+c$. It is not difficult to check that if $x<r$, the corresponding elementary operation is moving in $\mathcal{L}_s^e(\lam)$ the bead in position $(c, r-x+kr)$ to $(c, r-x-1+kr)$. If $x=r$,  the corresponding elementary operation is moving in $\mathcal{L}_s^e(\lam)$ the bead in position $(c, kr)$ to $(c, kr-1)$. Then it is a routine to check that the diagram commutes.
\end{proof}

Operation sets give rise to the definition of moving vectors.

\begin{definition}\label{move vector definition}
Let $\mathcal{F}$ be the operation set from $L_\bs(\blam)$ to $L_\bu(\bmu)$.
Define $$m_i=\sharp\{[(i, h), x]\in \mathcal{F}\mid h\in \mathbb{Z}, x\in \mathbb{N^+}\}.$$
Then $\mathcal{M}=(m_1, m_2, \dots, m_r)$ is called the moving vector from $L_\bs(\blam)$ to $L_\bu(\bmu)$.
To simplify the description, we sometimes write $m_{r+i}=m_i$ for $0\leq i< r$.
\end{definition}

Let us give an example of operation sets and moving vectors.

\begin{example} Let $e=3$, $\bs=(0, 2, 1)$ and $\blam=((2, 1), (3, 2), (4, 3, 1))$.
Then the associated $(e, \bs)$-abacus of pair $(\blam, \bs)$ can be represented as follows.

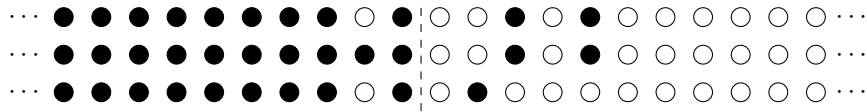
\begin{center}

\begin{tikzpicture}[scale=0.5, bb/.style={draw,circle,fill,minimum size=2.5mm,inner sep=0pt,outer sep=0pt},
wb/.style={draw,circle,fill=white,minimum size=2.5mm,inner sep=0pt,outer sep=0pt}]
	
\foreach \x in {12,-10}
\foreach \y in {0, 2, 1}
{
\node at (\x,\y) {$\cdots$};
}		
	\node [wb] at (11,0) {};
	\node [wb] at (10,0) {};
	\node [wb] at (9,0) {};
	\node [wb] at (8,0) {};
	\node [wb] at (7,0) {};
	\node [wb] at (6,0) {};
	\node [wb] at (5,0) {};
	\node [wb] at (4,0) {};
	\node [wb] at (3,0) {};
	\node [bb] at (2,0) {};
	\node [wb] at (1,0) {};
	\node [bb] at (0,0) {};
	\node [wb] at (-1,0) {};
	\node [bb] at (-2,0) {};
	\node [bb] at (-3,0) {};
	\node [bb] at (-4,0) {};
	\node [bb] at (-5,0) {};
	\node [bb] at (-6,0) {};
	\node [bb] at (-7,0) {};
	\node [bb] at (-8,0) {};
	\node [bb] at (-9,0) {};
	
	\node [wb] at (11,1) {};
	\node [wb] at (10,1) {};
	\node [wb] at (9,1) {};
	\node [wb] at (8,1) {};
	\node [wb] at (7,1) {};
	\node [wb] at (6,1) {};
	\node [bb] at (5,1) {};
	\node [wb] at (4,1) {};
	\node [bb] at (3,1) {};
	\node [wb] at (2,1) {};
	\node [wb] at (1,1) {};
	\node [bb] at (0,1) {};
	\node [bb] at (-1,1) {};
	\node [bb] at (-2,1) {};
	\node [bb] at (-3,1) {};
	\node [bb] at (-4,1) {};
	\node [bb] at (-5,1) {};
	\node [bb] at (-6,1) {};
	\node [bb] at (-7,1) {};
	\node [bb] at (-8,1) {};
	\node [bb] at (-9,1) {};
	
	\node [wb] at (11,2) {};
	\node [wb] at (10,2) {};
	\node [wb] at (9,2) {};
	\node [wb] at (8,2) {};
	\node [wb] at (7,2) {};
	\node [wb] at (6,2) {};
	\node [bb] at (5,2) {};
	\node [wb] at (4,2) {};
	\node [bb] at (3,2) {};
	\node [wb] at (2,2) {};
	\node [wb] at (1,2) {};
	\node [bb] at (0,2) {};
	\node [wb] at (-1,2) {};
	\node [bb] at (-2,2) {};
	\node [bb] at (-3,2) {};
	\node [bb] at (-4,2) {};
	\node [bb] at (-5,2) {};
	\node [bb] at (-6,2) {};
	\node [bb] at (-7,2) {};
	\node [bb] at (-8,2) {};
	\node [bb] at (-9,2) {};

\draw[dashed](0.5,-0.5)--node[]{}(0.5,2.5);
\end{tikzpicture}
\end{center}
Let $\bmu=(\varnothing, (4, 3, 1), (3, 2))$ and $\bv=(0, 1, 2)$.
Then the associated $(e, \bs)$-abacus $L_\bv(\bmu)$ can be represented as follows.

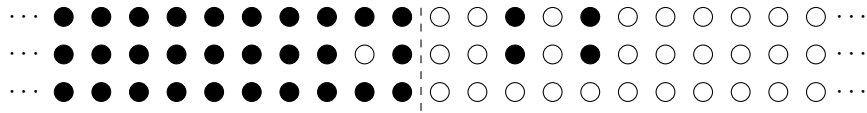
\begin{center}

\begin{tikzpicture}[scale=0.5, bb/.style={draw,circle,fill,minimum size=2.5mm,inner sep=0pt,outer sep=0pt},
wb/.style={draw,circle,fill=white,minimum size=2.5mm,inner sep=0pt,outer sep=0pt}]
	
\foreach \x in {12,-10}
\foreach \y in {0, 2, 1}
{
\node at (\x,\y) {$\cdots$};
}	
	\node [wb] at (11,0) {};
	\node [wb] at (10,0) {};
	\node [wb] at (9,0) {};
	\node [wb] at (8,0) {};
	\node [wb] at (7,0) {};
	\node [wb] at (6,0) {};
	\node [wb] at (5,0) {};
	\node [wb] at (4,0) {};
	\node [wb] at (3,0) {};
	\node [wb] at (2,0) {};
	\node [wb] at (1,0) {};
	\node [bb] at (0,0) {};
	\node [bb] at (-1,0) {};
	\node [bb] at (-2,0) {};
	\node [bb] at (-3,0) {};
	\node [bb] at (-4,0) {};
	\node [bb] at (-5,0) {};
	\node [bb] at (-6,0) {};
	\node [bb] at (-7,0) {};
	\node [bb] at (-8,0) {};
	\node [bb] at (-9,0) {};
	
	\node [wb] at (11,1) {};
	\node [wb] at (10,1) {};
	\node [wb] at (9,1) {};
	\node [wb] at (8,1) {};
	\node [wb] at (7,1) {};
	\node [wb] at (6,1) {};
	\node [bb] at (5,1) {};
	\node [wb] at (4,1) {};
	\node [bb] at (3,1) {};
	\node [wb] at (2,1) {};
	\node [wb] at (1,1) {};
	\node [bb] at (0,1) {};
	\node [wb] at (-1,1) {};
	\node [bb] at (-2,1) {};
	\node [bb] at (-3,1) {};
	\node [bb] at (-4,1) {};
	\node [bb] at (-5,1) {};
	\node [bb] at (-6,1) {};
	\node [bb] at (-7,1) {};
	\node [bb] at (-8,1) {};
	\node [bb] at (-9,1) {};
	
	\node [wb] at (11,2) {};
	\node [wb] at (10,2) {};
	\node [wb] at (9,2) {};
	\node [wb] at (8,2) {};
	\node [wb] at (7,2) {};
	\node [wb] at (6,2) {};
	\node [bb] at (5,2) {};
	\node [wb] at (4,2) {};
	\node [bb] at (3,2) {};
	\node [wb] at (2,2) {};
	\node [wb] at (1,2) {};
	\node [bb] at (0,2) {};
	\node [bb] at (-1,2) {};
	\node [bb] at (-2,2) {};
	\node [bb] at (-3,2) {};
	\node [bb] at (-4,2) {};
	\node [bb] at (-5,2) {};
	\node [bb] at (-6,2) {};
	\node [bb] at (-7,2) {};
	\node [bb] at (-8,2) {};
	\node [bb] at (-9,2) {};

\draw[dashed](0.5,-0.5)--node[]{}(0.5,2.5);
\end{tikzpicture}
\end{center}

The operation set from $L_\bs(\blam)$ to $L_\bv(\bmu)$ is
$$\mathcal{F} = \{[(2, -2), 4], [(1, 1), 3], [(2, 1), 3], [(3, 1), 3]\},$$
and consequently, $\mathcal{M}=(1, 2, 1)$.
\end{example}

The following lemma reveals that the moving vector from abacus $L_\bs(\blam)$ to another one can reflect
variations of multicharge $\bs$ and the shape of $L_\bs(\blam)$.
It is a simple corollary of Definition \ref{move vector definition} and
Lemma \ref{charge minus}. We omit the proof and leave it as an exercise.
Recall Definition \ref{nxk} of $\mathfrak{n}_x^k$.

\begin{lemma}\label{move vector minus}
Let $\mathcal{M}=(m_1, m_2, \dots, m_r)$ be the moving vector from $L_\bs(\blam)$
to $L_\bu(\bmu)$ and let $h$ be an integer such that all positions $\{(x, y)\mid 1\leq x \leq r, y\leq h\}$ in both $L_\bs(\blam)$
and $L_\bu(\bmu)$ have a bead placed. Then for $1\le i\le r$, we have
$\mathfrak{n}_i^h(L_{\bs}(\blam))-\mathfrak{n}_i^h(L_{\bu}(\bmu))= m_i-m_{i-1}=s_i-u_i$ (the meaning of $m_0$ is $m_r$).
\end{lemma}

\begin{lemma}\label{same block}
Let $\bs\in \overline{\mathcal A}^r_e$ and $\bu, \bv\in \Z^r$.
If the moving vector from $L_\bs(\blam)$ to $L_{\bv}(\bnu)$ is equal to that from $L_{\bu}(\bmu)$ to $L_{\bv}(\bnu)$,
then $\bs=\bu$ and $(\blam, \bs)$ and $(\bmu, \bs)$ belong to the same block.
\end{lemma}

\begin{proof}
Let $(m_1, \dots, m_r)$ be the moving vector from $L_\bs(\blam)$ to $L_\bv(\bnu)$ with $m_1+\cdots+m_r=l$.
Then by Lemma \ref{move vector minus} $s_i-v_i=m_i-m_{i-1}=u_i-v_i$
for $1\le i\le r$, in which $m_0$ means $m_r$. This implies $\bs=\bu$.

Let $\tau_{e, \bs}(\blam)=\lambda, \,\tau_{e, \bs}(\bmu)=
\mu, \,\tau_{e, \bv}(\bnu)=\nu$ and $s=\sum^r_{i=1}s_i$, $v=\sum^r_i v_i$.
If we need $k$ elementary operations from $\mathcal L^e_v(\nu)$ to its $e$-core,
then by Lemma \ref{movecorrespondence} we can obtain this $e$-core from both $\mathcal L^e_s(\lambda)$ and
$\mathcal L^e_s(\mu)$ by $k+l$ elementary operations, respectively.
Therefore, $(\lambda, s)$ and $(\mu, s)$ belong to the same block.
This is equivalent to $\mathcal C_{e, s}(\tau_{e, \bs}(\blam))=\mathcal C_{e, s}(\tau_{e, \bs}(\bmu))$.
By \cite[Corollary 2.23]{JL}, $\mathcal C_{e, \bs}(\blam)=\mathcal C_{e, \bs}(\bmu)$.
We deduce from Lemma \ref{residueblock} that $(\blam, \bs)$ and $(\bmu, \bs)$ belong to the same block.
\end{proof}

The significance of a moving vector is that it induces a new block invariant,
which may play a more important role than the weight. Let us prove a lemma for giving the definition.

\begin{lemma}\label{3.4.8}
Let $(\blam, \bs)\in \mathcal H^\bLambda_\bbeta$ with $\bs \in \overline{\mathcal A}^r_e$
and $L_{\bs^*}(\blam^*)$ its core. If $(\bmu, \bs)\in \mathcal H^\bLambda_\bbeta$,
then its core is also $L_{\bs^*}(\blam^*)$
and the moving vector from $L_\bs(\bmu)$  to $L_{\bs^*}(\blam^*)$
is equal to that from $L_\bs(\blam)$ to $L_{\bs^*}(\blam^*)$.
\end{lemma}

\begin{proof}
Since $(\blam, \bs)$ and  $(\bmu, \bs)$ are in the same block, we have from  \cite[Corollary 2.23]{JL} 
that $(\lam, s)$ and  $(\mu, s)$ are in the same block and therefore $(\lam, s)$ and  $(\mu, s)$ 
have the same $e$-core $(\lam^\ast, s^\ast)$. Note that Uglov map is injective. 
This implies that $(\lam^\ast, s^\ast)$ has a unique preimage, which is a reduced $(e, \bs)$-core by Lemma \ref{movecorrespondence}. 
Consequently, the preimage has to be $L_{\bs^*}(\blam^*)$.

Suppose that the weight of $\mathcal H^\bLambda_\bbeta$ is $w$, the moving vector
from $L_\bs(\blam)$ to $L_{\bs^*}(\blam^*)$ is $(m_1, \dots, m_r)$,
and the moving vector from $L_\bs(\bmu)$ to $L_{\bs^*}(\blam^*)$ is $(m'_1, \dots, m'_r)$.
Then Lemma \ref{move vector minus} implies $m_i-m_{i-1}=
s_i-s^\ast_i=m'_i-m'_{i-1}$ for $1\le i\le r$, in which $m_0$ means $m_r$.
Then the result follows from $m_1+\cdots+m_r=w=m'_1+\cdots+m'_r$.
\end{proof}

Given a pair $(\blam, \bs)\in \mathcal H^\bLambda_\bbeta$ with $\bs \in \overline{\mathcal A}^r_e$,
in the light of Lemma \ref{3.4.8}, we also say $L_{\bs^*}(\blam^*)$ is the core of block $\mathcal H^\bLambda_\bbeta$.
Note that the core of a block depends on the choice of $\bs$, or we can say the core of a block only if $\bs$ is fixed.

Fix a dominant weight $\bLambda$. According to the definition of the action of affine Weyl group $W$, it is not difficult to check that in each $W$-orbit of blocks,
there exists $\bbeta\in Q_+$ such that $(\balpha_j, \bLambda-\bbeta)\geq 0$ for all $0\leq j \leq e-1$.

\begin{lemma}\label{empty core blocks}
Let $\bs,\, \bs^*\in\overline{\mathcal A}^r_e$ and $\blam$ an $r$-partition. Let $L_{\bs^*}(\blam^*)$ be the core of block $\mathcal H^\bLambda_\bbeta$. Then $\blam^*=\bvarnothing$ if and only if $(\balpha_j, \bLambda-\bbeta)\geq 0$ for all $0\leq j \leq e-1$.
\end{lemma}

\begin{proof}
If $\blam^*=\bvarnothing$, then $\mathfrak{c}_l(L_{\bs^*}(\blam^*))\leq \mathfrak{c}_k(L_{\bs^*}(\blam^*))$ for arbitrary $l, k\in \Z$ with $l\geq k$ since $\bs^*\in\overline{\mathcal A}^r_e$. As a result, $\mathfrak{m}^{j-1}_j\geq 0$ for all $0\leq j \leq e-1$. This implies that $(\balpha_j, \bLambda-\bbeta)\geq 0$ for all $0\leq j \leq e-1$ by Lemma \ref{new4.3.1}. 

Now assume that $(\balpha_j, \bLambda-\bbeta)\geq 0$ for all $0\leq j \leq e-1$. Note that $L_{\bs^*}(\blam^*)$ is complete. Then it is sufficient to prove $\mathfrak{c}_l\leq \mathfrak{c}_{l-1}$ for arbitrary $l\in\Z$ to prove $\blam^*=\bvarnothing$. This is an easy corollary of Lemmas \ref{completeabaciproperty} and \ref{new4.3.1}.
\end{proof}

\begin{definition}\label{block moving vector}
Let $(\blam, \bs)\in\mathcal H^\bLambda_\bbeta$  with $\bs\in\overline{\mathcal A}^r_e$. Then the block moving vector of  $\mathcal H^\bLambda_\bbeta$  is defined to be the moving vector from $L_\bs(\blam)$ to its core.
\end{definition}

By Lemma \ref{3.2.6} the block moving vector is an invariant of an orbit of blocks.

\begin{lemma}\label{weight moving vector}
Assume that the block moving vector of block $\mathcal H^\bLambda_\bbeta$ is $(m_1, m_2, \dots, m_r)$. Then $\sum_im_i=w(\mathcal H^\bLambda_\bbeta)$.
\end{lemma}

\begin{proof}
Let $L_{\bs}(\blam)\in\mathcal H^\bLambda_\bbeta$ with core $L_{\bs^*}(\blam^*)$, 
and $L_{s}(\lam)$ and $L_{s^*}(\lam^*)$ the images under Uglov map, respectively.
Then by Lemma \ref{movecorrespondence} there is a bijection between the operation set $\mathcal{F}$ from $L_{\bs}(\blam)$ to $L_{\bs^*}(\blam^*)$ 
and the operation set $\tau_{\mathcal{F}}$ from $L_{s}(\lam)$ to $L_{s^*}(\lam^*)$. It is well-known that the number of operations in $\tau_{\mathcal{F}}$
is equal to the weight of $L_{s}(\lam)$. Moreover, we have from \cite[Theorem 3.7]{JL} that the weight of $L_{\bs}(\blam)$ is equal to that of $L_{s}(\lam)$. 
That is, the weight of $L_{\bs}(\blam)$ is equal to the number of operations in set $\mathcal{F}$, and the lemma follows.
\end{proof}

%%%%%%%%%%%%%%%%%%%%%%%%%%%%%%%%%%%%%%%%%%%%%%%%%%%%%%%%%%%%%%%%%%%%%%%%%%%%%%%%%%%%%%%%%%%%%%%%%%%%%%%%%%%%%%%%%%%%%%%%%%%%%%%%%%%%%%%%%%%
%%%%%%%%%%%%%%%%%%%%%%%%%%%%%%%%%%%%%%%%%%%%%%%%%%%%%%%%%%%%%%%%%%%%%%%%%%%%%%%%%%%%%%%%%%%%%%%%%%%%%%%%%%%%%%%%%%%%%%%%%%%%%%%%%%%%%%%%%%%

\subsection{Simple properties}
In this subsection, we study some simple properties of operation sets and moving vectors from three aspects as follows.

Firstly, we study the moving vector related to a simple deformation of abaci.
Given an abacus $L_{\bs}(\blam)$ in $\mathcal{H}^{\bLambda}_{\bbeta}$ with $\bs\in \overline{\mathcal A}^r_e$ and $e<\infty$,
denote by $L_{\bu} (\bmu)$ the abacus obtained by deleting the first $i$ rows
and putting $L_{s_x+e}(\blam^{(x)})$, $1\leq x\leq i$ on the top in $L_{\bs}(\blam)$.
Let $\mathcal M=(m_1, \dots, m_r)$ and $\mathcal M'=(m'_1, \dots, m'_r)$
be the moving vectors and $\mathcal F$ and $\mathcal F'$ the operation sets
from $L_\bs(\blam)$ and $L_\bu(\bmu)$ to their cores, respectively.
Then we have some easy results as follows.
\begin{lemma}\label{3.4.9}
The followings hold.
\begin{enumerate}
\item[{\rm (1)}] $m'_j=m_{i+j}$ for $1\le j \le r-i$.
\item[{\rm (2)}] $m'_j=m_{i+j-r}$ for $r-i+1\le j\le r$.
\item[{\rm (3)}] $[(j, h), *]\in \mathcal F'$ if and only if $[(i+j, h), *]\in \mathcal F$ for $1\le j\le r$.
\item[{\rm (4)}] $L_{\bu^*}(\bmu^*)$ can be obtained from $L_{\bs^*}(\blam^*)$ by deleting the the first $i$ rows
and putting $L_{s^\ast_x+e}(\blam^{\ast(x)})$, $1\leq x\leq i$ in the original order on the top in $L_{\bs^\ast}(\blam)$.
\item[{\rm (5)}] $\bu\in \overline{\mathcal A}^r_e$ and $(\bmu, \bu)\in\mathcal{H}^{\bLambda}_{\bbeta}$.
\end{enumerate}
\end{lemma}

Secondly, we consider the composition of certain elementary operations.
The following lemma is a generalization of \cite[Remark 4.3 (1)]{JL}. It is easy and we omit the proof.

\begin{lemma}\label{operation composition}
Assume that in $L_\bs(\blam)$ there is an empty position $(j, l-ke)$ before the bead $(i, l)$,
where $1\le i, j\le r$, $l\in \Z$, $k\in \mathbb N$, and $k=0$ when $e=\infty$.
Move in $L_\bs(\blam)$ the bead at position $(i, l)$ to position
$(j, l-ke)$ and denote by $L_\bu(\bmu)$ the new abacus.
Then $L_\bu(\bmu)$ can be obtained from $L_\bs(\blam)$ by elementary operations
$\{[(i, l), *],\, \cdots,\, [(r, l), *], \,[(1, l-e), *], \,\cdots, \,[(r, l-e), *],
\,\dots,\, [(r, l-ke+e), *],\, [(1, l-ke), *], \,\dots, \,[(j, l-ke), *]\}$.
\end{lemma}

The following two lemmas are direct corollaries of Lemma \ref{operation composition}.
We list them below without proofs.

\begin{lemma}\label{3.4.11}
Let $\mathcal F$ be the operation set from $L_{\bs}(\blam)$ to its core $(\blam^*, \bs^*)$.
If in $L_\bs(\blam)$, position $(i, h)$ has a bead, and
position  $(j, h-ke)$ is empty and before $(i, h)$, where $k\in\mathbb{N}$  and $k=0$ if $e=\infty$.
Then the operations moving the bead at position $(i, h)$ to $(j, h-ke)$ are contained in $\mathcal F$.
\end{lemma}

\begin{lemma}\label{3.4.12}
Let $L_\bs(\blam)$ be an abacus with position $(i, l)$ having a bead
and position $(j, l)$ being empty, where $1\le i<j\le r$.
Let $L_\bu(\bmu)$ be the abacus obtained by moving in $L_\bs(\blam)$
the bead at position $(i, l)$ to position $(j, l)$.
Then $L_\bu(\bmu)$ can be obtained from $L_\bs(\blam)$ by elementary operations
with moving vector $\mathcal M=(m_1, \dots, m_r)$, where $m_t=1$ if $i\le t\le j-1$ and $m_t=0$ otherwise.
\end{lemma}

Let us give some more details of a special case of Lemma \ref{operation composition}.

\begin{lemma}\label{deleting rim}
Let $L_\bs(\blam)$ be an abacus.
\begin{enumerate}
\item[{\rm (1)}] Assume that position $(i, l)$ is empty and position $(i, l+e)$ has a bead.
Denote by $L_\bu(\bmu)$ the abacus obtained from $L_\bs(\blam)$ by moving the bead at position $(i, l+e)$ to position $(i, l)$.
Then $\bu=\bs$ and the moving vector from $L_\bs(\blam)$ to $L_\bu(\bmu)$ is $(1, \dots, 1)$.
Moreover, $[\bmu]$ can be obtained by deleting a rim $e$-hook from $[\blam^{(i)}]$.
\item[{\rm (2)}] If $[\bmu]$ can be obtained by deleting a rim $e$-hook from $[\blam^{(i)}]$, then
there exists $l\in\Z$ such that in $L_\bs(\blam)$, position $(i, l)$ is empty and position $(i, l+e)$ has a bead.
\end{enumerate}
\end{lemma}

\begin{proof}
It is an evident corollary of Lemma \ref{move vector minus}, \ref{operation composition} and
 \cite[Remark 4.3(1)]{JL}.
\end{proof}

Finally, we introduce the so-called dual abacus, which in fact appeared in the work of Uglov \cite{U}.
\begin{definition}\label{dual definition}
Given an abacus $L_\bs(\blam)$, the dual of $L_\bs(\blam)$, denoted by $L_{\bs^D}(\blam^D)$,
is an abacus, in which position $(i, h)$ is empty if and only if in $L_\bs(\blam)$,
there is a bead at position $(r-i+1, -h-1)$. Furthermore, $(\blam^D, \bs^D)$ is called
the dual pair of $(\blam, \bs)$.
\end{definition}

The following two lemmas about dual abaci are easy.

\begin{lemma}\label{3.4.16}
Let $L_{\bs^D}(\blam^D)$ be the dual of $L_{\bs}(\blam)$. Then $\blam^D=\blam'$, $s_i^D=-s_{r-i+1}$ for $1\leq i\leq r$
and consequently, if $\bs\in \overline{\mathcal A}^r_e$, then $\bs^D\in \overline{\mathcal A}^r_e$ and if $\bs\in \mathcal  A^r_e$, then $\bs^D\in \mathcal A^r_e$.
\end{lemma}

\begin{lemma}\label{3.4.17}
Let $(\blam, \bs)$ and $(\bmu, \bs)$ be two pairs. Then
\begin{enumerate}
\item[{\rm (1)}] $(\blam, \bs)$ and $(\bmu, \bs)$ belong to the same block
if and only if $(\blam^D, \bs^D)$ and $(\bmu^D, \bs^D)$ belong to the same block.
\item[{\rm (2)}] $L_{\bs}(\blam) \parallel L_{\bs}(\bmu)$ if and only if $L_{\bs^D}(\blam^D)\parallel L_{\bs^D}(\bmu^D)$.
\end{enumerate}
\end{lemma}

We end this subsection by a lemma as follows.

\begin{lemma}\label{dual operator}
Let $\mathcal F$, $\mathcal M$ and $\mathcal F'$, $\mathcal M'$ be the operation sets and moving vectors
from $L_{\bs}(\blam)$ and $L_{\bs^D}(\blam^D)$ to their cores, respectively.
Then $[(i, h), *]\in \mathcal F$ if and only if $[(r-i, -h-1), *]\in \mathcal F'$.
Moreover, $m_i=m'_{r-i+1}$ for all $1\le i\le r$.
\end{lemma}

\begin{proof}
We first consider operation sets. It is clear by Definition \ref{dual definition} that we only need to prove
if $[(i, h), *]\in \mathcal F$, then $[(r-i, -h-1), *]\in \mathcal F'$.
In fact, take an element $[(i, h), *]\in \mathcal F$ such that
in $L_\bs(\blam)$, position is empty and position $(i, h)$ has a bead.
Then in $L_{\bs^D}(\blam^D)$, position $(r-i, -h-1)$ has a bead and position $(r-i+1, -h-1)$ is empty.
This implies $[(r-i, -h-1), *]\in \mathcal F'$. By moving the bead at position $(i, h)$ to position $(i+1, h)$ in $L_\bs(\blam)$,
and moving the bead at position $(r-i, -h-1)$ to position $(r-i+1, -h-1)$ in $L_{\bs^D}(\blam^D)$, we clearly get two dual abaci.
Since $\mathcal F$ is a finite set, by repeating this process finite many times,
abacus $L_{\bs}(\blam)$ arrives to its core.
As a direct corollary, we can obtain
$m_i=m'_{r-i+1}$ for all $1\le i\le r$.
\end{proof}

%%%%%%%%%%%%%%%%%%%%%%%%%%%%%%%%%%%%%%%%%%%%%%%%%%%%%%%%%%%%%%%%%%%%%%%%%%%%%%%%%%%%%%%%%%%%%%%%%%%%%%%%%%%%%%%%%%%%%%%%%%%%%%%%%%%%
%%%%%%%%%%%%%%%%%%%%%%%%%%%%%%%%%%%%%%%%%%%%%%%%%%%%%%%%%%%%%%%%%%%%%%%%%%%%%%%%%%%%%%%%%%%%%%%%%%%%%%%%%%%%%%%%%%%%%%%%%%%%%%%%%%%%

\subsection{A construction} 
We will construct an abacus with block moving vector being $\mathcal{M}$ whenever $\mathcal{M}$ satisfies certain conditions. 
This construction will be used in the whole series of papers. Let us consider a special case first.

\begin{lemma}\label{specialconstruction}
Let $r\geq 2$ and $\bs, \bs^\ast\in \overline{\mathcal A}^r_e$. If $\mathcal{M}=(m_1, \dots, m_{r-1}, 0)\in \mathbb{N}^r$ satisfies $s_i^\ast=s_i-m_i+m_{i-1}$ for $1\leq i\leq r$, then there exists an $r$-partition $\blam$ such that the moving vector from $L_\bs(\blam)$ to $L_{\bs^{\ast}}(\bvarnothing)$ is $\mathcal{M}$.
\end{lemma}

\begin{proof}
For later use let us first point out a simple fact: $\sum s_i=\sum s_i^\ast$.

We shall prove the lemma by induction. Let us verify the basic step at $r=2$. By conditions $s_i^\ast=s_i-m_i+m_{i-1}$ and $m_2=0$, we get $s_2^\ast-s_1^\ast=s_2-s_1+2m_1$. Note that $\bs\in \overline{\mathcal A}^r_e$ implies that $s_2\geq s_1$. Therefore, $s_2^\ast-s_1^\ast\geq 2m_1$. Move in $L_{\bs^\ast}(\bvarnothing)$ the beads $\CIRCLE_1^2, \dots, \CIRCLE_{m_1}^2$ to positions under them in row 1, respectively. Denote by $L_\bu(\blam)$ the new abacus. It is not difficult to check that $u_1=s_1^\ast+m_1=s_1$ and  $u_2=s_2^\ast-m_1=s_2$, that is, $\bu=\bs$. The moving vector from $L_\bs(\blam)$ to $L_{\bs^\ast}(\bvarnothing)$ is $(m_1, 0)$ and the basic step is proved. 
 
Suppose inductively that the lemma holds for $r=k-1$. Since $\sum_{i=1}^k s_i=\sum_{i=1}^k s_i^\ast$, there exists $1\leq j\leq k$ such that $s_1\leq s_j^\ast$. Assume that $j$ is the least one. We consider two cases according to the value of $j$.

If $j=1$, then $m_1$ is forced to be zero. Denote by $\underline{\bs}=(s_2, \dots, s_k)$ and $\underline{\bs}^\ast=(s_2^\ast, \dots, s_k^\ast)$. By the inductive hypothesis there exists a $k-1$-partition $\bnu$ such that the moving vector from $L_{\underline{\bs}}(\bnu)$ to $L_{\underline{\bs}^\ast}(\bvarnothing)$ is $\underline{\mathcal{M}}=(m_2, \dots, m_{k-1}, 0)$. Consequently, the moving vector from $L_{\bs}(\varnothing, \bnu)$ to $L_{\bs^\ast}(\bvarnothing)$ is $\mathcal{M}$. 

If $j>1$, by the definition of $j$ we know that $s_i^\ast<s_1$ for all $1\leq i<j$. This yields that $s_{j-1}^\ast-s_1^\ast<m_1$. On the other hand, it follows from $s_1\leq s_j^\ast$, or $s_1-s_{j-1}^\ast\leq s_j^\ast-s_{j-1}^\ast$ that is, $m_1-s_{j-1}^\ast+s_{1}^\ast\leq s_j^\ast-s_{j-1}^\ast$. This makes the following practicable. Move in $L_{\bs^\ast}(\bvarnothing)$ the beads $\CIRCLE_1^i, \dots, \CIRCLE_{s_{i}^\ast-s_{i-1}^\ast}^i$ for all $2\leq i< j$ and $\CIRCLE_1^j, \dots, \CIRCLE_{m_1-s_{j-1}^\ast+s_1^\ast}^j$ to positions under them in row 1, respectively. Denote the new abacus by $L_{\bu}(\bmu)$. Clearly, $\underline{\bmu}=(\mu^{(2)}, \dots, \mu^{(k)})=\bvarnothing$, and the moving vector from $L_{\bu}(\bmu)$ to $L_{\bs^\ast}(\bvarnothing)$ is $\mathcal{M}'=(m_1', \dots, m_k')$, where
$$m_i'=\begin{cases}
m_1-(s_i^\ast-s_1^\ast), & \text{if}\,\, 1\leq i<j;\\
0, & \text{if}\,\, j\leq i\leq k.
\end{cases}
$$

We claim that $m_i'\leq m_i$ for all $1\leq i\leq k$. Clearly, $m_1=m_1'$ and $0=m_i'\leq m_i$ for $j\leq i\leq k$. Moreover, for $2\leq i<j$ we have $m_i-m_i'=m_i-m_1+s_i^\ast-s_1^\ast=s_i-s_1+m_{i-1}.$ Note that $\bs\in\overline{\mathcal A}^k_e$ and thus $s_1\leq s_i$. This completes the proof of the claim about $\mathcal{M'}$. 

Let us study $\bu$ now. We shall prove
$$u_i=\begin{cases}
s_{1}, & \text{if}\,\, i=1;\\
s_{i-1}^\ast, & \text{if}\,\, 2\leq i<j;\\
s_j^\ast+s_{j-1}^\ast-m_1-s_1^\ast, & \text{if}\,\, i=j;\\
s_i^\ast , & \text{if}\,\, j<i\leq k.
\end{cases}
$$
In fact, by Lemma \ref{move vector minus} we have $u_i=s_i^\ast+m_i'-m_{i-1}'$ for $1\leq i\leq k$. Then $u_1=s_1^\ast+m_1'=s_1-m_1+m_1'=s_1$ and since $m_j'=0$, we have
$u_j=s_j^\ast-m_1+s_{j-1}^\ast-s_1^\ast.$  
For $2\leq i<j$,
$$u_i=s_i^\ast+m_1-(s_i^\ast-s_1^\ast)-(m_1-(s_{i-1}^\ast-s_1^\ast))=s_{i-1}^\ast.$$
It is clear that $u_i=s_i^\ast$ for $j<i\leq k$. So far, all cases have been considered.

Based on the above result we can prove $\underline{\bu}\in\overline{\mathcal A}^{k-1}_e$ now. In fact,
because $s_j^\ast-m_1-s_1^\ast=s_j^\ast-s_1\geq 0$, we get $u_j\geq s_{j-1}^\ast$ and consequently
$u_j\geq s_{j-2}^\ast=u_{j-1}$.
Note that $s_{j-1}^\ast-m_1-s_1^\ast=s_{j-1}^\ast-s_1 < 0$. This implies that $u_j<s_j^\ast$. If $j=k$, then $u_k< s_1^\ast+e=u_2+e$ because $s_j^\ast\leq s_1^\ast+e$, and $\underline{\bu}\in\overline{\mathcal A}^{k-1}_e$ follows. If $j<k$, then $u_j<s_{j+1}^\ast=u_{j+1}$.
Moreover, $u_k=s_k^\ast\leq s_1^\ast+e= u_2+e$. We obtain $\underline{\bu}\in\overline{\mathcal A}^{k-1}_e$ too.

\smallskip

Define $\mathcal{M}''=\mathcal{M}-\mathcal{M}'$. Then $s_i^\ast+m_i'-m_{i-1}'=s_i-m_i''+m_{i-1}''$ for $1\leq i\leq k$. As a result, $u_i=s_i-m_i''+m_{i-1}''$ for $1\leq i\leq k$.
Write $(m_2'', \dots, m_{k-1}'', 0)$ as $\underline{\mathcal{M}}''$. Then by the inductive hypothesis there exists a $k-1$-partition $\bnu$ such that the moving vector from $L_{\underline{\bs}}(\bnu)$ to $L_{\underline{\bu}}(\bvarnothing)$ is $\underline{\mathcal{M}}''$. Note that $s_1=u_1$, $m_1=m_1'$ and $\underline{\bmu}=\bvarnothing$. Define $\blam=(\mu^{(1)}, \bnu)$. We can derive that the moving vector from $L_{\bs}(\blam)$ to $L_{\bu}(\bmu)$ is $\mathcal{M}''$. This implies that the moving vector from $L_{\bs}(\blam)$ to $L_{\bs^\ast}(\bvarnothing)$ is $\mathcal{M}$.
\end{proof}

By using Lemma \ref{specialconstruction} we can obtain the general construction.

\begin{lemma}\label{generalconstruction}
Let $r\geq 2$ and $\bs, \bs^\ast\in \overline{\mathcal A}^r_e$. If $\mathcal{M}=(m_1, \dots, m_{r})\in \mathbb{N}^r$ satisfies $s_i^\ast=s_i-m_i+m_{i-1}$ for $1\leq i\leq r$, then there exists an $r$-partition $\blam$ such that the moving vector from $L_\bs(\blam)$ to $L_{\bs^{\ast}}(\bvarnothing)$ is $\mathcal{M}$.
\end{lemma}

\begin{proof}
Suppose that $m_i$ is a minimal in $\{m_1, \dots, m_r\}$. 
Define $$\bs'=(s_1', \dots, s_r')=(s_{i+1}-e, \dots, s_r-e,\, s_1, \dots, s_i),$$
$$\bs^{\ast '}=(s_1^{\ast '}, \dots, s_r^{\ast '})=(s_{i+1}^{\ast}-e, \dots, s_r^{\ast}-e,\, s_1^{\ast}, \dots, s_i^{\ast})$$ and 
$$\mathcal{M'}=(m_1', \dots, m_r')=(m_{i+1}-m_i, \dots, m_r-m_i,\, m_1-m_i, \dots, m_{i-1}-m_i,\, 0).$$
Clearly, $\bs', \bs^{\ast '}\in \overline{\mathcal A}^r_e$ and $s_x^{\ast '}=s_x'-m_x'+m_{x-1}'$ for $1\leq x\leq r$.
Then by Lemma \ref{specialconstruction} there exists an $r$-partition $\bmu$ such that the moving vector from $L_{\bs'}(\bmu)$ to $L_{\bs^{\ast '}}(\bvarnothing)$ is $\mathcal{M'}$. Assume that the bead $\CIRCLE_1^1$ in $L_{\bs'}(\bmu)$ is in column $h$.
Move in $L_{\bs'}(\bmu)$ the bead to position $(1, h+m_ie)$ and denote the new abacus by $L_{\bs'}(\overline{\blam})$.
Let $\blam=(\overline{\blam}^{(r-i+1)}, \dots, \,\overline{\blam}^{(r)}, \,\overline{\blam}^{(1)}, \dots, \,\overline{\blam}^{(r-i)})$. Then the moving vector from $L_\bs(\blam)$ to $L_{\bs^{\ast}}(\bvarnothing)$ is $\mathcal{M}$.
\end{proof}

\begin{remarks}
In fact, Lemma \ref{generalconstruction} still holds if we replace $\bvarnothing$ by an $r$-partition $\blam^{\ast}$ with $L_{\bs^{\ast}}(\blam^{\ast})$ being complete.
\end{remarks}

%\smallskip
%%%%%%%%%%%%%%%%%%%%%%%%%%%%%%%%%%%%%%%%%%%%%%%%%%%%%%%%%%%%%%%%%%%%%%%%%%%%%%%%%%%%%%%%%%%%%%%%%%%%%%%%%%%%%%%%%%%%%%%%%%%%%%%%%%%%
%%%%%%%%%%%%%%%%%%%%%%%%%%%%%%%%%%%%%%%%%%%%%%%%%%%%%%%%%%%%%%%%%%%%%%%%%%%%%%%%%%%%%%%%%%%%%%%%%%%%%%%%%%%%%%%%%%%%%%%%%%%%%%%%%%%%

\subsection{Incomparable abaci and moving vectors}
Based on Section 4.3, this subsection is devoted to provide some circumstances, in which one can construct incomparable abaci. We always assume that $r>2$.

\begin{lemma}\label{2 runners 2 columns}
Let $(\blam, \bs)\in \mathcal{H}^{\bLambda}_{\bbeta}$ with $\bs\in \overline{\mathcal A}^r_e$.
If there exist $1\leq j\leq r$, $h_1,\, h_2\in\mathbb{Z}$ with $h_1\ne h_2$ such that in $L_{\bs}(\blam)$,
\begin{enumerate}
\item[{\rm (1)}] there is a bead at position $(j, h_1)$ and position $(j+1, h_1)$ is empty;
\item[{\rm (2)}] there is a bead at position $(j, h_2)$ and position $(j+1, h_2)$ is empty,
\end{enumerate}
then there exist pairs $(\bmu, \bs),\, (\bnu, \bs)\in \mathcal{H}^{\bLambda}_{\bbeta}$
with $L_\bs(\bmu)\parallel L_\bs(\bnu)$.
\end{lemma}

\begin{proof}
Move the beads at positions $(j, h_1)$ and $(j, h_2)$ in $L_\bs(\blam)$
to positions $(j+1, h_1)$ and $(j+1, h_2)$, respectively.
Denote by $L_{\bar{\bs}}(\bar{\blam})$ the new abacus obtained. Then the moving vector from
$L_\bs(\blam)$ to $L_{\bar{\bs}}(\bar{\blam})$ is $\mathcal{M}=(m_1, m_2, \cdots, m_r)$, where
$m_j=2$, and 0 otherwise.
It follows from Lemma \ref{abacus simple property} (2) that there exist $h_3, h_4\in\mathbb{Z}$
such that in $L_\bs(\blam)$ the positions $(j, h_3)$ and $(j, h_4)$ are empty
and positions $(j+1, h_3)$ and $(j+1, h_4)$ are occupied by a bead.
Then in $L_{\bar{\bs}}(\bar{\blam})$, positions
$(j, h_1), (j, h_2), (j, h_3)$ and $(j, h_4)$ are empty and there are beads at positions
$(j+1, h_1)$, $(j+1, h_2)$, $(j+1, h_3)$ and $(j+1, h_4)$.
Define $\{l_1, l_2, l_3, l_4\}$ to be equal to $\{h_1, h_2, h_3, h_4\}$ as a set satisfying $l_1<l_2<l_3<l_4$.

Move the beads at positions
$(j+1, l_1)$ and $(j+1, l_4)$ to positions
$(j, l_1)$ and $(j, l_4)$ in $L_{\bar{\bs}}(\bar{\blam})$, respectively. Denote by $L_{\bu}(\bmu)$ the new abacus.
Move the beads at positions
$(j+1, l_2)$ and $(j+1, l_3)$ to positions
$(j, l_2)$ and $(j, l_3)$ in $L_{\bar{\bs}}(\bar{\blam})$, respectively.
Denote the new abacus by $L_{\bv}(\bnu)$.
Clearly, both the moving vectors from $L_{\bu}(\bmu)$ and $L_{\bv}(\bnu)$ to $L_{\bar{\bs}}(\bar{\blam})$ are equal to $\mathcal{M}$.
Hence we can deduce from Lemma \ref{same block} that $\bs=\bu=\bv$ and $(\bmu, \bs), (\bnu, \bs)\in \mathcal{H}^{\bLambda}_{\bbeta}$.
Moreover, take $(\kappa_1, \iota_1)=(j,\, l_4)$ and $(\kappa_2, \iota_2)=(j+1,\, l_1)$.
It is not difficult to check that the conditions $(1)$ and $(2)$ of Definition \ref{incomparable definition} are satisfied.
That is, $L_\bs(\bmu)\parallel L_\bs(\bnu)$.
\end{proof}

The following lemma is useful only if $r>3$.

\begin{lemma}\label{4 runners}
Let $(\blam, \bs)\in \mathcal{H}^{\bLambda}_{\bbeta}$ with $\bs\in \overline{\mathcal A}^r_e$,
where $r\ge 4$. Assume that in $L_{\bs}(\blam)$,
\begin{enumerate}
\item[{\rm (1)}] there is a bead at position $(i, l)$ and $(i+1, l)$ is an empty position;
\item[{\rm (2)}] there is a bead at position $(j, h)$ and $(j+1, h)$ is an empty position,
\end{enumerate}
where $l, h\in \mathbb{Z}$, $1\leq i, j\leq r$ with $i+1<j$ and $i\neq 1$ if $j=r$.
Then there exist pairs $(\bmu, \bs), (\bnu, \bs) \in \mathcal{H}^{\bLambda}_{\bbeta}$ with
 $L_\bs(\bmu)\parallel L_\bs(\bnu)$.
\end{lemma}

\begin{proof}
Move the beads at positions $(i, l)$ and $(j, h)$ in $L_{\bs}(\blam)$
to positions $(i+1, l)$ and $(j+1, h)$, respectively.
Denote by $L_{\bar{\bs}}(\bar{\blam})$ the new abacus. Then the moving vector from
$L_\bs(\blam)$ to $L_{\bar{\bs}}(\bar{\blam})$ is $\mathcal{M}=(m_1, m_2, \dots, m_r)$, where
$m_i=m_j=1$, and
0 otherwise.

It follows from Lemma \ref{abacus simple property} (2) that there exist $l', h'\,\in \mathbb{Z}$ such
that in $L_\bs(\blam)$ the positions $(i, l')$ and $(j, h')$ are empty
and there are beads at positions $(i+1, l')$ and $(j+1, h')$.
Then in $L_{\bar{\bs}}(\bar{\blam})$, positions
$(i, l)$, $(i, l')$, $(j, h)$ and $(j, h')$ are empty and there are beads at positions
$(i+1, l)$, $(i+1, l')$, $(j+1, h)$ and $(j+1, h')$.
Define $\{h_1, h_2\}$ to be equal to $\{h, h'\}$ as a set satisfying $h_1<h_2$.
Define $\{l_1, l_2\}$ to be equal to $\{l, l'\}$ as a set satisfying $l_1<l_2$.

Denote by $L_{\bu}(\bmu)$ the abacus obtained by moving the beads
at positions $(i+1, l_2)$ and $(j+1, h_1)$  in $L_{\bar{\bs}}(\bar{\blam})$
to positions $(i, l_2)$ and $(j, h_1)$, respectively
and denote by $L_{\bv}(\bnu)$ the abacus obtained by moving the beads
at positions $(i+1, l_1)$ and $(j+1, h_2)$  in $L_{\bar{\bs}}(\bar{\blam})$
to positions $(i, l_1)$ and $(j, h_2)$, respectively.
Clearly, both the moving vectors from $L_{\bu}(\bmu)$ and $L_{\bv}(\bnu)$ to $L_{\bar{\bs}}(\bar{\blam})$ are equal to $\mathcal{M}$.
Hence we can deduce from Lemma \ref{same block} that $\bs=\bu=\bv$ and $(\bmu, \bs), (\bnu, \bs)\in \mathcal{H}^{\bLambda}_{\bbeta}$.
Take $(\kappa_1, \iota_1)=(i, l_2)$ and $(\kappa_2, \iota_2)=(j+1, h_1)$, then $L_\bs(\bmu)$ and $L_\bs(\bnu)$ satisfy
 the conditions $(1)$ and $(2)$ of Definition \ref{incomparable definition}  and consequently, $L_\bs(\bmu)\parallel L_\bs(\bnu)$.
\end{proof}

\begin{lemma}\label{3 runners 3 columns}
Let $(\blam, \bs)\in \mathcal{H}^{\bLambda}_{\bbeta}$ and $\bs \in \overline{\mathcal A}^r_e$.
Assume there exists $1\leq i\leq r$ such that in $L_\bs(\blam)$,
\begin{enumerate}
\item[{\rm (1)}] there is a bead at position $(i, l_1)$ and position $(i+1, l_1)$ is empty;
\item[{\rm (2)}] position $(i, l_2)$ is empty and there is a bead at position $(i+1, l_2)$;
\item[{\rm (3)}] there is a bead at position $(i+1, l_3)$ and position $(i+2, l_3)$ is empty;
\item[{\rm (4)}] position $(i+1, l_4)$ is empty and there is a bead at position $(i+2, l_4)$ ,
\end{enumerate}
where $l_1, l_2, l_3, l_4\in \mathbb{Z}$ such that $l_1\ne l_4$ or $l_2\ne l_3$ holds.
Then there exist pairs $(\bmu, \bs), (\bnu, \bs)\in \mathcal{H}^{\bLambda}_{\bbeta}$
with $L_{\bs}(\bmu)\parallel L_\bs\bnu)$.
\end{lemma}

\begin{proof} According to the relationship among $l_1, l_2, l_3, l_4$,
we divide the proof into the following three cases.

\smallskip

{\em {\bf Case 1.} $l_1\ne l_4$ and $l_2\ne l_3$.}
In $L_\bs(\blam)$, move the bead at position $(i, l_1)$ to $(i+1, l_1)$ and
move the bead at position $(i+1, l_3)$  to $(i+2, l_3)$.
Denote by $L_{\bar{\bs}}(\bar{\blam})$ the new abacus obtained. Then the moving vector from
$L_\bs(\blam)$ to $L_{\bar{\bs}}(\bar{\blam})$ is $\mathcal{M}=(m_1, m_2, \dots, m_r)$, where
$m_k=1$ if $k=i, i+1$ and
0 otherwise.
In $L_{\bar{\bs}}(\bar{\blam})$ the positions
$(i, l_1)$, $(i, l_2)$, $(i+1, l_3)$ and $(i+1, l_4)$ are empty, and there are beads at positions
$(i+1, l_1)$, $(i+1, l_2)$, $(i+2, l_3)$ and $(i+2, l_4)$.
Define $\{h_1, h_2\}$ to be equal to $\{l_1, l_2\}$  as a set satisfying $h_1<h_2$ and
define $\{h_3, h_4\}$ to be equal to $\{l_3, l_4\}$ as a set satisfying $h_3<h_4$.

Denote by $L_{\bu}(\bmu)$ the abacus obtained by moving  in $L_{\bar{\bs}}(\bar{\blam})$ the beads
at positions $(i+1, h_2)$ and $(i+2, h_3)$ to positions $(i, h_3)$ and $(i+1, h_3)$, respectively.
Denote by $L_{\bv}(\bnu)$ the abacus obtained by moving  in $L_{\bar{\bs}}(\bar{\blam})$ the beads
at positions $(i+1, h_1)$ and $(i+2, h_4)$ to positions $(i, h_1)$ and $(i+1, h_4)$, respectively.
Clearly, both the moving vectors from $L_{\bu}(\bmu)$ and $L_{\bv}(\bnu)$
to $L_{\bar{\bs}}(\bar{\blam})$ are equal to $\mathcal{M}$.
We can deduce from Lemma \ref{same block} that $\bs=\bu=\bv$
and $(\bmu, \bs),\, (\bnu, \bs)\in \mathcal{H}^{\bLambda}_{\bbeta}$.
Furthermore, we get $L_\bs(\bmu) \parallel L_\bs(\bnu)$
by taking $(\kappa_1, \iota_1)=(i, h_2)$ and $(\kappa_2, \iota_2)=(i+2, h_3)$.

\smallskip

{\em {\bf Case 2.} $l_1=l_4$ and $l_2\neq l_3$}.  Since positions $(i+1, l_2)$ and $(i+1, l_3)$ have a bead placed
and position $(i+2, l_3)$ is empty, if position $(i+2, l_2)$ is empty,
then by Lemma \ref{2 runners 2 columns} the result follows.
Now we assume that there is a bead at position $(i+1, l_2)$.
Move in $L_\bs(\blam)$ the beads at positions $(i, l_1)$ and $(i+1, l_3)$
to positions $(i+1, l_1)$ and $(i+2, l_3)$, respectively.
Denote by $L_{\bar{\bs}}(\bar{\blam})$ the new abacus. The moving vector from
$L_\bs(\blam)$ to $L_{\bar{\bs}}(\bar{\blam})$ is $\mathcal{M}=(m_1, m_2, \dots, m_r)$, where
$m_k=1$ if $k=i, i+1$ and
0 otherwise.

Let us take a look at $L_{\bar{\bs}}(\bar{\blam})$. Positions $(i, l_1)$, $(i, l_2)$ and $(i+1, l_3)$ are empty,
and there are beads at positions $(i+1, l_1)$, $(i+1, l_2)$, $(i+2, l_1)$, $(i+2, l_2)$ and $(i+2, l_3)$.
Define $\{h_1, h_2\}$ to be equal to $\{l_1, l_2\}$ as a set satisfying $h_1<h_2$.
We divide the rest discussion into the following two subcases.

{\em Subcase 1. $l_3>h_2$ or $l_3<h_1$}.\,\,
Denote by $L_{\bu}(\bmu)$ the abacus obtained by moving in $L_{\bar{\bs}}(\bar{\blam})$
the bead at position $(i+1, h_2)$ to $(i, h_2)$ and next moving the bead at position $(i+2, h_2)$ to $(i+1, h_2)$.
Denote by $L_{\bv}(\bnu)$ the abacus obtained by moving in $L_{\bar{\bs}}(\bar{\blam})$
the beads at positions $(i+1, h_1)$ and $(i+2, l_3)$ to positions $(i, h_1)$ and $(i+1, l_3)$, respectively.
Clearly, both the moving vectors from $L_{\bu}(\bmu)$ and $L_{\bv}(\bnu)$
to $L_{\bar{\bs}}(\bar{\blam})$ are equal to $\mathcal{M}$.
Hence we can deduce from Lemma \ref{same block} that  $\bs=\bu=\bv$ and $(\bmu, \bs), (\bnu, \bs)\in \mathcal{H}^{\bLambda}_{\bbeta}$.
It is not difficult to check $L_\bs(\bmu)\parallel L_\bs(\bnu)$ by taking $(\kappa_1, \iota_1)=(i, h_2)$, $(\kappa_2, \iota_2)=(i+2, h_2)$ if
$l_3>h_2$ and by taking $(\kappa_1, \iota_1)=(i, h_2)$, $(\kappa_2, \iota_2)=(i+1, l_3)$ if $l_3<h_1$.

\smallskip

{\em Subcase 2. $h_1<l_3<h_2$}.\, We construct two incomparable abaci as follows. In $L_{\bar{\bs}}(\bar{\blam})$,
moving the bead at positions $(i+1, h_1)$ to $(i, h_1)$ and next moving the bead at position $(i+2, h_1)$ to
 $(i+1, h_1)$, gives an abacus $L_\bs(\bmu)$
and moving the beads at positions $(i+2, l_3)$ and $(i+1, h_2)$
to positions $(i+1, l_3)$ and $(i, h_2)$, respectively, gives another one, $L_\bs(\bnu)$.
Then by taking $(\kappa_1, \iota_1)=(i+1, h_2)$ and $(\kappa_2, \iota_2)=(i+2, h_1)$, we know $L_\bs(\bmu)\parallel L_\bs(\bnu)$.

\smallskip

{\em {\bf Case 3.} $l_1\ne l_4$ and $l_2=l_3$}.
Consider by Lemmas \ref{3.4.16} and \ref{3.4.17} the dual abacus of $L_{\bs}(\blam)$. It is of Case 2.
\end{proof}

\begin{corollary}\label{3 runners 3 columns(1)}
Let $\bs\in \overline{\mathcal A}^r_e$ and $(\blam, \bs)\in \mathcal{H}^{\bLambda}_{\bbeta}$.
Assume there exist $1\le i\le r$ and $h_1, h_2\in\mathbb{Z}$ such that in $L_\bs(\blam)$,
\begin{enumerate}
\item[{\rm (1)}] position $(i, h_1)$ has a bead and position $(i+1, h_1)$ is empty;
\item[{\rm (2)}] position $(i+1, h_2)$ has a bead and position $(i+2, h_2)$ is empty;
\item[{\rm (3)}] position $(i+2, h_1)$ is empty or position $(i, h_2)$ has a bead.
\end{enumerate}
Then there exist $(\bmu, \bs), (\bnu, \bs) \in \mathcal H^{\bLambda}_{\bbeta}$ with $L_\bs(\bmu) \parallel L_\bs(\bnu)$.
\end{corollary}

\begin{proof}
By Lemma \ref{abacus simple property}, there exist $h_3, h_4 \in \Z$ such that in $L_\bs(\blam)$
\begin{enumerate}
\item[$\bullet$] position $(i, h_3)$ is empty and position $(i+1, h_3)$ has a bead;
\item[$\bullet$] position $(i+1, h_4)$ is empty and position $(i+2, h_4)$ has a bead.
\end{enumerate}
Clearly, if position $(i, h_2)$ has a bead, then $h_3\ne h_2$, and if position $(i+2, h_1)$ is empty, then $h_4\ne h_1$.
By Lemma \ref{3 runners 3 columns} the proof is completed.
\end{proof}

We conclude this section by a lemma that will be used in Section 5.4.

\begin{lemma}\label{4.1.12}
Let $(\blam, \bs)\in \mathcal H^{\bLambda}_{\bbeta}$ with $\bs \in \overline{\mathcal A}^r_e$.
If there exist $1\le i_1<i_2<i_3<i_4\le r$ such that
\begin{enumerate}
\item[{\rm (1)}] position $(i_1, h)$ has a bead and position $(i_3, h)$ is empty;
\item[{\rm (2)}] position $(i_2, h)$ has a bead and position $(i_4, h)$ is empty,
\end{enumerate}
then there exist $(\bmu, \bs), (\bnu, \bs)\in \mathcal{H}^{\bLambda}_{\bbeta}$
such that $L_\bs(\bmu) \parallel L_\bs(\bnu)$.
\end{lemma}

\begin{proof}
Move in $L_\bs(\blam)$ the beads at positions $(i_1, h)$ and $(i_2, h)$
to positions $(i_3, h)$ and $(i_4, h)$, respectively.
Denote by $L_{\bar \bs}(\bar{\blam})$ the new abacus. Then the moving vector from
$L_\bs(\blam)$ to $L_{\bar \bs}(\bar{\blam})$ is $\mathcal{M}=(m_1, m_2, \dots, m_r)$, where
$$m_t=\begin{cases}
1, & \text{if}\,\, i_1\le t \le i_2-1\,\, \text{or} \,\,i_3\le t\le i_4-1 ;\\
2, & \text{if}\,\, i_2\le t \le i_3-1 ;\\
0, & \text{if}\,\, \text{others}.
\end{cases}
$$
It follows from Lemma \ref{abacus simple property} (2) that there exist $h_1\ne h, h_2\ne h$ such that in $L_\bs(\blam)$,
and consequently in $L_{\bar{\bs}}(\bar{\blam})$
\begin{enumerate}
\item[$\bullet$] position $(i_1, h_1)$ is empty and position $(i_3, h_1)$ has a bead;
\item[$\bullet$] position $(i_2, h_2)$ is empty and position $(i_4, h_2)$ has a bead.
\end{enumerate}
Moreover, in $L_{\bar{\bs}}(\bar{\blam})$,
\begin{enumerate}
\item[$\bullet$] position $(i_1, h)$ is empty and position $(i_3, h)$ has a bead;
\item[$\bullet$] position $(i_2, h)$ is empty and position $(i_4, h)$ has a bead.
\end{enumerate}

Define $\{l_1, l_3\}$ to be equal to $\{h, h_1\}$ as a set satisfying $l_1<l_3$
and define $\{l_2, l_4\}$ to be equal to $\{h, h_2\}$ as a set satisfying $l_2<l_4$.
Denote by $L_{\bs'}(\bmu)$ the abacus obtained by moving in $L_{\bar{\bs}}(\bar{\blam})$
the beads at positions $(i_3, l_3)$ and $(i_4, l_2)$ to positions $(i_1, l_3)$  and $(i_2, l_2)$, respectively  and
denote by $L_{\bs''}(\bnu)$ the abacus obtained by moving  in $L_{\bar{\bs}}(\bar{\blam})$
the beads at positions $(i_3, l_1)$ and $(i_4, l_4)$ to positions $(i_1, l_1)$ and $(i_2, l_4)$, respectively.
It is not difficult to check that both the moving vectors from $L_{\bs'}(\bmu)$
and $L_{\bs''}(\bnu)$ to $L_{\bar{\bs}}(\bar{\blam})$ are equal to $\mathcal{M}$.
Hence we can deduce from Lemma \ref{same block} that $\bs=\bs'=\bs''$
and $(\bmu, \bs), (\bnu, \bs)\in \mathcal{H}^{\bLambda}_{\bbeta}$.
Finally, it is a routine task to check that $L_\bs(\bmu)\parallel L_\bs(\bnu)$
by taking $(\kappa_1, \iota_2)=(i_1, l_3)$ and $(\kappa_2, \iota_2)=(i_4, l_2)$.
\end{proof}

\medskip
%%%%%%%%%%%%%%%%%%%%%%%%%%%%%%%%%%%%%%%%%%%%%%%%%%%%%%%%%%%%%%%%%%%%%%%%%%%%%%%%%%%%%%%%%%%%%%%%%%%%%%%%%%%%%%
%%%%%%%%%%%%%%%%%%%%%%%%%%%%%%%%%%%%%%%%%%%%%%%%%%%%%%%%%%%%%%%%%%%%%%%%%%%%%%%%%%%%%%%%%%%%%%%%%%%%%%%%%%%%%%
%%%%%%%%%%%%%%%%%%%%%%%%%%%%%%%%%%%%%%%%%%%%%%%%%%%%%%%%%%%%%%%%%%%%%%%%%%%%%%%%%%%%%%%%%%%%%%%%%%%%%%%%%%%%%%
%%%%%%%%%%%%%%%%%%%%%%%%%%%%%%%%%%%%%%%%%%%%%%%%%%%%%%%%%%%%%%%%%%%%%%%%%%%%%%%%%%%%%%%%%%%%%%%%%%%%%%%%%%%%%%

\section{Proof of Theorem}

Based on Sections 2. 3. and 4, we can now prove {\bf Theorem} in this section.

\subsection{Frame of the proof} Let $K$ be an algebraically
closed field with $char K\neq 2$
and $q\in K^\times$, $q\neq 1$ with  quantum characteristic $e$.
Let $\bs=(s_1, s_2, \dots, s_r)\in \mathbb{Z}^r$ be a multicharge.
Define $\widetilde{\bs}:=(s'_1,\dots,s'_r)\in \{0,\dots,e-1\}^r$
such that $s'_i\equiv s_i\,({\rm mod} \,e)$. Then there exists a unique $\sigma_{\bs}\in \mathfrak S_r$ such that
\begin{enumerate}
	\item[(1)]\, $s'_{\sigma_{\bs}(1)}\le s'_{\sigma_{\bs}(2)}\le \cdots \le s'_{\sigma_{\bs}(r)}$
	
	\item[(2)]\, If $s_{\sigma_{\bs}(i)}=s_{\sigma_{\bs}(i+1)}$  for $i\in \{1,\dots, r-1\}$, then $\sigma_{\bs}(i)<\sigma_{\bs}(i+1)$.
\end{enumerate}

Set $\widetilde{\bs}^{\sigma_{\bs}}:=(s'_{\sigma_{s}(1)},s'_{\sigma_{\bs}(2)},\dots,s'_{\sigma_{\bs}(r)})$. Then clearly,
$\widetilde{\bs}^{\sigma_{\bs}}\in \mathcal{A}^r_e$. Since $(\blam, \bs)$ and $(\blam^{\sigma_{\bs}}, \widetilde{\bs}^{\sigma_{\bs}})$ belong to two isomorphic blocks,
without loss of generality
we can assume $\bs\in\mathcal{A}^r_e$.
Because the representation type of blocks of an Iwahori-Hecke algebra of type A or B have been determined by
Erdmann and Nakano in \cite{EN} and by Ariki in \cite{A4}, respectively, we always assume in this paper that $r>2$.
Moreover, a weight 0 block is simple and hence has finite type. We only need to consider blocks of weight more than 0.

Given an Ariki-Koike algebra $\mathcal {H}_n(q, Q)$, let $\mathcal{S}_{n, r}(q, Q_1, Q_2, \dots, Q_r)$ 
be the associated cyclotomic $q$-Schur algebra. Take $\bnu=(\varnothing, \dots, \varnothing, (1^n))$.
Let $\varphi_\bnu\in \mathcal{S}_{n, r}$ be the identity map on $M^\bnu$ and zero on others.
Then $\varphi_\bnu$ is an idempotent of $\mathcal{S}_{n, r}$ and $\varphi_\bnu\mathcal{S}_{n, r}\varphi_\bnu$
is isomorphic to $\mathcal {H}_n(q, Q)$. If $\mathcal{B}$ is a block of $\mathcal{S}_{n, r}$, then
$\varphi_\bnu\mathcal{B}\varphi_\bnu$ is isomorphic to a block of $\mathcal {H}_n(q, Q)$.
Moreover, if $\varphi_\bnu\mathcal{B}\varphi_\bnu$ has infinite type,
then so is $\mathcal{B}$ \cite{E}. If the weight of a block of $\mathcal{S}_{n, r}$ is 1,
then by \cite[Theorem 4.12]{F1} and \cite[Proposition 1.7]{W} the block has finite representation type.
Consequently, the weight one blocks of $\mathcal{H}_n(q, Q)$ has finite type.
So we only need to handle the blocks of weight more than one.

Based on the above analysis, the proof of {\bf Theorem} is divided into three parts according
to the characteristic of the block moving vector $(m_1, m_2, \dots, m_r)$.

\medskip

{\bf Part I}.\, All $m_i$ are equal to 1, for $i=1, \dots, r$.

\smallskip

{\bf Part II}.\, All $m_i$ are less or equal to 1, and there exist at least one $m_j=0$.

\smallskip

{\bf Part III}.\, There exists some $m_i\geq 2$.

\smallskip

Note that the stable equivalence preserves representation type \cite{K} and that a derived equivalence is a stable equivalence for a symmetric algebra.
According to a classical result proved in \cite{CR} by Chuang and Rouquier, all blocks in the same orbit are derived equivalent.
This implies that all blocks in the same orbit have the same representation type. Thus we can only consider the block $\mathcal{H}^{\bLambda}_{\bbeta}$ satisfying the condition $(\balpha_j, \bLambda-\bbeta)\geq 0$ for all $0\leq j \leq e-1$ if necessary. Under this condition,
it follows from Lemmas \ref{move vector minus} and \ref{empty core blocks} that the core of $\mathcal{H}^{\bLambda}_{\bbeta}$ is $L_{\bs}(\bvarnothing)$.

%%%%%%%%%%%%%%%%%%%%%%%%%%%%%%%%%%%%%%%%%%%%%%%%%%%%%%%%%%%%%%%%%%%%%%%%%%%%%%%%%%%%%%%%%%%%%%%%%%%%%%%%%%%%%%%%%%
%%%%%%%%%%%%%%%%%%%%%%%%%%%%%%%%%%%%%%%%%%%%%%%%%%%%%%%%%%%%%%%%%%%%%%%%%%%%%%%%%%%%%%%%%%%%%%%%%%%%%%%%%%%%%%%%%%

\subsection{Proof of Part I}

Given a pair $(\blam, \bs)\in \mathcal{H}^{\bLambda}_{\bbeta}$ with $\bs\in \mathcal{A}^r_e$
and $r>2$, let $\mathcal{M}=(1, 1, \dots, 1)$ be the block moving vector from $L_\bs(\blam)$ to its core.
We prove in this subsection that $\mathcal{H}^{\bLambda}_{\bbeta}$ has infinite representation type.
Clearly, the condition on $\mathcal M$ in this section forces $e<\infty$.
Moreover, $\bs$ has to be one of the following two types.

\smallskip

Type I: There exists $1\le j<r$ such that $s_j\ne s_{j+1}$.

Type II: $s_j=s_{j+1}$ for all $1\le j<r$.

\smallskip

\noindent{\bf Proof of Type I.}
Since $\bs \in \mathcal A^r_e$, we have $s_r<s_1+e$. Therefore, in  $L_{\bs}(\bvarnothing)$ position $(1, s_1-1)$ has a bead and position $(r, s_1+e-1)$ is empty. 
Moreover, the existence of $1\le j<r$ with $s_j\ne s_{j+1}$ gives $s_1<s_r$ and thus in  $L_{\bs}(\bvarnothing)$ position $(r, s_r-1)$ has a bead and position $(1, s_r-1)$ is empty.

Now move in $L_{\bs}(\bvarnothing)$ the beads at positions $(1, s_1-1)$ and $(r, s_r-1)$ to $(r, s_1+e-1)$ and $(1, s_r-1)$, respectively. Denote the new abacus by $L_ {\bu}(\bmu)$. On the other hand, move in $L_{\bs}(\bvarnothing)$ the bead at position $(2, s_2-1)$ to $(2, s_2+e-1)$ and denote the abacus obtained by $L_ {\bv}(\bnu)$. It is not difficult to check that both the moving vectors from $L_{\bu}(\bmu)$ and $L_{\bv}(\bnu)$ to $L_{\bs}(\bvarnothing)$ are equal to $\mathcal M=(1, 1, \dots, 1)$.
We can deduce from Lemma \ref{same block} that $\bs=\bu=\bv$ and $(\bmu, \bs), (\bnu, \bs)\in \mathcal{H}^{\bLambda}_{\bbeta}$.
Moreover, take $(\kappa_1, \iota_1)=(1, s_r-1)$ and $(\kappa_2, \iota_2)=(r, s_r-1)$.
It is not difficult to check that the conditions $(1)$ and $(2)$ of Definition \ref{incomparable definition} are satisfied.
That is, $L_\bs(\blam)\parallel L_\bs(\bmu)$.

\smallskip

\noindent{\bf Proof of Type II.}
Clearly all the rows of $L_{\bs}(\bvarnothing)$ are the same in this case. Based on this observation, we can depict the frame of abaci in the block.

\begin{lemma}\label{5.2.2}
The pair $(\bmu, \bs)$ is in $\mathcal H^\bLambda_\bbeta$ if and only if
in $L_{\bs}(\bmu)$ there is an empty position $(j, l)$ with $s_1-e\le l\le s_1-1$ such that position
$(j, l+e)$ has a  bead and $L_{\bs}(\bvarnothing)$ can be obtained from $L_{\bs}(\bmu)$
by moving the  bead at position $(j, l+e)$ to position $(j, l)$.
\end{lemma}

\begin{proof}
``$\Leftarrow$" is a direct corollary of Lemmas \ref{same block} and \ref{deleting rim}.
Thus we only need to prove``$\Rightarrow$".

Since $(\bmu, \bs)\in \mathcal H^\bLambda_\bbeta$,
we have the core of  $L_\bs(\bmu)$ is $L_{\bs}(\bvarnothing)$. Note that all rows of $L_{\bs}(\bvarnothing)$ are the same.
This forces the final operation from $L_{\bs}(\bmu)$
to $L_{\bs}(\blam^*)$ is $[(r, l+e),\, \ast]$ for some $l\in\Z$.
Suppose that the final operation is done in abacus $L_\bv(\bnu)$.
Then in $L_\bv(\bnu)$, we have
\begin{enumerate}
\item[$\bullet$] position $(1, l)$ is empty and all the other positions in column $l$ have a bead placed;
\item[$\bullet$] position $(r, l+e)$ has a  bead and all the other positions in column $l+e$ are empty;
\item[$\bullet$] for each $h\ne l, l+e$, the number of beads in column $h$ is either $r$ or 0.
\end{enumerate}
Consequently, operations from $L_\bs(\bmu)$ to $L_\bv(\bnu)$ must happen in columns $l$ and $l+e$.
Furthermore, in column $l$ of $L_\bs(\bmu)$, there is only one empty position $(i, l)$
and in column $l+e$ there is only one position $(j, l+e)$ occupied by a  bead.
Note that $s_1 = \cdots = s_r$. This forces $i=j$. Moveover, each position in column $l$ of $L_{\bs}(\bvarnothing)$ has a bead 
and all positions in column $l+e$ are empty. This implies that $s_1-e\le l\le s_1-1$.
\end{proof}

The following corollary is easy and we omit its proof.

\begin{corollary}\label{5.2.3}
All $r$-partitions in $\mathcal H^\bLambda_\bbeta$ is a totally ordered set with respect to the dominance order.
Moreover, assume one can obtain $\blam^\ast$ from $\bmu$ and $\bnu$ by
deleting a rim $e$-hook in $\bmu^{(i)}$ and $\bnu^{(j)}$, respectively.
If $i<j$, then $\bmu\rhd\bnu$.
\end{corollary}

Next we consider the simple modules of a block satisfying assumptions of this subsection.

\begin{lemma}\label{5.2.4}
If the pair $(\blam, \bs)$ is a Kleshchev $r$-partition, then $\blam^{(r)}\ne \varnothing$.
\end{lemma}

\begin{proof}
By Lemma \ref{5.2.2} there exists $1\le j\le r$ such that $\blam^{(j)}\ne \varnothing$.
If $j<r$, let $\sigma$ be the transposition $(j, r)$ and write $\blam^\sigma=\bnu$.
Then $\bnu=(\varnothing, \dots, \varnothing, \blam^{(j)})$.
We have from Lemma \ref{5.2.2} that $(\bnu, \bs)\in\mathcal H^\bLambda_\bbeta$.
It is clear that if $\t$ is a standard $(\blam, \bs)$-tableau, then $\t_\sigma$ is a standard $(\bnu, \bs)$-tableau.
Note that $s_1=s_2=\cdots=s_r$. This implies that $\res_{\blam, \bs}(\t)=\res_{\bnu, \bs}(\t_\sigma)$.
By Corollary \ref{5.2.3}, $\blam\rhd\bnu$. It is easy to know
that pair $(\blam, \bs)$ is not Kleshchev. This completes the proof.
\end{proof}

Based on the above preparation, we divide the proof of Type II into two cases according to whether or not $e=2$.

\smallskip

{\bf Case 1. $e\ne 2$.} It is a classical result (see \cite[Exercise 5.10]{M}) that there are exactly $e$ partitions
$\lam_{(1)}\rhd\lam_{(2)}\rhd\cdots\rhd\lam_{(e)}$ of weight 1 with
$e$-core $\blam^{\ast(r-1)}$, and because $e\neq 2$, the numbers of standard $\lam_{(e)}$-tableaux
and standard $\blam_{(e-1)}$-tableaux are not the same.
Let $\bmu=(\blam^{\ast(1)}, \dots, \blam^{\ast(r-2)}, \lam_{(e)}, \blam^{\ast(r)})$
and $\bnu=(\blam^{\ast(1)}, \dots, \blam^{\ast(r-2)}, \lam_{(e-1)}, \blam^{\ast(r)}).$
Clearly, the numbers of standard $\bmu$-tableaux and standard $\bnu$-tableaux are not the same.
We have from Lemma \ref{5.2.2}  that $(\bmu, \bs), (\bnu, \bs)\in \mathcal H^\bLambda_\bbeta$.
Moreover, it follows from Lemma \ref{5.2.4} that neither $\bmu$ nor $\bnu$ is a Kleshchev $r$-partition,
and from Corollary \ref{5.2.3} that both $\bmu$ and $\bnu$ are not maximal. We deduce from Corollary \ref{2.18} that
block $\mathcal H^\bLambda_\bbeta$ has infinite representation type.

\medskip

{\bf Case 2.} $e=2$. We can deduce from Lemma \ref{5.2.2} that the core
$L_{\bs}(\bvarnothing)$ can be obtained from  $L_\bs(\blam)\in \mathcal H^\bLambda_\bbeta$ by moving certain
bead at position $(j, l+2)$ to empty position $(j, l)$, where $s_1-2\le l\le s_1-1$. This implies that there are $2r$ abaci in $\mathcal H^\bLambda_\bbeta$, and if $L_\bs(\blam)\in \mathcal H^\bLambda_\bbeta$ then there exists $1\le j\le r$ such that $\blam^{(j)}$ is $(2)$ or $(1, 1)$ and $\blam^{(t)}=\varnothing$ for $t\ne j$. As a result, we get $\bLambda=r\bLambda_{s_1}$ and $\bbeta=\balpha_{s_1}+\balpha_{s_1+1}$.
We have from Lemma \ref{HMbasis} that this block has a basis $\{e_{\balpha_{s_1}+\balpha_{s_1+1}}y_1^iy_2^j\mid 0\leq i<r,\, 0\leq j<2\}$,
or the block is isomorphic to $K[y_1, y_2]/\<y_1^r, y_2^2\>$. Clearly, the block has infinite representation type.

%%%%%%%%%%%%%%%%%%%%%%%%%%%%%%%%%%%%%%%%%%%%%%%%%%%%%%%%%%%%%%%%%%%%%%%%%%%%%%%%%%%%%%%%%%%%%%%%%%%%%%%%%%%%%%%%%%%%%%%%%%%%
%%%%%%%%%%%%%%%%%%%%%%%%%%%%%%%%%%%%%%%%%%%%%%%%%%%%%%%%%%%%%%%%%%%%%%%%%%%%%%%%%%%%%%%%%%%%%%%%%%%%%%%%%%%%%%%%%%%%%%%%%%%%

\subsection{Proof of Part II} 
Let $\mathcal{H}_{\bbeta}^{\bLambda}$ be a block with block moving vector $\mathcal M=(m_1, \dots, m_r)$ satisfying
\begin{enumerate}
\item[(1)] $w=\sum_i m_i\geq 2$;
\item[(2)] $m_i\leq 1$ for $1\leq i\leq r$;
\item[(3)] $\prod_i m_i=0$.
\end{enumerate}
Then $\mathcal{H}_{\bbeta}^{\bLambda}$ may be either of infinite type or finite type. Before the proof begins, we give a definition.
An oriented quiver $\Gamma_r$ associated with $\mathcal{H}_{\bbeta}^{\bLambda}$ is defined as follows.
The vertex set is $I=\mathbb{Z}/r\mathbb{Z}=\{\bar{1}, \bar{2}, \dots, \bar{r}\}$
and directed edges are $\bar{i}\longrightarrow \overline{i+1}$ for all $m_i=1$. Recall Definition \ref{e-infinite} if necessary below.

\smallskip

\noindent{\bf Infinite representation type cases.}
To simplify the proof, we can assume according to Lemma \ref{3.4.9} in this subsection that $m_r=0$.
A direct benefit is that we can write elements $\bar{i}$ in $\mathbb{Z}/r\mathbb{Z}$ as $i$
and compare them as natural numbers without any confusion.
Let $(\blam, \bs)\in \mathcal H^\bLambda_\bbeta$ with $\bs \in \overline{\mathcal A}^r_e$
and $\Gamma_r$ the associated oriented quiver.
Denote by $L_{\bs^\ast}(\blam^\ast)$ the core of $L_\bs(\blam)$.
We handle four cases in this subsection, which are all of infinite representation type.

\smallskip

{\em {\bf Case 1.} There exist at least two connected components (not isolated dots) in $\Gamma_r$.}
Since there exist at least two connected components (not isolated dots) in $\Gamma_r$,
there exist $1\le i_1< i_2+1<i_3\le i_4<r$ such that the path from $i_1$ to $i_2+1$
and path from $i_3$ to $i_4+1$ are two connected components.
To prove block $\mathcal H^\bLambda_\bbeta$ has infinite representation type,
we construct two incomparable abaci from $L_\bs(\blam)$.

By Lemma \ref{move vector minus}, we have $s^*_{i_1}=s_{i_1}-1$, $s^*_{i_2+1}=s_{i_2+1}+1$, $s^*_{i_3}=s_{i_3}-1$
and $s^*_{i_4+1}=s_{i_4+1}+1$. Note that $\bs \in \overline{\mathcal A}^r_e$.
This implies $s_{i_1}\le s_{i_2+1}$ and $s_{i_3}\le s_{i_4+1}$.
Therefore, $s^*_{i_1}+2\le s^*_{i_2+1}$, $s^*_{i_3}+2\le s^*_{i_4+1}$.
Then we can deduce from Lemma \ref{abacus simple property} (4) that there exist
$h_1, h_2, h_3, h_4\in \Z$ with $h_1<h_2$ and $h_3<h_4$ such that in $L_{\bs^*}(\blam^*)$,
\begin{enumerate}
\item[$\bullet$] positions $(i_1, h_1)$ and $(i_1, h_2)$ are empty and positions $(i_2+1, h_1)$ and $(i_2+1, h_2)$ have a bead placed;
\item[$\bullet$] positions $(i_3, h_3)$ and $(i_3, h_4)$ are empty and positions $(i_4+1, h_3)$ and $(i_4+1, h_4)$ have a bead placed.
\end{enumerate}
Define $L_{\bar \bs}(\bar \blam)$ to be the abacus such that the moving vector
from $L_\bs(\blam)$ to it is $\mathcal M=(m_1, \dots, m_r)$, where
$m_j=1$ if $i_1\le j\le i_2$ or $i_3\le j\le i_4$, and
0 otherwise. It is clear that $L_{\bar \bs}(\bar \blam)$ is uniquely determined
and rows $i_1, \dots, i_2, i_2+1$ and $i_3, \dots, i_4, i_4+1$
in $L_{\bar \bs}(\bar \blam)$ and $L_{\bs^*}(\blam^*)$ are the same, respectively.
Denote by $L_\bu(\bmu)$ the abacus obtained by moving in $L_{\bar \bs}(\bar \blam)$
the bead at positions $(i_2+1, h_2)$ and $(i_4+1, h_3)$ to positions $(i_1, h_2)$ and $(i_3, h_3)$ respectively.
Moreover, move in $L_{\bar \bs}(\bar \blam)$ the beads at positions $(i_2+1, h_1)$
and $(i_4+1, h_4)$ to positions $(i_1, h_1)$ and $(i_3, h_4)$, respectively.
Denote by $L_\bv(\bnu)$ the new abacus obtained.
By Lemma \ref{operation composition}, both the moving vectors from abaci $L_\bu(\bmu)$ and $L_\bv(\bnu)$
to $L_{\tilde \bs}(\tilde \blam)$ are equal to $\mathcal M$.
It follows from Lemma \ref{same block} that $\bu=\bv=\bs$ and $(\bmu, \bs),\, (\bnu, \bs)\in \mathcal H^\bLambda_\bbeta$.
To prove $L_\bs(\bmu) \parallel L_\bs(\bnu)$, we only need to take
$(\kappa_1, \iota_1)=(i_1, h_2)$ and $(\kappa_2, \iota_2)=(i_4+1, h_3)$.

\smallskip

{\bf In the following three cases, we assume in $\Gamma_r$,
there is only one connected component (not isolated dot),
which is a path from $i_1$ to $i_2+1$ with length not less than 2.
It is not difficult to check that the moving vector from $L_\bs(\blam)$
to $L_{\bs^*}(\blam^*)$ is $\mathcal M=(m_1, \dots, m_r)$, where
$m_j=1$ if $i_1\le j\le i_2$, and 0 otherwise.}

\smallskip

{\em {\bf Case 2.} There exists $i_1<i_3<i_2$ such that $s_{i_3}\ne s_{i_3+1}$.}
We have from Lemma \ref{move vector minus} that $s^*_{i_1}=s_{i_1}-1,\,\, s^*_{i_2+1}=s_{i_2+1}+1$
and $s^*_j=s_j$ for all $i_1<j\le i_2$ and thus $s^*_{i_3}=s_{i_3}$.
Since $\bs \in \overline{\mathcal A}^r_e$ and $s_{i_3}\ne s_{i_3+1}$, we have $s_{i_3}<s_{i_3+1}$.
And then $s^*_{i_3}=s_{i_3}<s_{i_3+1}\le s^*_{i_3+1}$, or $s^*_{i_3}+1\le s^*_{i_3+1}$.
Furthermore, combining $s^*_{i_1+1}=s_{i_1+1}$, $s^*_{i_1}+1=s_{i_1}$
with $s_{i_1}\le s_{i_1+1}$ gives $s^*_{i_1}+1\le s^*_{i_1+1}$.
Similarly, $s^*_{i_2}+1\le s^*_{i_2+1}$.
Note that $L_{\bs^*}(\blam^*)$ is complete. Then by Lemma \ref{abacus simple property} (4),
$s^*_{i_1}+1\le s^*_{i_1+1}$ implies that there exists $l_1\in\Z$ such that in $L(\blam^*, \bs^*)$,
position $(i_1, l_1)$ is empty and positions $(i_1+1, l_1), \dots, (i_2+1, l_1)$ have a bead placed.
Similarly, there exists $l_2\in\Z$ such that
positions $(i_1, l_2), \dots, (i_3, l_2)$ are empty and positions $(i_3+1, l_2), \dots, (i_2+1, l_2)$ have a bead placed,
and there exists $l_3\in\Z$ such that
positions $(i_1, l_3), \cdots, (i_2, l_3)$ are empty and position $(i_2+1, l_3)$ has a bead.
The configuration of $L(\blam^*, \bs^*)$ described above implies that $l_1,\, l_2$ and $l_3$ are different from each other.
Define $\{h_1, h_2, h_3\}$ to be equal to $\{l_1, l_2, l_3\}$ as a set satisfying $h_1<h_2<h_3$,
and define $\{j_1, j_2, j_3\}$ to be equal to $\{i_1+1, i_3+1, i_2+1\}$ as a set such that in $L_{\bs^*}(\blam^*)$,
\begin{enumerate}
\item[$\bullet$] positions $(i_1, h_1), \dots, (j_1-1, h_1)$ are empty and positions $(j_1, h_1), \dots, (i_2+1, h_1)$ have a bead placed;
\item[$\bullet$] positions $(i_1, h_2), \dots, (j_2-1, h_2)$ are empty and positions $(j_2, h_2), \dots, (i_2+1, h_2)$ have a bead placed;
\item[$\bullet$] positions $(i_1, h_3), \dots, (j_3-1, h_3)$ are empty and positions $(j_3, h_3), \dots, (i_2+1, h_3)$ have a bead placed.
\end{enumerate}

We consider two possibilities.

(1) $j_1>j_3$. \,\, Move in $L_{\bs^*}(\blam^*)$ the  beads at positions $(j_3, h_3)$ and $(i_2+1, h_1)$
to positions $(i_1, h_3)$ and $(j_3, h_1)$, respectively. Denote by $L_\bu(\bmu)$ the new abacus.
Another abacus $L_\bv(\bnu)$ is obtained by moving in $L_{\bs^*}(\blam^*)$
the beads at position $(i_2+1, h_2)$ to position $(i_1, h_2)$.
By Lemma \ref{operation composition}, both the moving vectors from $L_\bu(\bmu)$ and $L_\bv(\bnu)$ to $L_{\bs^*}(\blam^*)$ are equal to $\mathcal M$.
We reach a conclusion by Lemma \ref{same block} that $\bu=\bv=\bs$ and $(\bmu, \bs), (\bnu, \bs)\in \mathcal H^\bLambda_\bbeta$.
By taking $(\kappa_1, \iota_1)=(i_1, h_3)$ and $(\kappa_2, \iota_2)=(i_2+1, h_1)$,
we get $L_\bs(\bmu) \parallel L_\bs(\bnu)$.

(2) $j_1<j_3$.\,\, Denote by $L_\bu(\bmu)$ the abacus obtained by moving in
$L_{\bs^*}(\blam^*)$ the  bead at position $(i_2+1, h_2)$ to position $(i_1, h_2)$.
On the other hand, move in $L_{\bs^*}(\blam^*)$ the  beads at positions $(j_1, h_1)$ and $(i_2+1, h_3)$
to positions $(i_1, h_1)$ and $(j_1, h_3)$, respectively, and denote the new abacus by $L_\bv(\bnu)$.
Be Lemma \ref{same block}, we have $\bu=\bv=\bs$ and $(\bmu, \bs), (\bnu, \bs)\in \mathcal H^\bLambda_\bbeta$.
To reach $L_\bs(\bmu) \parallel L_\bs(\bnu)$, we can choose $(\kappa_1, \iota_1)=(i_1, h_2)$ and $(\kappa_2, \iota_2)=(i_2+1, h_2)$.

\smallskip

{\em {\bf Case 3.} $s_{i_1}\ne s_{i_1+1}$.}
By analyzing similarly as in Case 2, we get $s^*_{i_1}+2\le s^*_{i_1+1}$ and $s^*_{i_2}+1\le s^*_{i_2+1}$.
According to Lemma \ref{abacus simple property} (4), there exist integers $h_1< h_2$ such that in $L_{\bs^*}(\blam^*)$,
positions $(i_1, h_1)$ and $(i_1, h_2)$ are empty and at positions $h_1$ and $h_2$ all rows $i_1+1, \dots, i_2+1$ have a bead placed.
For the same reason, there exists $l_1\in \Z$ such that in $L_{\bs^*}(\blam^*)$
positions $(i_1, l_1), \dots, (i_2, l_1)$ are empty and position $(i_2+1, l_1)$ has a  bead.
Let us consider three possibilities. We will only illustrate incomparable abaci without details.

(1) $h_1<h_2<l_1$. Move in $L_{\bs^*}(\blam^*)$ the  bead at position $(i_2+1, h_2)$
to position $(i_1, h_2)$. The new abacus is denoted by $L_\bs(\bmu)$.
Move in $L_{\bs^*}(\blam^*)$ the  bead at positions $(i_1+1, h_1)$ and $(i_2+1, l_1)$
to positions $(i_1, h_1)$ and $(i_1+1, l_1)$, respectively. The abacus obtained is denoted by $L_\bs(\bnu)$.
By taking $(\kappa_1, \iota_1)=(i_1, h_2)$ and $(\kappa_2, \iota_2)=(i_2+1, h_2)$,
we arrive at $L_\bs(\bmu) \parallel L_\bs(\bnu)$.

(2) $l_1<h_1<h_2$. \, The proof is similar to (1).

(3) $h_1<l_1<h_2$.
By moving in $L_{\bs^*}(\blam^*)$ the  bead at position $(i_2+1, h_1)$ to position $(i_1, h_1)$, we get abacus $L_\bs(\bmu)$.
Abacus $L_\bs(\bnu)$ is obtained by moving in $L_{\bs^*}(\blam^*)$ the  beads
at positions $(i_2+1, l_1)$ and $(i_1+1, h_2)$ to positions $(i_1+1, l_1)$ and $(i_1, h_2)$, respectively.
It is a routine task to check $L_\bs(\bmu) \parallel L_\bs(\bnu)$ by taking $(\kappa_1, \iota_1)=(i_1+1, h_2)$ and $(\kappa_2, \iota_2)=(i_2+1, h_1)$.

\smallskip

{\em {\bf Case 4.} $s_{i_2}\ne s_{i_2+1}$.}
The proof is similar to Case 3.

\smallskip

\noindent{\bf Finite representation type case.}
Let $(\blam, \bs)\in \mathcal H^\bLambda_\bbeta$ with $\bs \in \mathcal A^r_e$. 
Note that to consider the problem of representation type, we can assume that $(\balpha_j, \bLambda-\bbeta)\geq 0$ for all $0\leq j \leq e-1$. 
Then by Lemma \ref{empty core blocks} the core of $L_\bs(\blam)$ is $L_{\bs^\ast}(\bvarnothing)$.
Let $\Gamma_r$ be the associated oriented quiver.  Suppose that there is only one connected component (not isolated dot) in $\Gamma_r$, which is a path from $i$ to $i+w$. 
According to the results on infinite type in this subsection, if $m_r=1$, then the block has infinite representation type.
Then there is only one case to consider, that is, $s_i=s_{i+1}=\cdots=s_{i+w}$ and $m_r=0$.
Fix the meaning of $i$ and $w$. According to the moving vector, we can depict the form of $L_{\bs^*}(\bvarnothing)$:
\begin{enumerate}
\item[$\bullet$] position $(i, h)$ has a bead if and only if $h<s_i^\ast$;
\item[$\bullet$] position $(i+w, h)$ has a bead if and only if $h<s_i^\ast+2$;
\item[$\bullet$] for $0<t<w$, position $(i+t, h)$ has a bead if and only if $h<s_i^\ast+1$. 
\end{enumerate}
Move in $L_{\bs^*}(\bvarnothing)$ the bead at position $(i+w, s^*_{i}+1)$ to $(i, s^*_{i}+1)$ and denote the new abacus by $L_\bu(\bmu)$.
Then the moving vector from  $L_\bu(\bmu)$ to $L_{\bs^*}(\bvarnothing)$ is $\mathcal{M}$. 
We have from Lemma \ref{same block} that $\bu=\bs$ and $L_\bs(\bmu)$ belongs to $\mathcal H^\bLambda_\bbeta$. 
Clearly, $\bmu=(\varnothing,\dots,\varnothing, (1), \varnothing,\dots,\varnothing)$, where $(1)$ is $\bmu^{(i)}$. This implies that $\bbeta=\balpha_j$, where $j\equiv s_i \pmod e$.

Note that for $1\leq t<i$, we have $s_t=s^*_t\le s^*_{i}=s_{i}-1$, or $s_t<s_{i}$. For $i+w<t\le r$, we have $s_{i+w}+1=s^*_{i+w}\le s^*_t=s_t$, or $s_{i+w}<s_t$. Therefore, $(\bLambda, \balpha_{s_i})=w+1$ and by the definition of cyclotomic KLR algebras  $\mathcal{H}^{\bLambda}_{{\bbeta}}$
is isomorphic to $k[x]/(x^{w+1})$, which has finite representation type. 
Consequently, all blocks that are in the same $W$-orbit with $\mathcal{H}^{\bLambda}_{{\bbeta}}$
are of finite representation type.
By \cite[Theorem 6.8]{AKMW}, this forces these blocks being Morita equivalent with a Brauer tree algebra,
whose Brauer tree is $T_1$ in Lemma \ref{2.9} with $m(1)=w$.
That is, these blocks are Morita equivalent to $k[x]/(x^{w+1})$.

%%%%%%%%%%%%%%%%%%%%%%%%%%%%%%%%%%%%%%%%%%%%%%%%%%%%%%%%%%%%%%%%%%%%%%%%%%%%%%%%%%%%%%%%%%%%%%%%%%%%%%%%%%%%%%%%%%%%%%
%%%%%%%%%%%%%%%%%%%%%%%%%%%%%%%%%%%%%%%%%%%%%%%%%%%%%%%%%%%%%%%%%%%%%%%%%%%%%%%%%%%%%%%%%%%%%%%%%%%%%%%%%%%%%%%%%%%%%%

\subsection{Proof of Part III}

Given a pair $(\blam, \bs)\in \mathcal{H}^{\bLambda}_{\bbeta}$ with $\bs\in\mathcal A^r_e$, let $(m_1, m_2, \dots, m_r)$ be the block moving vector.
In this subsection, we prove that if there exists $1\leq i\leq r$ such that $m_i\geq 2$, then
$\mathcal{H}^{\bLambda}_{\bbeta}$ has infinite representation type.
To this aim, we only need to form incomparable abaci in $\mathcal{H}^{\bLambda}_{\bbeta}$ in the light of Propositions \ref{2.16} and \ref{3.3.5}.
The condition $\bs\in\mathcal A^r_e$ will be extended to $\bs\in\overline{\mathcal A}^r_e$ for convenience. We divide the proof according to if $\prod_jm_j=0$.

\smallskip

\noindent{\bf Proof under condition $\prod_jm_j=0$.} By Lemma \ref{3.4.9} we can assume that $m_r=0$.
Let $\mathcal{F}$ be the operation set from $L_{\bs}(\blam)$ to its core.
It is not difficult to know that $\mathcal{F}$ must be one of the following two types:

\smallskip

Type I. There exists $1\leq i\leq r$ such that $[(i, h_1), \ast],\, [(i, h_2), \ast]\in \mathcal{F}$ with $h_1\neq h_2$;

Type II. $[(i, h_1), \ast], \,[(i, h_2), \ast]\in \mathcal{F}$ implies $h_1=h_2$ for all $1\leq i\leq r$.

\smallskip

If in block $\mathcal{H}^{\bLambda}_{\bbeta}$, there is a pair $(\blam, \bs)$ such that
the operation set $\mathcal{F}$ from $L_\bs(\blam)$ to its core is of {\bf Type I},
we can divide the proof into {\bf 7 cases} according to the possible shape of abaci.

\smallskip

{\em {\bf Case 1.}
Positions $(i+1, h_1)$ and $(i+1, h_2)$ are empty and positions $(i, h_1)$ and $(i, h_2)$ have a bead placed.}
This is just the result of Lemma \ref{2 runners 2 columns}.

\smallskip

{\em {\bf Case 2.}
Position $(i+1, h_1)$ is empty and positions $(i, h_1)$,
$(i, h_2)$ and $(i+1, h_2)$ have a bead placed.}
By Lemma \ref{3.4.3}, there exist at least one empty position before $(i+1, h_2)$.
Let $(j+1, h_2)$ be the first one, where $i<j<r$.
Then all the positions between $(j+1, h_2)$ and $(i-1, h_2)$ are not empty. We consider two possibilities.

\smallskip

(1) $i+1<j$. Then in $L_\bs(\blam)$, positions $(j+1, h_2)$ and $(i+1, h_1)$ are empty
and positions $(j, h_2)$ and $(i, h_1)$ have a bead placed.
These are just the requirements of Lemma \ref{4 runners}.

\smallskip

(2) $i+1=j$. Note that

\begin{enumerate}
\item[$\bullet$] there is a bead at position $(i, h_1)$ and position $(i+1, h_1)$ is empty;
\item[$\bullet$] there is a bead at position $(i+1, h_2)$ and position $(i+2, h_2)$ is empty;
\item[$\bullet$] there is bead at position $(i, h_2)$,
\end{enumerate}
we get all the conditions required by Corollary \ref{3 runners 3 columns(1)}.

\smallskip

{\em {\bf Case 3.} Positions $(i, h_1)$, $(i+1, h_1)$ and $(i+1, h_2)$ are empty and
 position $(i, h_2)$ has a bead.}
We can consider the dual abacus and then we reach Case 2 by Lemma \ref{dual operator}.

\smallskip

{\em {\bf Case 4.} Positions $(i+1, h_1)$ and $(i, h_2)$ are empty and
positions $(i, h_1)$  and $(i+1, h_2)$ have a bead placed.}
Note that the existence of an empty position before position $(i+1, h_2)$ is ensured by Lemma \ref{3.4.3}, and thus we
suppose that position $(j+1, h_2)$ is the first one, where $i<j<r$.
Then there is a bead at position $(j, h_2)$. Furthermore, Lemma \ref{3.4.3} also tells us
there is a position occupied by a bead after position $(i, h_2)$.
Let $(l, h_2)$ be the first one, where$1\le l<i$. Then position $(l+1, h_2)$ is empty.
Clearly $l+1<j$. Now we have gathered all conditions required by Lemma \ref{4 runners}.

\smallskip

{\em {\bf Case 5.} Positions $(i, h_1)$ and $(i+1, h_1)$ are empty and
positions $(i, h_2)$ and $(i+1, h_2)$ have a bead placed.}
By Lemma \ref{3.4.3}, there exists a bead after position $(i, h_1)$.
Let the bead at $(l, h_1)$ be the first one, where $1\leq l<i$.
Then all the positions between $(i+1, h_1)$ and $(l, h_1)$ are empty.
Particularly, position $(l+1, h_1)$ is empty. Moreover, we have from Lemma \ref{3.4.3} that
there exists an empty position before position $(i+1, h_2)$. Let position $(j+1, h_2)$ be the first one,
where $i<j<r$. Then all positions between $(j+1, h_2)$ and $(i, h_2)$
are occupied by a bead, including position $(j, h_2)$.
Moreover, $l+1<j$ is clear. Then all the requirements of Lemma \ref{4 runners} are satisfied.

\smallskip

{\em {\bf Case 6.} There are beads at positions $(i+1, h_1)$ and $(i+1, h_2)$.}
In this case, the operation set $\mathcal{F}$
has to contain $[(i+1, h_1), \ast]$ and $[(i+1, h_2), \ast]$, where $i+1\leq r$. If one of positions
$(i+2, h_1)$ and $(i+2, h_2)$ is empty, then we arrive at Case 1 or Case 2.
Otherwise, repeat the above analysis process. Clearly, the process will end after finite times
since $\mathcal{F}$ is finite and this completes the proof.

\smallskip

{\em {\bf Case 7.} Positions $(i, h_1)$ and $(i, h_2)$ are empty.}
This forces $[(i-1, h_1), \ast],\, [(i-1, h_2), \ast]\in \mathcal{F}$, where $i>1$.
If one of positions $(i-1, h_1)$ and $(i-1, h_2)$ is not empty,
then we are in the circumstances of Case 1 or Case 3.
Otherwise, we can repeat the above analysis process. Note that $\mathcal{F}$
is finite. This implies that we only need to do the analysis process finite times.

\smallskip

Now assume that there is a pair $(\blam, \bs)\in\mathcal{H}^{\bLambda}_{\bbeta}$ such that
the operation set $\mathcal{F}$ from $L_\bs(\blam)$ to its core is of {\bf Type II}.
Assume $m_i\geq 2$ and the operations happen in column $h$.
We divide the proof into the following {\bf 4 cases}.

\smallskip

{\em {\bf Case 1.}\, Position $(i, h)$ is empty and there is a bead at position $(i+1, h)$.}
In this case, it is clear that there exist at least two empty positions before position $(i+1, h)$.
Let position $(j_1, h)$ and $(j_2, h)$ be the first and the second one, respectively, where $i+1<j_1<j_2\leq r$. 
On the other hand, there exist at least two positions occupied by a bead
after position $(i+1, h)$. Let positions $(j_3, h)$ and $(j_4, h)$
be the first and the second one, respectively, where $1\leq j_4<j_3<i$.
It is not difficult to check that
\begin{enumerate}
\item[$\bullet$] positions $(j_1, h)$ and $(j_2, h)$ are empty;
\item[$\bullet$] there are beads at positions $(j_4, h)$ and $(j_3, h)$;
\item[$\bullet$] $1\leq j_4<j_3<j_1<j_2 \leq r$.
\end{enumerate}
By Lemma \ref{4.1.12}, the result follows.

The following three cases can be handled similarly as Case 1. We omit the details.

\smallskip

{\em {\bf Case 2.}\, Position $(i, h)$ is occupied by a bead and position $(i+1, h)$ is empty.}

\smallskip

{\em {\bf Case 3.}\, Both positions $(i, h)$ and $(i+1, h)$ are occupied by a bead.}

\smallskip

{\em {\bf Case 4.}\, Both positions $(i, h)$ and $(i+1, h)$ are empty.}

\smallskip

\noindent{\bf Proof under condition $\prod_jm_j\neq 0$.} Assume that $(\balpha_j, \bLambda-\bbeta)\geq 0$ for all $0\leq j \leq e-1$. We prove a lemma first.
\begin{lemma}\label{posetrelation}
Suppose that the block moving vectors of blocks  $\mathcal H^\bLambda_{\bbeta}$ and $\mathcal H^\bLambda_{\bbeta'}$ 
are $\mathcal{M}$ and $\mathcal{M'}$, respectively, with $\mathcal M=\mathcal M'+(1, \dots, 1)$. If $r$-partitions in 
$\mathcal H^\bLambda_{\bbeta'}$ do not form a totally ordered set, then neither do partitioins in $\mathcal H^\bLambda_\bbeta$.
\end{lemma}

\begin{proof}
Let $L_\bs(\bar \blam)$ and $L_\bs(\bar \bmu)$ be two incomparable abaci in $\mathcal H^\bLambda_{\bbeta'}$. 
Assume that the bead $\CIRCLE_1^1(\bar\blam, \bs)$ is at position $(1, h)$. Move it to $(1, h+e)$ and denote the new abacus by $L_\bu(\blam)$. 
Similarly, we can get a new abacus $L_\bv(\bmu)$ from $L_\bs(\bar \bmu)$. 
It follows from Lemma \ref{same block} that $\bs=\bu=\bv$ and $(\blam, \bs), (\bmu, \bs)\in \mathcal H^\bLambda_\bbeta$. 
Clearly, $\blam$ and $\bmu$ are incomparable.
\end{proof}

Now we can give the proof under the condition $\bs \in \mathcal A^r_e$. Write $m=\min \{m_1,\dots, m_r\}$ and $\mathcal M'=\mathcal M-(m, \dots, m)$.
By Lemma \ref{generalconstruction}, there exists a block $\mathcal H^\bLambda_{\bbeta'}$ with block moving vector $\mathcal M'$. 
A clear fact on $\bbeta'$ is that $(\balpha_j, \bLambda-\bbeta')\geq 0$ for all $0\leq j \leq e-1$.
We divide the proof into two cases.

\smallskip

{\em {\bf Case 1.}\, $\mathcal M'=(0, \dots, 0)$}.
If there exists some $1\le i\le r-1$ such that $s_i\ne s_{i+1}$, then by using Lemma \ref{posetrelation} finite times 
the $r$-partitions in $\mathcal H^\bLambda_{\bbeta}$ can not form a totally ordered set  because the $r$-partitions 
in a block with moving vector $(1, 1, \dots, 1)$ can not form a totally ordered set (see Type I of Part I). 
Now let $s_1=s_2=\cdots=s_r$. Assume that the bead $\CIRCLE_1^1(\bvarnothing, \bs^{\ast})$ is at position $(1, h)$. 
Then the bead $\CIRCLE_1^r(\bvarnothing, \bs^{\ast})$ is at $(r, h)$. Move the two beads aforementioned to $(1, h+(m-1)e)$ and $(r, h+e)$, respectively 
and denote the new abacus by $L_\bu(\blam)$. Assume that 
the bead $\CIRCLE_1^2(\bvarnothing, \bs^{\ast})$ is at position $(2, h)$. Move it to $(2, h+me)$ and denote the new abacus by $L_\bu(\blam)$. 
Then clearly $(\bmu, \bs)$ and $(\blam, \bs)$  belong to $\mathcal H^\bLambda_{\bbeta}$ and $\blam\, \| \,\bmu$.

\smallskip

{\em {\bf Case 2.}\, $\mathcal M'\ne(0, \dots, 0)$}.
If there exists $m_x\geq m+2$, then by ``Proof under condition $\prod m_i=0$" and Lemma \ref{posetrelation}, nothing need to prove. 

Now we assume that all $m_x\leq m+1$. If the $r$-partitions in $\mathcal H^\bLambda_{\bbeta'}$ is not a totally ordered set, 
then nothing need to prove too. Otherwise, $\mathcal H^\bLambda_{\bbeta'}$ is forced to be either a truncated polynomial ring or a weight one block. 

If $\mathcal H^\bLambda_{\bbeta'}$ is a truncated polynomial ring,
let $\bar{\blam}$  be an $r$-partition with $\bar{\blam}^{(i)}=(1)$ and $\bar{\blam}^{(t)}=\varnothing$ if $t\ne i$, and 
let $\bar{\bmu}$  be an $r$-partition with $\bar{\bmu}^{(i+1)}=(1)$ and $\bar{\blam}^{(t)}=\varnothing$ if $t\ne i+1$. 
It is not difficult to check that $(\bar{\blam}, \bs)$ and $(\bar{\bmu}, \bs)$ are in $\mathcal H^\bLambda_{\bbeta'}$.
Assume that the beads $\CIRCLE_1^{i+w}(\bar{\blam}, \bs)$ at position $(i+w, h_1)$ and $\CIRCLE_1^{i+1}(\bar{\bmu}, \bs)$ at $(i+1, h_2)$. 
Move the bead $\CIRCLE_1^{i+w}(\bar{\blam}, \bs)$ to position $(i+w, h_1+me)$ and denote the new abacus by $L_{\bs}(\blam)$. 
Move the bead $\CIRCLE_1^{i+1}(\bar{\bmu}, \bs)$ to position $(i+1, h_2+me)$ and denote the new abacus by $L_{\bs}(\bmu)$. 
Then clearly $(\blam, \bs)$ and $(\bmu, \bs)$ belong to $\mathcal H^\bLambda_{\bbeta}$ and $\blam\, \| \,\bmu$.

If $\mathcal H^\bLambda_{\bbeta'}$ is a weight one block, 
we can assume without loss generality by Lemma \ref{3.4.9} that $m'_1=1$. However, we only have $\bs\in\overline{\mathcal A}^r_e$ now.
Note that $s_1\leq s_2$, $s_1=s_1^\ast+1$ and $s_2=s_1^\ast-1$. Then $s_1^\ast+2\leq s_2^\ast$. Let $s_2^\ast-s_1^\ast-1=k$. Clearly, $k\geq 1$.
Move in $L_{\bs^{\ast}}(\bvarnothing)$ the bead at position $(2, s_2^\ast-1)$ to $(1, s_2^\ast-1)$ and denote the new abacus by $L_{\bs^{\ast}}(\bar \blam)$.
Then ${\bar \blam}^{(1)}=(k)$ and ${\bar \blam}^{(t)}=\varnothing$ for $t\neq 1$. 
Furthermore, move in $L_{\bs^{\ast}}(\bvarnothing)$ the bead at position $(2, s_1^\ast)$ to $(1, s_1^\ast)$ and denote the new abacus by $L_{\bs^{\ast}}(\bar \bmu)$.
Then ${\bar \bmu}^{(2)}=(1^k)$ and ${\bar \blam}^t=\varnothing$ for $t\neq 2$. 
It is not difficult to check that $({\bar \blam}, \bs)$ and $({\bar \bmu}, \bs)$ are in $\mathcal H^\bLambda_{\bbeta'}$.
Assume that $\CIRCLE_1^{r}(\bar{\blam}, \bs)$ is at position $(r, h_1)$ and move it to $(r, h_1+me)$. Denote the new abacus by $L_{\bs}(\blam)$. 
Similarly, assume that $\CIRCLE_1^{2}(\bar{\bmu}, \bs)$ is at position $(2, h_2)$ and move it to $(2, h_2+me)$. Denote this new abacus by $L_{\bs}(\bmu)$.
Then one can check that $L_{\bs}(\blam)$ and $L_{\bs}(\bmu)$ are in $\mathcal H^\bLambda_{\bbeta}$ and $\blam\,\|\,\bmu$.

\medskip

%%%%%%%%%%%%%%%%%%%%%%%%%%%%%%%%%%%%%%%%%%%%%%%%%%%%%%%%%%%%%%%%%%%%%%%%%%%%%%%%%%
%%%%%%%%%%%%%%%%%%%%%%%%%%%%%%%%%%%%%%%%%%%%%%%%%%%%%%%%%%%%%%%%%%%%%%%%%%%%%%%%%%
%%%%%%%%%%%%%%%%%%%%%%%%%%%%%%%%%%%%%%%%%%%%%%%%%%%%%%%%%%%%%%%%%%%%%%%%%%%%%%%%%%
%%%%%%%%%%%%%%%%%%%%%%%%%%%%%%%%%%%%%%%%%%%%%%%%%%%%%%%%%%%%%%%%%%%%%%%%%%%%%%%%%%
%%%%%%%%%%%%%%%%%%%%%%%%%%%%%%%%%%%%%%%%%%%%%%%%%%%%%%%%%%%%%%%%%%%%%%%%%%%%%%%%%%

\section{Derived equivalence of blocks}

It is an open problem to give a necessary and sufficient condition for two
blocks of Ariki-Koike algebras being derived equivalent.
In this section, we give examples of blocks with the same weight associated with the same multicharge that are not derived equivalent and examples 
of derived equivalent blocks being in different orbits under the adjoint action of the affine Weyl group
by using the theory developed in this paper. 

Note that we can easily find two blocks with the same weights, one of them being infinite type and the other
being Morita equivalent to a truncated polynomial ring, which has finite representation type.
So our discussion will center on constructing examples of derived equivalent blocks being in different orbits.

We first give a necessary and sufficient condition for two weight one blocks (Morita equivalent to Brauer tree algebras) being derived equivalent. 
Note that the problem of derived equivalence of Brauer tree algebras is completely resolved in \cite{R}.
We will translate the results \cite[Theorem 6.8]{AKMW} and \cite[Theorem 4.2]{R}
into the language of the so-called subabacus moving vector of a block.

\begin{definition}
Let $\mathcal{F}$ be the operation set from $L_s(\blam)\in\mathcal{H}^{\Lambda}_{\bbeta}$ to its core.
Define $\mathcal{W_{\blam}}=(w_{0\blam}, w_{1\blam}, \dots, w_{e-1\blam})$,
where $w_{i\blam}=\sharp\{[(x, h), \ast]\in \mathcal{F}\mid 1\leq x \leq r, \,h\equiv i \,({\rm mod} \,e)\}.$
Then $\mathcal{W}(\mathcal{H}^{\Lambda}_{\bbeta})=(w_0, w_1, \dots, w_{e-1})
=\sum_{(\blam, \bs)\in \mathcal{H}^{\Lambda}_{\bbeta}}\mathcal{W_{\blam}}$ is
called the subabacus moving vector of block $\mathcal{H}^{\Lambda}_{\bbeta}$.
\end{definition}

Given two pairs $(\blam, \bs)$ and $(\bmu, \bu)$ with $\bs, \bu\in \overline{\mathcal A}^r_e$
and $(\blam^\ast, \bs^\ast)$ and $(\bmu^\ast, \bu^\ast)$ the corresponding cores, respectively.
Denote by $B_{\blam, \bs}$ and $B_{\bmu, \bu}$ the blocks containing $(\blam, \bs)$ and $(\bmu, \bu)$, respectively.
Suppose that $w(B_{\blam, \bs})=w(B_{\bmu, \bu})=1$.

\begin{proposition}\label{7.1.1}
Blocks $B_{\blam, \bs}$ and $B_{\bmu, \bu}$ are derived equivalent
if and only if the number of non-zero components in $\mathcal{W}(B_{\blam, \bs})$
is equal to that of $\mathcal{W}(B_{\bmu, \bu})$.
\end{proposition}

\begin{proof}
Let $\mathcal{M}=(m_1, m_2, \dots, m_r)$ be the block moving vector of $B_{\blam, \bs}$
with $m_j=1$ and $m_i=0$ for all $1\leq i\leq r$, $i\neq j$.
Assume that $s_{j+1}-s_j=a$ ($s_{r+1}$ is defined to be $s_1+e$ if $e< \infty$).
We claim that the number of $r$-partitions in $B_{\blam, \bs}$ is $a+2$.
In fact, if $j<r$, we have from Lemma \ref{move vector minus} that $s^*_{j+1}=s_{j+1}+1$ and $s^*_j=s_j-1$
and thus $s^*_j+a+2=s^*_{j+1}$. It follows from Lemma \ref{abacus simple property} (4) that
there exist $h_1, \dots, h_{a+2}\in \Z$ such that in $L_{\bs^*}(\blam^*)$,
positions $(j, h_1), \dots, (j, h_{a+2})$ are empty and
positions $(j+1, h_1), \dots, (j+1, h_{a+2})$ have a bead placed.
Let $h$ be an integer such that all positions $(x,y)$ in $L_{\bs^*}(\blam^*)$ are occupied by a
bead, where $1 \leq x \leq r$ and $y \leq h$. In the light of Lemma \ref{charge minus}, $\mathfrak n^h_{j+1}-\mathfrak n^h_j=a+2$.
Note that $L_{\bs^*}(\blam^*)$ is complete. This implies that
if position $(j, k)$ is empty and position $(j+1, k)$ has a bead,
then $k\in\{h_1, \dots, h_{a+2}\}$.
Because $w(B_{\blam, \bs})=1$, this forces each abacus of a pair in $B_{\blam, \bs}$
has to be obtained by moving in $L_{\bs^*}(\blam^*)$ the bead at position $(j+1, k)$ to position $(j, k)$
with $k\in \{h_1, \dots, h_{a+2}\}$. The case $j=r$ can be transformed into the case $j<r$ by using Lemma \ref{3.4.9}.

Combining \cite[Theorem 4.12]{F1} with \cite[Theorem 6.8]{AKMW} gives that $B_{\blam, \bs}$
is Morita equivalent to a Brauer tree algebra, whose Brauer tree has $a+1$ edges without exceptional vertex.

Moreover, by Lemma \ref{abacus simple property} (3) we have $e \nmid h_x-h_y$ for all $x, y\in \{1, 2, \dots, a+2\}$.
This implies that for arbitrary pair $(\blam, \bs)\in B_{\blam, \bs}$, columns $h_1, \dots, h_{a+2}$
are in different subabaci of $L_\bs(\blam)$, and consequently, the number of non-zero components
of $\mathcal{W}(B_{\blam, \bs})$ is $a+2$.

By the same reason, block $B_{\bmu, \bu}$
is also Morita equivalent to a Brauer tree algebra, whose Brauer tree has no exceptional vertex. According to \cite[Theorem 4.2]{R} derived equivalence classes
of Brauer tree algebras are determined by the number of edges and the multiplicity of the exceptional vertex.
Then  the proposition follows.
\end{proof}

Now let us give some examples.
We begin with the blocks that Morita equivalent to truncated polynomial rings.

\begin{example}
Let $e=5$ and $\bs=(1, 1, 1, 3, 3, 3)$. Take
$\blam=((1), \varnothing, \varnothing, \varnothing, \varnothing, \varnothing)$ and
$\bmu=(\varnothing, \varnothing, \varnothing, (1), \varnothing, \varnothing)$.
It is not difficult to check that $\mathcal{M}(B_{\blam, \bs})=(1, 1, 0, 0, 0, 0)$ and
$\mathcal{M}(B_{\bmu, \bs})=(0, 0, 0, 1, 1, 0)$. According to Section 5.3,
both $B_{\blam, \bs}$ and $B_{\bmu, \bs}$ are isomorphic to $K[x]/(x^3)$ and
all abaci in block $B_{\bmu, \bs}$ can be determined completely.
In the accordance with the definition of the action of an affine Weyl group on abaci,
none of the abaci in $B_{\bmu, \bs}$ belong to the image of $L_\bs(\blam)$.
This implies by Proposition \ref{3.2.6}
that blocks $B_{\blam, \bs}$ and $B_{\bmu, \bs}$ are in different orbits.
\end{example}

Another example is from the blocks of weight one.

\begin{example}
For arbitrary $r\geq 3$, let $e=r$ and $\bs=(0, 1, 2, \dots, r-1)$.
For $1\le i\le r$, define $\blam\,[i]=(\underbrace{\varnothing, \dots, \varnothing}_{i-1}, (2), \underbrace{\varnothing, \dots, \varnothing}_{r-i})$.
Clearly, $\mathcal{M}(B_{\blam\,[i], \bs})=(m_1, m_2, \dots, m_r)$, where $m_i=1$ and $m_j=0$ if $j\neq i$. It follows from Proposition \ref{7.1.1}
that all blocks $B_{\blam\,[i], \bs}$ are derived equivalent, and that all abaci in
$B_{\blam\,[i], \bs}$ can be listed completely. By Proposition \ref{3.2.6},
all blocks $B_{\blam\,[i], \bs}$ are in different orbits.
\end{example}

The last example is a derived equivalent class of infinite representation type blocks. We omit the details here.

\begin{example}
Let $e=\infty$ and $k\geq 2$ be an integer. Take $r=3k$ and $\bs=(1, 1, 2, \dots, 2j-1, 2j-1, 2j, \dots, 2k-1, 2k-1, 2k)$.
For $1\le i\le k$, define $\blam\,[i]=(\underbrace{\varnothing, \dots,  \varnothing}_{3(i-1)}, (2),
\underbrace{\varnothing, \dots, \varnothing}_{r-3i+2})$.
It is not difficult to check $B_{\blam [i], \bs}$ is of infinite representation type.
Moreover, all blocks $B_{\blam [i], \bs}$, $1\le i\le k$, are derived equivalent and are
 in different orbits.
\end{example}

\smallskip

%%%%%%%%%%%%%%%%%%%%%%%%%%%%%%%%%%%%%%%%%%%%%%%%%%%%%%%%%%%%%%%%%%%%%%%%%%%%%%%%%%%%%%%%%%%%%%%%%%%%%%%%%%%%%%%%%%%%%%%%%%%%%%%
%%%%%%%%%%%%%%%%%%%%%%%%%%%%%%%%%%%%%%%%%%%%%%%%%%%%%%%%%%%%%%%%%%%%%%%%%%%%%%%%%%%%%%%%%%%%%%%%%%%%%%%%%%%%%%%%%%%%%%%%%%%%%%%
%%%%%%%%%%%%%%%%%%%%%%%%%%%%%%%%%%%%%%%%%%%%%%%%%%%%%%%%%%%%%%%%%%%%%%%%%%%%%%%%%%%%%%%%%%%%%%%%%%%%%%%%%%%%%%%%%%%%%%%%%%%%%%%
%%%%%%%%%%%%%%%%%%%%%%%%%%%%%%%%%%%%%%%%%%%%%%%%%%%%%%%%%%%%%%%%%%%%%%%%%%%%%%%%%%%%%%%%%%%%%%%%%%%%%%%%%%%%%%%%%%%%%%%%%%%%%%%

\section{Blocks of cyclotomic $q$-Schur algebra}

We end our paper by a remark on representation type of blocks of a cyclotomic $q$-Schur algebra.

Combining Theorem A with Theorem B in \cite{A4} gives that a block of an Iwahori-Hecke algebra of type B
has finite representation type if and only if its weight is not more than one.
This implies that the representation type of a block of the cyclotomic $q$-Schur algebra
associated to a type B Iwahori-Hecke algebra is finite if and only if its weight is not more than one.
On the other hand, in \cite[Corollary 3.20]{W}, Wada proved that under certain conditions
any block of $\mathcal{S}_{n, r}(q, Q_1, Q_2, \dots, Q_r)$ is Morita equivalent to a
certain block of $\mathcal{S}_{n', 2}(q, Q_i, Q_j)$ for some $i, j\in \{1, 2, \dots, r\}$.
%Therefore, the representation type of each block of a cyclotomic $q$-Schur algebra can be determined in this sense.
By using {\bf Theorem}, we can give a direct result under no conditions and without using the Morita equivalence mentioned above.

We point out that if a block $B$ of $\mathcal{H}_{n}(q, Q)$ is Morita equivalent to a truncated polynomial ring,
then the Auslander algebra (see \cite{ARS} for definition) of $B$ is just the corresponding block of $\mathcal{S}_{n, r}$.
It is well-known that the Auslander algebra of $K[x]/(x^i)$ has finite representation type if and only if $0<i<4$.
Let $\mathcal{B}$ be a block of $\mathcal{S}_{n, r}$ and $B$ the corresponding block in $\mathcal{H}_{n}(q, Q)$.
Then by {\bf Theorem}, we get
\begin{enumerate}
\item[(1)] If $w(\mathcal{B})>2$, then $\mathcal{B}$ has infinite representation type.
\item[(2)] If $w(\mathcal{B})<2$, then $\mathcal{B}$ has finite representation type.
\item[(3)] If $w(\mathcal{B})=2$, then $\mathcal{B}$ has finite type if and only if $B$ has finite type.
\end{enumerate}

According to the above result, $w(\mathcal{B})=1$ is only a sufficient condition for $\mathcal{B}$
being of finite representation type.

\bigskip

%%%%%%%%%%%%%%%%%%%%%%%%%%%%%%%%%%%%%%%%%%%%%%%%%%%%%%%%%%%%%%%%%%%%%%%%%%%%%%%%%%%%%%%%%%%%%%%%%%%%%%%%%%%%%%%%%%%%%%%%%%%%%%%
%%%%%%%%%%%%%%%%%%%%%%%%%%%%%%%%%%%%%%%%%%%%%%%%%%%%%%%%%%%%%%%%%%%%%%%%%%%%%%%%%%%%%%%%%%%%%%%%%%%%%%%%%%%%%%%%%%%%%%%%%%%%%%%

\noindent{\bf Acknowledgement}
The authors would like to thank the anonymous referee for her/his numerous helpful comments and corrections.
The authors are grateful to Prof. W. Hu for some helpful conversations. They also thank Prof. Kai Meng Tan for some comments
on a draft of this article. Most of this work was done when Li visited Institute of Algebra
and Number Theory at University of Stuttgart from August 2021 to September 2022.
He takes this opportunity to express his sincere thanks to the institute and Prof. S. Koenig for the hospitality during the visit.


\bigskip

\begin{thebibliography}{}

\bibitem{A} J. Alperin, Local representation theory. Cambridge University Press. Cambridge, 1986.

\bibitem{A1} S. Ariki, {\em On the decomposition numbers of the Hecke algebra of $G(m, 1, n)$},
J. Math. Kyoto Univ. {\bf 36} (1996) 789-808.

\bibitem{A2} S. Ariki, {\em On the classification of simple modules for cyclotomic
Hecke algebras of type $G(r, 1, n)$ and Kleshchev multi-partitions},
Osaka J. Math. {\bf 38} (2001) 827-837.

\bibitem{A3} S. Ariki, {\em Hecke algebras of classical type and their representation type},
Proc. Lond. Math. Soc. {\bf 91} (2005) 355-413.

\bibitem{A4} S. Ariki, {\em Representation type for block algebras of Hecke algebras of classical type}, Adv. Math. {\bf 317} (2017) 823-845.

\bibitem{AIP} S. Ariki, K. Iijima and E. Park, {\em Representation type of finite quiver Hecke algebras of type $A_l^{(1)}$ for arbitrary parameters},
Int. Math. Res. Not. {\bf 2015} (2015) 6070-6135.

\bibitem{AKMW} S. Ariki, R. Kase, K. Miyamoto and K. Wada, {\em Self-injective cellular algebras of polynomial growth representation type},
 Algebr. Represent. Theory {\bf 23} (2020) 833-871.

\bibitem{AK} S. Ariki and K. Koike, {\em A Hecke algebra of
$(\mathbb{Z}/r\mathbb{Z})\wr\mathfrak{S}_n$ and construction of its
irreducible representations},  Adv. Math. {\bf 106} (1994) 216-243.

\bibitem{AM} S. Ariki and A. Mathas, {\em The representation type of Hecke algebras of type B}, Adv. Math. {\bf 181} (2004) 134-159.

\bibitem{AP1} S. Ariki and E. Park, {\em Representation type of finite quiver Hecke algebras of type $A_{2l}^{(2)}$}, J. Algebra {\bf 397} (2014) 457-488.

\bibitem{ASW} S. Ariki, L. Song and Q. Wang, {\em Representation type of cyclotomic quiver Hecke algebras of type $A_\ell^{(1)}$}, arXiv: 2302.14477.

\bibitem{ARS} M. Auslander, I. Reiten and S. O. Smalo, Representation Theory of Artin Algebras. Cambridge
University Press, Cambridge, 1995.

\bibitem{BM} M. Brou$\rm\acute{e}$ and G. Malle, {\em Zyklotomische Heckealgebren},
Ast\'{e}risque {\bf 212} (1993) 119-189.

\bibitem{BK} J. Brundan and A. Kleshchev, {\em Blocks of cyclotomic Hecke algebras and Khovanov-Lauda algebras}, Invent. Math. {\bf 178} (2009) 451-484.

\bibitem{BK2} J. Brundan and A. Kleshchev, {\em Graded decomposition numbers for cyclotomic Hecke algebras}, Adv. Math. {\bf 222} (2009) 1883–1942.

\bibitem{BKW} J.~Brundan, A.~Kleshchev and W.~Wang, {\em Graded Specht modules}, J. reine angew. Math. {\bf 655} (2011) 61-87.

\bibitem{C} I.V. Cherednik, {\em A new interpretation of Gelfand-Tzetlin bases}, Duke Math. J. {\bf 54} (1987) 563-577.

\bibitem{CR} J. Chuang and R. Rouquier, {\em Derived equivalences for symmetric groups and $sl_2$-categorification}, Ann. of Math. {\bf 167} (2008) 245-298.

\bibitem{DJM} R. Dipper, G. James and A. Mathas, {\em Cyclotomic $q$-Schur
algebras,} Math. Z. {\bf 229} (1999) 385-416.

\bibitem{DM} R. Dipper and A. Mathas, {\em Morita equivalences of Ariki-Koike algebras}, Math. Z. {\bf 240} (2002) 579-610.

\bibitem{Dr} Yu. A. Drozd, Tame and wild matrix problems, representations and quadratic forms, Institute
of Mathematics, Academy of Sciences, Ukrainian SSR, Kiev 1979. A.M.S Transl. 128
(1986) 31-55.

\bibitem{E} K. Erdmann, Blocks of tame representation type and related algebras, Lect. Note Math. vol.1428, Springer-Verlag, 1990.

\bibitem{EN} K. Erdmann and D.K. Nakano, {\em Representation type of Hecke algebras of type A}, Trans. Amer. Math. Soc. {\bf 354} (2002) 275-285.

\bibitem{F1} M. Fayers, {\em Weights of multipartitions and representations of
Ariki-Koike algebras}, Adv. Math. \textbf{206} (2008) 112-144.

\bibitem{F2} M. Fayers, {\em Core blocks of Ariki-Koike algebras}, J. Algebr. Comb. {\bf 26} (2007) 47–81.

\bibitem{G1} M. Geck, {\em Brauer trees of Hecke algebras}, Comm. Algebra {\bf 20} (1992) 2937-2973.

\bibitem{G2} M. Geck, {\em Hecke algebras of finite type are cellular}, Invent. Math. {\bf 169} (2007) 501-517.

\bibitem{GL} J. Graham and G. Lehrer, {\em Cellular algebras},
 Invent. Math. {\bf 123} (1996) 1-34.
 
\bibitem{HC} J. Hu, {\em The number of simple modules for the Hecke algebras of type $G(r, p, n)$ (with an appedix by Xiaoyi Cui)}, J. Algebra {\bf 321} (2009) 3375-3396.

%\bibitem{H} J. Hu, {\em On a generalization of the Dipper-James-Murphy conjecture}, J. Comb. Theor. Series A {\bf 118} (2011) 78-93.

\bibitem{HM} J. Hu and A. Mathas, {\em Graded cellular bases for the
cyclotomic Khovanov-Lauda-Rouquier algebras of type $A$}, Adv. Math.
{\bf 225} (2010) 598-642.

%\bibitem{HS} J. Hu and L. Shi, {\em Proof of the Center Conjectures for the cyclotomic Hecke and KLR algebras of type A}, arXiv: 2211.07069.

\bibitem{HLQ} W. Hu, F. Huang, Y. Li and X. Qi, {\em Moving vectors II: Lie thoery},
in preparation.

\bibitem{J} N. Jacon, {\em Kleshchev multipartitions and extended Young diagrams}, Adv. Math. {\bf 339} (2018) 367-403.

\bibitem{JL} N. Jacon and C. Lecouvey, {\em Cores of Ariki-Koike algebras}, Doc. Math. {\bf 26} (2021) 103-124.

\bibitem{J2} G. James, {\em Some combinatorial results involving Young diagrams}, Proc. Cambridge Philos. Soc. {\bf 83} (1978) 1-10.

\bibitem{Kac} V.~G. Kac, {\em Infinite dimensional Lie algebras}, CUP, Cambridge, third~ed., 1994.

\bibitem{KL} M. Khovanov and A. D. Lauda, {\em A diagrammatic approach to categorification of quantum groups.
I}, Represent. Theory {\bf 13} (2009) 309–347.

\bibitem{KX2} S. Koenig and C. C. Xi, {\em On the structure of cellular algebras}, In: I.Reiten, S.Smal${\o}$ and ${\O}$. Solberg (Eds.): Algebras
and Modules II. Canadian Mathematical Society Conference Proc. {\bf 24} (1998), 365-386.

\bibitem{KX5} S. Koenig and C.C. Xi, {\em Cellular algebras: Inflations and Morita
equivalences}, J. London Math. Soc. \textbf{60} (1999) 700-722.

\bibitem{KZ} S. Koenig and A. Zimmermann, Derived equivalences for group rings, Lect. Note Math. 1685 Springer 1998.

\bibitem{K} H. Krause, {\em Stable equivalence preserves representation type}, Comment. Math. Helv. {\bf 72} (1997) 266-284.

\bibitem{LQT} Y. Li, X. Qi and K. M. Tan, {\em Moving vectors and core blocks of Ariki-Koike algebras},
in preparation.

\bibitem{LM} S. Lyle and A. Mathas, {\em Blocks of cyclotomic Hecke
algebras}, Adv. Math. {\bf 216} (2007) 854-878.

\bibitem{LR} S. Lyle and O. Ruff, {\em Graded decomposition numbers of Ariki–Koike algebras for
blocks of small weight}, J. Pure Appl. Algebra {\bf 220} (2016) 2112–2142.

\bibitem{MM} G. Malle and A. Mathas, {\em Symmetric cyclotomic Hecke algebras},  J. Algebra {\bf 205} (1998)
275-293.

\bibitem{M} A. Mathas, {\em Iwahori-Hecke algebras and Schur algebras of the
symmetric group}, University Lecture Series 15, Amer. Math. Soc. (1999).

\bibitem{Mr} G. Murphy, {\em The representations of Hecke algebras of type $A_n$}, J.
Algebra {\bf 173} (1995) 97-121.

\bibitem{R} J. Rickard, {\em Derived categories and stable equivalence}, J. Pure. Appl. Alg. {\bf 61} (1989) 303-317.

\bibitem{Rou} R. Rouquier, {\em 2-Kac-Moody algebras}, Preprint (2008), arXiv: 0812.5023.

\bibitem{U} D. Uglov, Canonical bases of higher level q-deformed Fock spaces and Kazhdan-Lusztig polynomials, in Physial
Combinatorics (ed. M. Kashiwara, T. Miwa), Progress in Math. 191, Birkhauser (2000).

\bibitem{W} K. Wada, {\em The representation type of Ariki-Koike algebras and
cyclotomic q-Schur algebras}, Adv. Math. {\bf 224} (2010) 539-560.

\end{thebibliography}
\end{document}